\documentclass[a4paper]{amsart}
\usepackage[latin1]{inputenc}
\usepackage[T1]{fontenc}
\DeclareFontFamily{T1}{wncyr}{}
\DeclareFontShape{T1}{wncyr}{m}{n}{%
  <5><6><7><8><9>gen*wncyr%
  <10><10.95><12><14.4><17.28><20.74><24.88>wncyr10}{}
\usepackage{amssymb}	% Des symboles.
\usepackage{calc}	    	% Faire des calculs sur les longueurs.
\usepackage{multicol}		% Plusieurs colonnes.
\usepackage{wasysym}
\usepackage{verbatim}	% Pour les codes sources en informatique.
\usepackage{mathrsfs}	% Des lettres majuscules cursives (\mathscr).
\usepackage{dsfont}
\usepackage[dvipsnames]{xcolor}
\usepackage{pgf}
\usepackage{pgfpages}
\usepackage{tikz}
\usetikzlibrary{trees,shapes,%
  matrix,arrows,decorations.pathmorphing,backgrounds,positioning,fit}
\usepackage{indentfirst}
\usepackage{cancel}
\usepackage{fixltx2e} %%% pour avoir des \textsubscript
\usepackage{hyperref}
\theoremstyle{plain}

\newtheorem{thm}{Theorem}[section]
\newtheorem{lem}[thm]{Lemma}
\newtheorem{prop}[thm]{Proposition}
\newtheorem{coro}[thm]{Corollary}
\newtheorem*{thma}{Theorem}

\theoremstyle{definition}

\newtheorem{defn}[thm]{Definition}
\newtheorem{exm}[thm]{Example}

\theoremstyle{remark}
\newtheorem{rem}[thm]{Remark}

\newcommand{\on}{\operatorname}

\newcommand{\mc}{\mathcal}
\newcommand{\mb}{\mathbf}

\newcommand{\mf}{\mathfrak}

\newcommand{\id}{\ensuremath{\mathop{\rm id\,}\nolimits}}

\newcommand{\Z}{\mathbb{Z}}

\newcommand{\C}{\mathbb{C}}
\newcommand{\Q}{\mathbb{Q}}
\newcommand{\K}{\mathbb{K}}

\newcommand{\A}{\mathbb{A}}

\newcommand{\Sn}[1][n]{\mf S_{#1}}
\newcommand{\vp}{\varphi}
\newcommand{\ve}{\varepsilon}

\newcommand{\ol}{\overline}
\newcommand{\ul}{\underline}

\newcommand{\st}{\stackrel}
\newcommand{\mx}{\mbox}
\newcommand{\geqs}{\geqslant}
\newcommand{\leqs}{\leqslant}

\newcommand{\w}{\wedge}

\newcommand{\wt}{\widetilde}

\newcommand{\sm}{\setminus}

\newcommand{\ds}{\displaystyle}

\newcommand{\ra}{\rightarrow}
\newcommand{\lra}{\longrightarrow}

\newcommand{\im}{\operatorname{im}}

\newcommand{\Gr}{\on{Gr}}
\newcommand{\HH}{\on{H}}

\newcommand{\CH}{\on{CH}}
\newcommand{\p}{\mathbb{P}}

\newcommand{\Sp}{\on{Spec}}
\newcommand{\Spec}{\Sp}

\newcommand{\Supp}{\on{Supp}}
\newcommand{\dme}[1][]{\mc{DM}_{gm}^{eff}}
\newcommand{\dm}[1][]{\mc{DM}_{gm}}

\newcommand{\DM}{\on{DM}}

\newcommand{\MTM}{\on{MTM}}

\newcommand{\SmCorr}[1][]{\on{\SmCorr}_{#1}}

\newcommand{\pst}[1]{\p^1 #1 \setminus \{0,1,\infty\} }
\newcommand{\ps}{\pst{}}

\newcommand{\we}{wt}

\newcommand{\Nc}[1][X]{{\mc N}^{\bullet}_{#1}}
\newcommand{\cN}[1][1]{{\mc N}^{#1}_{X}}
\newcommand{\cNg}[2]{{\mc N}^{#2}_{#1}}

\newcommand{\Nge}[2]{{\mc N}^{eq, \,#2}_{#1}}
\newcommand{\Ngqf}[2]{{\mc N}^{q.f. \,#2}_{#1}}

\newcommand{\Mm}[1][X]{\cLie_{#1}}
\newcommand{\Alt}{{\mc A}lt}
\newcommand{\Lcal}{\mc L}
\newcommand{\Lcz}{\Lcal_{0}^1}
\newcommand{\Lco}{\Lcal_{1}^0}
\newcommand{\Lc}{\Lcal^{0}}
\newcommand{\Lcu}{\Lcal^{1}}

\newcommand{\Lcgqf}[2]{\mc L_{qf,\, #1}^{#2}}
\newcommand{\Lczqf}{\Lcgqf{0}{}}
\newcommand{\Lcoqf}{\Lcgqf{1}{}}
\newcommand{\Lcqf}[1]{\Lcgqf{#1}{0}}
\newcommand{\Lcuqf}[1]{\Lcgqf{#1}{1}}

\newcommand{\AL}{R_{\Lc}}
\newcommand{\ALu}{R_{\Lcu}}
\newcommand{\RAL}{R_{\Ac{}}}
\newcommand{\RALu}{R_{\Acu{}}}

\newcommand{\Tc}[2][0]{T_{#2^*}^{#1}}%{{\mf T}_{#1}}
\newcommand{\Tcu}[1]{\Tc[1]{#1}}%{\Tc{#1}^{1}}
\newcommand{\Tcg}{T}
\newcommand{\Ac}[2][0]{A_{#2}^{#1}}%{{\mf T}_{#1}}
\newcommand{\Acu}[1]{\Ac[1]{#1}}%{\Tc{#1}^{1}}
\newcommand{\Acz}{\Acu 0}
\newcommand{\Aco}{\Ac 1}

\newcommand{\Zc}{\mc Z}
\newcommand{\Ze}[1][\bullet]{\mc Z_{eq}^{#1}}
\newcommand{\Zqf}[1][\bullet]{\mc Z_{q.f.}^{#1}}

\newcommand{\ap}{{a'}}
\newcommand{\bp}{{b'}}

\newcommand{\dN}{\partial}
\newcommand{\da}[1][]{\partial_{#1}}
\newcommand{\dL}{d_{Lie}}
\newcommand{\dc}{d_{cy}}
\newcommand{\dggl}{d_{ggl}}

\newcommand{\cLie}{\textrm{co}\mf{L}}
\newcommand{\Lie}{\on{Lie}}
\newcommand{\Ttr}{\mc T^{tri}}
\newcommand{\Ttrg}[1][x]{\mc T^{tri}_{#1}}
\newcommand{\Ttq}{\Q[\mc T^{tri}]}
\newcommand{\Ttg}[1][x]{\Q[\mc T^{tri}_{#1}]}
\newcommand{\Tts}[1][]{\Q[\mc T^{tri,<}_{#1}]}
\newcommand{\TL}[1][]{\mc T^{Lie}_{#1}}
\newcommand{\Tcl}[1][]{\mc T^{coL}_{#1}}

\newcommand{\Bs}[1][]{\mf B^{<}_{#1}}

\newcommand{\T}[1]{T_{#1^*}}

\newcommand{\tw}[1]{\tau_{#1}}
\newlength{\ledge}
\newlength{\sibdis}
\newlength{\ecan}
\newlength{\ecas}
\newlength{\ecae}
\newlength{\ecao}
\newlength{\eca}
\newlength{\lbullet}
\tikzset{deftree/.style={level distance=\ledge,sibling distance=\sibdis}}
\tikzset{edgesp/.style={level distance=#1*\ledge,sibling distance=#1\sibdis}}
\tikzset{labf/.style={mathsc,yshift=-\eca}}
\tikzset{labfs/.style={mathss,yshift=-\eca}}
\tikzset{labr/.style={mathsc,yshift=\eca}}
\tikzset{labrs/.style={mathss,yshift=\eca}}
\tikzset{math mode/.style = {execute at begin node=$, execute at end node=$}}
\tikzset{mathscript mode/.style =%
 {execute at begin node=$\scriptstyle , execute at end node=$}}
\tikzset{math/.style = {execute at begin node=$, execute at end node=$}}
\tikzset{mathsc/.style =%
 {execute at begin node=$\scriptstyle , execute at end node=$}}
\tikzset{mathss/.style =%
 {execute at begin node=$\scriptscriptstyle , execute at end node=$}}
\tikzset{root/.style={draw,circle,inner sep=1pt,execute at begin node=$\bullet,
    execute at end node=$}}
\tikzset{roottest/.style={draw,circle,inner sep=#1pt}}
\tikzset{roots/.style={draw,circle,inner sep=2pt}}
\tikzset{bull/.style={fill,circle,minimum size=2pt,inner sep=0pt}}
\tikzset{leaf/.style={minimum size=2pt,inner sep=1pt}}
\tikzset{leafb/.style={minimum size=0pt,inner sep=0pt}}
\tikzset{intvertex/.style={mathsc,fill,circle,minimum size=0.6ex, inner sep=0pt}}
\tikzset{Reda/.style={-,double distance=0.3ex, draw=black}}
\tikzset{N/.style={-,thin,draw=black}}
\tikzset{Nd/.style={-,dotted, thin,draw=black}}
\newcommand{\prodracine}{%
\begin{tikzpicture}[%
baseline={([yshift=-0.5ex]current bounding box.center)},scale=0.4]
\tikzstyle{every child node}=[mathscript mode,minimum size=0pt, inner sep=0pt]
\node[minimum size=0pt, inner sep=0pt]%[draw,circle,inner sep=0.5pt](root) 
{}
[level distance=1.5em,sibling distance=3ex]
child {node[fill, circle, minimum size=2pt,inner sep=0pt]{}[level distance=1.5em]
  child{ node{}}
  child{ node{}}
};\end{tikzpicture}
}
\DeclareMathOperator{\prac}{\prodracine}

\newcommand{\rootaa}[3][1]{
\begin{tikzpicture}[baseline=(current bounding box.center),scale=#1]
\tikzstyle{every child node}=[mathscript mode,minimum size=0pt, inner sep=0pt]
\node[roots](root) {}
[deftree]
child {node[fill, circle, minimum size=2pt,inner sep=0pt]
           {}[level distance=1.5em]
  child{ node[leaf]{#2}}
  child{ node[leaf]{#3}}
};
\fill (root.center) circle (1/#1*\lbullet) ;
\end{tikzpicture}
}
\newcommand{\rootaar}[4][1]{
\begin{tikzpicture}[baseline=(current bounding box.center),scale=#1]
\tikzstyle{every child node}=[mathscript mode,minimum size=0pt, inner sep=0pt]
\node[roots](root) {}
[deftree]
child {node[fill, circle, minimum size=2pt,inner sep=0pt]
           {}[level distance=1.5em]
  child{ node[leaf]{#2}}
  child{ node[leaf]{#3}}
};
\fill (root.center) circle (1/#1*\lbullet) ;
\node[mathsc, xshift=-1ex] at (root.west) {#4};
\end{tikzpicture}
}

\newcommand{\roota}[2][1]{
\begin{tikzpicture}[baseline=(current bounding box.center),scale=#1]
\tikzstyle{every child node}=[mathscript mode,minimum size=0pt, inner sep=0pt]%[intvertex]%[fill,circle,minimum size=3pt, inner sep=0pt]
\node[roots](root) {}
[deftree]
child {node[leaf](1){#2} 
};
\fill (root.center) circle (1/#1*\lbullet) ;
\end{tikzpicture}
}
\newcommand{\rootar}[3][1]{
\begin{tikzpicture}[baseline=(current bounding box.center),scale=#1]
\tikzstyle{every child node}=[mathscript mode,minimum size=0pt, inner sep=0pt]%[intvertex]%[fill,circle,minimum size=3pt, inner sep=0pt]
\node[roots](root) {}
[deftree]
child {node[leaf](1){#2} 
};
\fill (root.center) circle (1/#1*\lbullet) ;
\node[mathsc, xshift=-1ex] at (root.west) {#3};
\end{tikzpicture}
}

\newcommand{\TLx}[2][0.6]{
\begin{tikzpicture}[baseline=(current bounding box.center),scale=#1]
\tikzstyle{every child node}=[intvertex]
\node[roots](root) {}
[deftree]
child {node(1){} 
}
;%%%
\fill (root.center) circle (1/#1*\lbullet) ;
\node[mathsc, xshift=-1ex] at (root.west) {#2};
\node[labf] at (1.south){0};
\end{tikzpicture} %
}
\newcommand{\TLy}[2][0.6]{
\begin{tikzpicture}[baseline=(current bounding box.center),scale=#1]
\tikzstyle{every child node}=[intvertex]
\node[roots](root) {}
[deftree]
child {node(1){} 
}
;%%%
\fill (root.center) circle (1/#1*\lbullet) ;
\node[mathsc, xshift=-1ex] at (root.west) {#2};
\node[labf] at (1.south){1};
\end{tikzpicture} %
}

\newcommand{\TLxy}[2][0.6]{
\begin{tikzpicture}[baseline=(current bounding box.center),scale=#1]
\tikzstyle{every child node}=[intvertex]
\node[roots](root) {}
[deftree]
child {node{} 
    child{ node(1){}}
    child{ node(2){}} 
}
;%%%
\fill (root.center) circle (1/#1*\lbullet) ;
\node[mathsc, xshift=-1ex] at (root.west) {#2};
\node[labf] at (1.south){0};
\node[labf] at (2.south){1};
\end{tikzpicture} 
}
\newcommand{\TLxxy}[2][0.6]{
  \begin{tikzpicture}[baseline=(current bounding box.center),scale=#1]
 \tikzstyle{every child node}=[intvertex]
 \node[roots](root) {}
 [deftree]
 child {node{}
   child[edgesp=2] {node(1){}}
   child {node{} 
     child{ node(2){}}  
     child{ node(3){}} 
   }
}
;
\fill (root.center) circle (1/#1*\lbullet) ;
\node[mathsc, xshift=-1ex] at (root.west) {#2};
\node[labf] at (1.south){0};
\node[labf] at (2.south){0};
\node[labf] at (3.south){1};
 \end{tikzpicture}
}
\newcommand{\TLxyy}[2][0.6]{
\begin{tikzpicture}[baseline=(current bounding box.center),scale=#1]
\tikzstyle{every child node}=[intvertex]
\node[roots](root) {}
[deftree]
child {node{} 
  child{node{}
    child{ node(1){}}
    child{ node(2){}}}  
  child[edgesp=2]{ node(3){}} 
}
;%%%
\fill (root.center) circle (1/#1*\lbullet) ;
\node[mathsc, xshift=-1ex] at (root.west) {#2};
\node[labf] at (1.south){0};
\node[labf] at (2.south){1};
\node[labf] at (3.south){1};
\end{tikzpicture}
}
\newcommand{\TLxxxy}[2][0.6]{
 \begin{tikzpicture}[baseline=(current bounding box.center),scale=#1]
\tikzstyle{every child node}=[intvertex]
\node[roots](root) {}
[deftree]
child{node{}
  child [edgesp=3]{node(1){}}
  child{node{}
    child [edgesp=2]{node(2){}}
    child {node{} 
      child{ node(3){}}  
      child{ node(4){}} 
    }
  }
};

\fill (root.center) circle (1/#1*\lbullet) ;
\node[mathsc, xshift=-1ex] at (root.west) {#2};
\node[labf] at (1.south){0};
\node[labf] at (2.south){0};
\node[labf] at (3.south){0};
\node[labf] at (4.south){1};
\end{tikzpicture}
}
\newcommand{\TLxxyy}[2][0.6]{
\begin{tikzpicture}[baseline=(current bounding box.center),scale=#1]
\tikzstyle{every child node}=[intvertex]
\node[roots](root) {}
[deftree]
child{node{}
  child[edgesp=3]{node(1){}}
  child {node{} 
    child{node{}
      child{ node(2){}}
      child{ node(3){}}
    }  
    child[edgesp=2]{ node(4){}} 
  }
};
\fill (root.center) circle (1/#1*\lbullet) ;
\node[mathsc, xshift=-1ex] at (root.west) {#2};
\node[labf] at (1.south){0};
\node[labf] at (2.south){0};
\node[labf] at (3.south){1};
\node[labf] at (4.south){1};
\end{tikzpicture} 
+
\begin{tikzpicture}[baseline=(current bounding box.center),scale=#1]
\tikzstyle{every child node}=[intvertex]
\node[roots](root) {}
[deftree]
child {node{} 
  child{node{}
    child[edgesp=2]{node(1){}}
    child{node{}
      child{ node(2){}}
      child{ node(3){}}
    }  
  }
  child[edgesp=3]{ node (4){}} 
};
\fill (root.center) circle (1/#1*\lbullet) ;
\node[mathsc, xshift=-1ex] at (root.west) {#2};
\node[labf] at (1.south){0};
\node[labf] at (2.south){0};
\node[labf] at (3.south){1};
\node[labf] at (4.south){1};
\end{tikzpicture} 
}
\newcommand{\TLxxyya}[2][0.6]{
\begin{tikzpicture}[baseline=(current bounding box.center),scale=#1]
\tikzstyle{every child node}=[intvertex]
\node[roots](root) {}
[deftree]
child{node{}
  child[edgesp=3]{node(1){}}
  child {node{} 
    child{node{}
      child{ node(2){}}
      child{ node(3){}}
    }  
    child[edgesp=2]{ node(4){}} 
  }
};
\fill (root.center) circle (1/#1*\lbullet) ;
\node[mathsc, xshift=-1ex] at (root.west) {#2};
\node[labf] at (1.south){0};
\node[labf] at (2.south){0};
\node[labf] at (3.south){1};
\node[labf] at (4.south){1};
\end{tikzpicture} 
}
\newcommand{\TLxyyy}[2][0.6]{
\begin{tikzpicture}[baseline=(current bounding box.center),scale=#1]
\tikzstyle{every child node}=[intvertex]
\node[roots](root) {}
[deftree]
child {node{} 
  child{node{}
    child{node{}
      child{node (1){}}
      child{ node (2){}}
    }
    child[edgesp=2]{ node (3){}}
  }  
  child[edgesp=3]{ node (4){}} 
};
\fill (root.center) circle (1/#1*\lbullet) ;
\node[mathsc, xshift=-1ex] at (root.west) {#2};
\node[labf] at (1.south){0};
\node[labf] at (2.south){1};
\node[labf] at (3.south){1};
\node[labf] at (4.south){1};
\end{tikzpicture}
}

\title[The cycle complex over $\p^1$ minus $3$ points]
{The cycle complex over $\p^1$ minus $3$ points 
: toward multiple zeta value cycles.}
\author{Ismael Soud{\`e}res}
\address{Fachbereich Mathematik \\
Universität Duisburg-Essen, Campus Essen \\
Universitätsstrasse 2\\
45117 Essen\\
Germany \\
ismael.souderes@uni-due.de}
\date{\today}
\thanks{This work was partially supported by DFG grant SFB/TR45 and by
  Prof. Levine's Humboldt Professorship. 
I would like to thank P. Cartier for his
  attention to my work, and H. Gangl for all his help, his comments and his
 insistent demand for me to use tree-related structures, which allowed me to
 understand the Lie-like underlying combinatorics. Finally, none of this would
 have been possible without M. Levine's patience and his explanations.}
\begin{document}
%\begin{multicols}{2}
\begin{abstract}

In this paper, we construct a family of algebraic cycles in Bloch's
cycle complex over $\p^1$ minus three points, which are expected to
correspond to multiple polylogarithms in one variable. Elements in this family
of weight $p$ belong to the cubical cycle group of codimension $p$ in $(\p^{1}\sm
\{0,1,\infty\}) \times (\p^1\sm\{1\})^{2p-1}$ and in weight greater 
than or equal to $2$, 
 they naturally extend as equidimensional cycles over $\A^1$. 

Thus, we can consider their fibers at the point $1$. This is one of the 
main differences with the work of Gangl, Goncharov and Levin. Considering the 
fiber of our cycles at $1$ makes it possible to view these cycles as those
corresponding to weight $n$ multiple zeta values which are viewed here as the
values at $1$ of multiple polylogarithms.

After the introduction, we recall some properties of Bloch's cycle
complex, and explain
the difficulties on a few examples. Then a large section is devoted to the
combinatorial situation, essentially involving the combinatorics of trivalent
trees in relation with the structure of the free Lie algebra on two
generators. In the last section, two families of cycles are constructed as
solutions to a ``differential system'' in Bloch's cycle complex. One of these
families contains only cycles with empty fiber at $0$; these correspond to
multiple polylogarithms.  
 \end{abstract}
\maketitle

\tableofcontents
\section{Introduction}
\subsection*{General goals}

%{\color{blue}

 This paper  is a first and crucial step \emph{toward} motivic multiple zeta
 values via algebraic cycles. 

In the algebraic cycles setting, motives
 arise as comodules over the Tannakian Lie coalgebra given by the $\HH^0$ of
 the bar construction over a differential graded  algebra $\cNg{}{\bullet}$, modulo
 products. The algebra $\cNg{}{\bullet}$ is built out of algebraic cycles.
 The Lie coalgebra generating its $1$-minimal model is isomorphic to the above
 Tannakian Lie coalgebra. 

In order to obtain a motive, that is a comodule, it is enough to have an element
in the relevant Tannakian Lie  coalgebra. Then the motive is the comodule
cogenerated by 
this element. Since the Tannakian 
 Lie coalgebra is a ``$\HH^0$ modulo product'', we write one of its elements as
 a class $[\Lcal^B]$ in this $\HH^0$. This class can be represented in the bar
 construction over 
 $\cNg{}{\bullet}$ by an element $\Lcal^B$. The element  $\Lcal^B$ is
 essentially determined by its projection onto
 its tensor degree $1$ part, which gives an algebraic cycle $\Lcal$ in
 $\cNg{}{\bullet}$. The cycle $\Lcal$ has a 
 decomposable  boundary  because the
 element $\Lcal^B$ leads to a class in the $\HH^0$. 

Hence a first step toward explicit motives via algebraic cycles is to build
algebraic cycles in   $\cNg{}{\bullet}$ having a decomposable boundary. The
main result of this paper (Theorem \ref{thm:cycleLcLcu}) provides such algebraic
cycles denoted by $\Lc_W$; here the indexing set $W$ consists in Lyndon
words. It is shown in \cite{SouBarbase} that the motives attached to these cycles
generate the Lie coalgebra associated to the Deligne-Goncharov fundamental group.

However  lifting algebraic cycles to objects in the bar construction requires, in
general,  a strong combinatorial or algebraic control of the boundaries
involved. In our context, this control is insured (Theorem 
\ref{thm:cycleLcLcu}) by the algebraic and combinatorial
structure of Ihara's special derivations studied at Proposition \ref{dcyTw}.

The next steps in order to obtain motivic multiple zeta values are:% the
%following:
\begin{enumerate}
\item %we lift cycles $\Lc_W$ to elements in $B(\cNg{X}{\bullet})$, the bar
 % construction over $\cNg{X}{\bullet}$. This goes as follows
%
%Writing $\Tcl[1;x]$ the Lie coalgebra
%  associated to Ihara's special derivation encoding our combinatorial
%  setting, 
Our main theorem  makes it possible to build a commutative
  differential graded algebra  morphism between the cobar construction over
  the Lie coalgebra associated to Ihara's special derivations (combinatorial
  structure) and $\cNg{}{\bullet}$ (algebraic cycles). Then the unit of the bar/cobar adjunction
  lifts  cycles $\Lc_W$, built in this paper, to elements of $B(\cNg{}{\bullet})$, the bar
  construction over $\cNg{}{\bullet}$.
%\item Because of geometric constraints, the above theorem proves the existence
 % of  two family of
 % cycles : the family of cycle $\Lc_W$ with an empty fiber at $0$ and the family
 % of cycle $\Lcu_W$ with an empty fiber at $1$. We show that at the motivic
 % level (i.e. in the Lie coalgebra) the bar classes coming from cycles
 % $\Lcu_W$ reduce to the one given by 
 % cycles $\Lc_W$ up to constant cycles. This uses an induction on the weight
 % $p$, the differential controlled by Ihara's cobracket  and the interplay
 % between cycles over $\A^1$ and cycles over $\ps$. 
\item Then %the combinatorial 
Proposition \ref{dcyTw} shows that the motivic cobracket of  elements induced by cycles $\Lc_W$ is exactly given by Ihara's cobracket. As a consequence, 
  the family of motivic elements (i.e. in the Lie coalgebra) arising from cycles
  $\Lc_W$ generates the Lie 
  coalgebra associated to the Deligne-Goncharov motivic fundamental group. This uses
  the Lie coalgebra version of the comparison between the Tannakian group of
  the category of mixed Tate motives and Deligne-Goncharov motivic fundamental group
  (cf. \cite{LEVTMFG}). 
\end{enumerate}
The above steps can be deduced from this paper using the bar/cobar adjunction and
classical motivic theory. In particular they do not require any other specific
algebraic or geometric structure than 
the ones developed in Proposition \ref{dcyTw} and Theorem \ref{thm:cycleLcLcu}.  

A more recent work, \cite{SouBarbase}, is devoted to proving the above
assertions. In particular, the cycles $\Lc_W$,
built in this paper, induce motivic elements which generate the Lie coalgebra
associated to the Deligne-Goncharov motivic fundamental group (see   \cite{SouBarbase}).

The Lie coalgebra associated to the Deligne-Goncharov motivic fundamental group
contains the multiple polylogarithms. Looking for multiple zeta values we want
the above construction to be compatible with specialization at ``$1$''. This is 
possible by our main result (Theorem \ref{thm:cycleLcLcu}) which gives
cycles that are not only fiberered  over $\ps$ but are also fibered over
$\A^1$. In particular cycles built at Theorem \ref{thm:cycleLcLcu} can be
specialized at $1$. This is a major improvement on previous attempt to build
algebraic cycles attached to \emph{multiple} polylogarithms.

The rest of this introduction gives, first, the general background about multiple
polylogarithms and algebraic cycles. Then it presents the general strategy and
states our mains results : Theorem \ref{thm:cycleLcLcu} which builds algebraic
cycles $L_W$ under the algebraic control given by Proposition \ref{dcyTw} which
concerns the Lie coalgebra associated to Ihara's special derivations.
%}
\subsection{Multiple polylogarithms}
Multiple polylogarithm functions are defined (cf. \cite{PAGGon}) by the power
series
\[
Li_{k_1, \ldots, k_m}(z_1,\ldots, z_m)=\sum_{n_1> \cdots > n_m>0} 
\frac{z_1^{n_1}}{n_1^{k_1}} \, \frac{z_2^{n_2}}{n_2^{k_2}} \cdots
\frac{z_m^{n_m}}{n_m^{k_m}}
\qquad (z_i \in \C, \, |z_i|<1),
\]
where the $k_i$'s are strictly positive integers.
They admit an analytic continuation to a Zariski open subset of $\C^m$. The case
$m=1$ is nothing but the classical \emph{polylogarithm} functions. The case
$z_1=z$ and $z_2=\cdots=z_m=1$ gives a one-variable version of the multiple
polylogarithm function 
\[
Li^{\C}_{k_1, \ldots , k_m}(z)=Li_{k_1, \ldots, k_m}(z,1,\ldots, 1)=\sum_{n_1> \cdots > n_m>0} 
\frac{z^{n_1}}{n_1^{k_1} n_2^{k_2} \cdots n_m^{k_m}}.
\] 
When $k_1$ is greater or equal to $2$, the series converges as $z$ 
tends to $1$, where we recover the multiple zeta value
\[
\zeta(k_1,\ldots, k_m) =Li^{\C}_{k_1, \ldots , k_m}(1) = 
Li_{k_1, \ldots, k_m}(1,\ldots, 1)=\sum_{n_1> \cdots > n_m>0} 
\frac{1}{n_1^{k_1} n_2^{k_2} \cdots n_m^{k_m}}.
\]

To the $m$-tuple of positive integers $(k_1, \ldots, k_m)$ of weight $n=\sum k_i$, 
we associate an $n$-tuple of $0$'s and $1$'s
\[
(\ve_n,\ldots,  \ve_1):=(\underbrace{0,\ldots , 0}_{k_1-1\mx{ \scriptsize  times}}, 1,\ldots
  ,\underbrace{0,\ldots ,0}_{k_m-1\mx{ \scriptsize  times}},
  1).
\]
This allows us to write  multiple polylogarithms as iterated integrals
($z_i\neq 0$ for all $i$)
\[
Li_{k_1,\ldots, k_m}(z_1, \ldots,z_m)=(-1)^m \int_{\Delta_{\gamma}} \frac{d
  t_1}{t_1-\ve_1 x_1} \w \cdots \w 
\frac{d t_n}{t_n-\ve_n x_n},
\]
where $\gamma$ is a path from $0$ to $1$ in $\C \sm \{x_1, \ldots, x_n\}$. The
integration domain $\Delta_{\gamma}$ is the associated real simplex consisting
of all $m$-tuples of points $(\gamma(t_1), \ldots, \gamma(t_n))$ with $t_i<t_j$
for $i<j$,  where  we have set %
\begin{multline*}
(x_n,\ldots, x_1):= \\
(\underbrace{z_1^{-1},\ldots , z_1^{-1}}_{k_1\mx{ \scriptsize  times}}
,\underbrace{(z_1z_2)^{-1},\ldots ,(z_1 z_2)^{-1}}_{k_2\mx{ \scriptsize  times}}
,\ldots
,\underbrace{(z_1\cdots z_m)^{-1},\ldots ,(z_1\cdots z_m)^{-1}}_{k_m\mx{ \scriptsize  times}} ).
\end{multline*}

As shown in \cite{GSFGGon}, iterated integrals have Hodge/motivic avatars living
in a Hopf algebra equivalent to the Tannakian Hopf algebra of mixed
$\Q$-Hodge-Tate structures. Working with these motivic/Hodge iterated integrals 
reveals more structure -- in particular the coproduct, which is not visible on the
level of numbers -- conjecturally without losing any information. 

\subsection{Multiple polylogarithms and algebraic cycles}
The relations between the motivic world and the higher Chow groups on
the one hand (e.g. \cite{KTMMLevine,MCHCGVoe}) and the relations
between multiple polylogarithms and regulators (e.g. \cite{PDZKthZag, PRACGon}) on the other, 
suggest the question of whether there exist avatars of the
multiple polylogarithms in terms of algebraic cycles.

Given a number field $\K$, in \cite{BKMTM}, Bloch and Kriz used
algebraic cycles to construct a graded Hopf algebra. They conjectured that this
Hopf algebra was isomorphic to the Tannakian
Hopf algebra 
of the category of mixed Tate motives over $\K$, a result that was later proved by
Spitzweck in 
\cite{SpitzweckSCVTM} (as presented in \cite{KTMMLevine}). Moreover, Bloch and
Kriz described a direct Hodge realization functor for these ``cyclic
motives''. For any integer $n$ greater than or equal to $2$ and any point $z$ in
$\K$, they produced an algebraic cycle $\on{Li}_n^{cy}(z)$. This cycle
$\on{Li}_n^{cy}(z)$ induces a motive. They showed in \cite[Theorem 9.1]{BKMTM} that the 
``bottom-left'' coefficient of its period matrix in the Hodge realization is
exactly $-Li_n(z)/(2i\pi)^n$.

More recently in \cite{GanGonLev05}, Gangl, Goncharov and Levin, using a
combinatorial approach,  
constructed algebraic cycles corresponding to the multiple polylogarithm values
$Li_{k_1,\ldots, k_m}(z_1, \ldots, z_m)$ with parameters $z_i$ in $\K^*$, under
the condition that the corresponding $x_i$ (as defined above) are all
distinct. In particular, all the $z_i$ except for $z_1$ must be different from $1$.
This implies that their method does not give algebraic cycles corresponding to
multiple zeta values.
   
\subsection{Algebraic cycles over $\ps$}
The goal of this paper is to develop a geometric construction for algebraic
cycles which removes the previous obstacle ($z_i\neq 1$); thus opening the
possibility to obtain algebraic cycle attached to multiple zeta values (as
explained at the very beginning of this introduction).

The general approach of  this project views cycles as fibered over a larger
base, and not just point-wise cycles for some fixed parameter $(z_1,
...,z_m)$.  Levine, in \cite{LEVTMFG}, shows that there exists a short
exact sequence relating the Bloch-Kriz Hopf algebra over $\Sp(\K)$, its
relative version over $\ps$ and the Hopf algebra associated to Goncharov and
Deligne's motivic fundamental group over $\ps$ which contains the motivic
iterated integrals associated to multiple polylogarithms in one variable. 

This one-variable version of multiple polylogarithms gives multiple zeta
values for $z=1$. Thus, in order to obtain motives  corresponding to multiple
zeta values in this framework of algebraic cycles, it is natural to 
first investigate the Bloch-Kriz construction over $\ps$. This should lead to
algebraic cycles and motives  corresponding to multiple polylogarithms in one
variable. The multiple zeta value objects then arise as limit motives, or
as limits of variations of mixed Hodge structure as $z$ tends to $1$.

 However, before computing any motives or any matrix periods, we first need to
 obtain explicit algebraic cycles 
over $\ps$ such that: \begin{itemize}
\item their boundary (in the Bloch-Kriz complex) is related to the differential (or derivative) of
  multiple polylogarithm functions;
\item their Zariski closures over $\A^1$ have a well defined fiber at $z=1$
as element of the Bloch-Kriz complex. 
\end{itemize}
This last condition is a priori guided by the fact that, even if the period
matrix corresponding to $Li_n(z)$ is not well defined at $z=1$ (because $\log(1-z)$
appears in some coefficients), its ``bottom-left coefficient''
$-Li_n(z)/(2i\pi)^n$ is well-defined at $z=1$.  Moreover it is also naturally 
imposed by our geometric construction (cf. Section \ref{treetocycle}).

This paper constructs a family of cycles satisfying these conditions. In the 
final remarks, we provide some extra evidence that it is a good family by computing
an integral in low weight.

\subsection{Strategy and Main results}
The Bloch-Kriz Hopf algebra and its relative version over $\ps$ is the $\HH^0$
of the bar construction over a commutative differential graded algebra (cdga)
$\cNg{X}{\bullet}$  constructed from 
algebraic cycles. We will use this algebra in the case $\K=\Q$ and $X=\Sp(\Q)$
or $X=\ps$ or  $X=\A^1$. The algebra $\cNg{X}{\bullet}$  comes from the cubical
construction of the 
higher Chow groups, and setting $\square^1=\p^1 \sm\{1\} \simeq \A^1$, we have: 
\[
\cNg{X}{\bullet}=\Q \oplus \left(\oplus_{p \geqs 1}\mc N_X^{\bullet}(p)\right),
\]
where the $ \cNg{X}{n}(p)$ are generated by codimension $p$ cycles in
$X\times \square^{2p-n}$ which are admissible (cf. Definition \ref{defZc} and
Remark \ref{remZc}). The cohomology of the complex
$\cNg{X}{\bullet}(p)$ recovers the higher Chow groups
$\on{CH}^p(X,2p-\bullet)$.

As the $\HH^0$ of the bar construction over $\cNg{\ps}{\bullet}$ is isomorphic
to one of its 
$1$-minimal models \cite{BKMTM}, our strategy  is 
to follow the inductive construction of this $1$-minimal model as presented in
\cite{DGMS}. We recall below how the $1$-minimal model is built in \cite{DGMS}
because it guided and motivated the present work.

Let $S^{gr}(V)=\oplus_n S^{gr,n}(V)$ denote the graded symmetric algebra over a
graded vector space 
$V$. Roughly speaking, the inductive construction of the $1$-minimal model of
$\cNg{\ps}{\bullet}$  proceeds as follows. We start with 
 $V_1=\HH^1(\cNg{\ps}{\bullet})$ (in degree $1$) and a map $\vp_1$ given by a choice of
 representative in $\cNg{\ps}{\bullet}$ of a basis of $V_1$. This map induces a
 map $S^{gr}(V_1) \ra \cNg{\ps}{\bullet}$ and a map 
\[
\HH^k(S^{gr}(V_1)) \st{\tilde{\vp_1}}{\lra} \HH^k(\cNg{\ps}{\bullet}),
\]
where the differential on $S^{gr}(V_1)$ is $0$. But the above map may not be
injective on $\HH^2$, which is one of the desired properties of a $1$-minimal
model. Hence the first inductive step consists in killing the kernel of
$\tilde{\vp_1}$ on $\HH^2$. In general, at the $i$-th step of the induction,  we define 
\[
V_{i+1}=V_i\oplus\ker(\tilde{\vp_i}|_{\HH^2}),
\]
where the kernel is added in degree $1$ and where $\tilde{\vp_i}|_{\HH^2}$
denotes the restriction of $\tilde{\vp_i}$ to $\HH^2(S^{gr}(V_1))$. Then
 $\vp_i$ is extended to a map 
\[
\vp_{i+1} : V_{i+1}\lra \cNg{\ps}{1}
\] by defining it on $\ker(\tilde{\vp_{i}}|_{\HH^2})$. We
 choose a family $(b_k^{(i)})$ of degree $2$ elements in $S^{gr}(V_i)$ inducing a
 basis $([b_k^{(i)}])$ of $\ker(\tilde{\vp_i}|_{\HH^2})$. The image of $[b_k^{(i)}]$ under
 $\vp_{i+1}$ is defined as follows: by definition $b_k^{(i)}$ is a $\Q$-linear combination
 of products of elements in $V_i$,
\[ 
b_k^{(i)}=\sum \alpha_{a,b}^k c_a^{(i)} \cdot c_b^{(i)}, 
\]
The element $b_k^{(i)}$ is mapped by $\vp_i$ to  $\sum \alpha_{a,b}^k \vp_i(c_a^{(i)}) \cdot
\vp_i(c_b^{(i)}) $ in $\cNg{\ps}{2}$. As $b_k^{(i)}$ gives a class in the $\HH^2(S^{gr}(V_1))$, the
differential in $\cNg{\ps}{\bullet}$ of the previous sum is $0$. Because $[b_k^{(i)}]$ lies in the kernel of
$\tilde{\vp_i}$, the sum 
\[
\vp_i(b_k^{(i)})=\sum \alpha_{a,b}^k \vp_i(c_a^{(i)}) \cdot \vp_i(c_b^{(i)})
\] 
is a boundary in $\cNg{\ps}{\bullet}$. That is, there exists $e_k^{(i)}$ in
$\cNg{\ps}{1}$ such that  
\[
\vp_i(b_k^{(i)})=\dN(e_k^{(i)})
\] 
where $\dN$ denotes the differential in $\cNg{\ps}{\bullet}$. The element
$\vp_{i+1}([b_k])$ is then defined by 
\[
\vp_{i+1}([b_k])=e_k^{(i)}.
\]
%The differential of $[b_k]$ in $S^{gr}(V_{i+1})$ is defined to be 
%\[
%\sum \alpha_{a,b}^k c_a \cdot c_b \quad \in V_i\cdot V_i.
%\]
The discussion above can be summarized in the following diagram
\[
\begin{tikzpicture}
\matrix (m) [matrix of math nodes,
 row sep=2em, column sep=4em, 
 text height=2.5ex,  text depth=0.25ex] 
{S^{gr,2}(V_i)  & \cNg{\ps}{2} & \\
b_k^{(i)}   & \sum \alpha_{a,b}^k\vp_i(c_a^{(i)}) \cdot \vp_i(c_b^{(i)})  & 0 \\
& \exists e_k^{(i)} \in \cNg{\ps}{1} & \\
};
 \path[->,font=\scriptsize]
(m-1-1) edge node[auto] {$\vp_i$} (m-1-2);
\path[|->,font=\scriptsize]
(m-2-1) edge node[auto] {$$} (m-2-2)%;
(m-2-2) edge node[auto] {$\dN$} (m-2-3);
\path[|->,dashed, font=\scriptsize]
(m-3-2) edge node[auto] {$\dN$} (m-2-2);
\end{tikzpicture}
\]

Although the construction developed in this paper does not exactly follow the
above description,  
it is largely inspired by its main aspects:
\begin{itemize}
\item Which linear combinations of products of degree $1$ elements are boundary?
 Note that these linear combinations are degree $2$ elements. 
\item Which degree $1$ elements $c$  are mapped by
  the differential onto these linear combinations of products?
\end{itemize}

The strategy of this paper finds linear combinations $\sum
\alpha_{i,j}c_i \cdot c_j$ that have a zero differential, and considers under what
conditions they can be written as an explicit boundary, i.e. as $\dN(e)$ for
some explicit cycle $e$ in $\cNg{X}{1}$. 

In weight $p$, we will consider linear combinations constructed from
elements obtained in lower weight; under some geometric conditions  the cycle $c$
can be constructed  easily. It is the pull-back of $\sum \alpha_{i,j}c_i \cdot
c_j$ induced by the multiplication map ($\square^1 \simeq \A ^1$, $X=\A^1$): 
\[
\begin{tikzpicture}
\matrix (m) [matrix of math nodes,
 row sep=2em, column sep=4em, 
 text height=2.5ex,  text depth=0.25ex] 
{X \times \A^1 \times \A^{2p-2}
%=\A^1 \sm \{0,1\} \times \A^1 \times \A^{2p-2} 
&
% \A^1 \sm \{0,1\} \times \A^{2p-2}=
 X \times \A^{2p-2}
 \\
~ \quad (t, s, x_1,\ldots,x_{2p-2})\qquad ~ & 
~\qquad ~(ts, x_1,\ldots,x_{2p-2}) .\qquad ~\\};
\path[->, font=\scriptsize]
(m-1-1) edge node[auto] {} (m-1-2);
\path[|->, font=\scriptsize]
(m-2-1) edge node[auto] {} (m-2-2);
\end{tikzpicture}
\]

Even though it is not stated formally in their paper, it is reasonable to believe
that  Bloch and Kriz used this idea to build their cycles
$\on{Li}_n^{cy}(z)$. Thus, it is natural to recover
these cycles using the method described above. However,
the cycles corresponding to multiple polylogarithms constructed using this method are
different from the ones proposed by Gangl, Goncharov and Levin in \cite{GanGonLev05}.

In particular, using the above pull-back by the multiplication insures that the
constructed cycles on $\ps$ admit an extension to $\cNg{\A^1}{\bullet}$
which is  dominant over $\A^1$ with pure relative dimension (cf. Definition
\ref{def:equi}), and have an empty fiber at $0$ and a well defined fiber at
$1$. Such cycles are called 
\emph{equidimensional} (over $\A^1$).

This geometric construction proceeds
within a combinatorial setting. More precisely, for any Lyndon
word $W$ in the letters $\{0, 1\}$, we consider
linear combinations of decorated rooted trivalent trees $\T W(x)$. These $\T
W(x)$ are dual to the basis of Lyndon brackets of the free Lie algebra; hence the
subscript $W^*$ in $\T W(x)$. In this notation,  $x$ 
denotes a parameter in $\A^1$ labeling the root. By dualizing the action of the
Lie algebra on itself by Ihara's special derivations, we obtain a twisted
cobracket $\dc$ which can applied to $\T W(x)$. 
 The main point of the
combinatorial setting is the following result (Proposition \ref{dcyTw}).

\begin{thma} For any Lyndon word $W$ in the
  alphabet $\{0,1\}$, the 
  cobracket of $\T W(x)$ can be expressed as:
\begin{equation} \label{equationintroT}
\dc( \T W(x)) = \sum \alpha_{i,j}\T{W_i}(x)\w \T{W_j}(x)+ \sum
\beta_{k,l}\T{W_k}(x)\w \T{W_l}(1)   
\end{equation}
where $\w$ is the exterior product and $W_i$, $W_j$, $W_k$
and $W_l$ are Lyndon words of length strictly smaller than $W$. The coefficients
$\alpha_{i,j}$'s and $\beta_{k,l}$'s are integers.
\end{thma}
Note that Gangl, Goncharov and Levin in
\cite{GanGonLev05} 
defined a differential $\dggl$ on a  commutative differential graded algebra
built on a closely related type of tree. These authors' approach  
is mainly based on the universal enveloping algebra of 
the Lie algebra used in this paper (see Remark \ref{relationGGL}). The
differential $\dggl$  was originally constructed to mimic the differential in
$\cNg{X}{\bullet}$ while the cobracket $\dc$ reveals the underlying structure of 
Ihara's (co)action by special derivations. 

The above Theorem controls the combinatorial structure of the elements we want to 
build. %
It plays a central role in  constructing our explicit algebraic cycles in a general
framework. Modifying the above ``differential system'' (a cobracket inducing a
differential),  we inductively
construct cycles $\Lc_W$  
corresponding to $\T W(x)$ (which we can think of as $\T W(x)-\T W(0)$) and cycles
$\Lcu_W$ corresponding to the difference 
$\T W(x)-\T W(1)$. In this way, we obtain (cf. Theorem \ref{thm:cycleLcLcu})
algebraic cycles that 
are expected, when specialized at $1$, to correspond to multiple zeta values under
Bloch-Kriz Hodge realization functor (which, up to the author understanding, is not completely computable in
any explicit and coherent way). 
\begin{thma} Let $X=\p^1\sm\{0,1, \infty\}$. 
 For any Lyndon word $W$ in the alphabet  $\{0,1\}$ of length $p\geqs 2$, there exists
  a non zero cycle $\Lc_W$  in  $\cNg{X}{1}(p)$, i.e. a cycle of codimension
  $1$ in $X \times \square^{2p-1}$, such that:%
\begin{itemize}
\item $\Lc_W$ has a decomposable boundary whose explicit expression is derived
  from Equation \eqref{equationintroT}, 
\item  $\Lc_W$ admits an equidimensional extension to  $\A^1$ with
  empty fiber at $0$.
\end{itemize}
A similar statement holds for $1$ in place of $0$.
\end{thma}
\begin{rem}
The main arguments in support of the importance of cycles $\Lc_W$ are :
\begin{itemize}
\item When $W=0 \cdots 0 1$ (only one ``$1$''), we recover the class of classical
  polylogarithms cycles introduced by Bloch and Kriz \cite{BKMTM}.
\item More generally, the differential equation satisfied by cycles $\Lc_W$ is
  controlled by  Ihara's special derivation and the induced cobracket.% At the
%  motivic level, ``controlled by'' can be replaced by ``exactly given by'' (see
%  \cite{SouB arbase})

On motivic iterated integrals, Ihara's cobracket is known to agree
  with Goncharov motivic coproduct which  corresponds to the differential
  equation satisfied by multiple 
  polylogarithm (see also the discussion on page  \pageref{labelpage:lastcomments}).

%The last step of our general program shows that the coproduct of the  motive
%induced by the cycle $\Lc_W$ is exactly  given by  Ihara's cobracket or
%equivalently by Goncharov motivic coproduct (see \cite{SouBarbase}).

 Hence the differential equations satisfied by cycles $\Lc_W$ is
  very closely related to the differential equations of multiple polylogarithms.  
\item Theorem \ref{thm:cycleLcLcu} insures that cycles $\Lc_W$ can be
  specialized at the point $1$ which is desirable because multiple zeta value
  are specialization of multiple polylogarithm at the point $1$. 
\item The computation
  of the actual integral for $W=011$ is done in Section \ref{subsec:int} and
  give $-\zeta(2,1)$ after specialization at $1$. Such computations can be done by
hand in small weight but they  do not increase the global understanding of the
Hodge realization for cycles above.

%\item $\Lc_W|_{\{x=z\}}$ gives an element of $\cNg{\Q}{\bullet}$. The next steps
%  of our general program show that the cycle $\Lc_W|_{\{x=z\}}$ corresponds (as a
%  motive) to a linear combination of motivic multiple
%  polylogarithms (cf.  \cite{SouBarbase}).

% (the last
%  steps of this program, proved in \cite{SouBarbase} show that they actually
%  agree at the motivic level).
\end{itemize}
All these reasons leads us to think of $\Lc_W|_{\{x=1\}}$, specialization of the
cycle $\Lc_W$ at $1$, as corresponding to a (linear combination of) multiple
zeta  value(s). 
%\begin{rem}

This intuition is confirmed in \cite{SouBarbase} which shows
that motives associated to cycles  $\Lc_W|_{\{x=z\}}$ generate the Deligne-Goncharov
motivic fundamental group. 

This is a consequence of this paper because their cobracket (corresponding to the
differential of cycles) is exactly given by Equation \eqref{equationintroT} and
hence by Ihara's cobracket.
%; that is it agrees with Goncharov motivic coproduct
%for iterated integrals. 
\end{rem}

The paper is organized as follows:
\begin{itemize}
\item The next section (Section \ref{mainsec:complex}) is devoted to a general
  review of the Bloch-Kriz cycle 
  complex. 
In particular, we detail the construction
  of the cycle complex and recall some of its main properties (relation to
  higher Chow groups, localization long exact sequence, etc.).
\item Then, we  apply in Section \ref{subs:poly}  our strategy to the nice examples of
  polylogarithms as described in \cite{BKMTM}. Then 
  we present the main difficulties via an example in weight $3$,
  and explain how to overcome them in general.
\item Next, in Section \ref{sec:Combi}, we deal with the combinatorial
  situation, beginning by presenting the 
  trivalent trees attached to Lyndon words and their relations with the free Lie
  algebra $\Lie(X_0,X_1)$ on two generators. We then review Ihara's action by
  stable derivations. Next we introduce linear combinations of
  trees $\T W$ corresponding to the 
  dual situation, i.e. in the Lie coalgebra graded dual to $\Lie(X_0,X_1)$,
  and study the cobracket corresponding to Ihara's action. Note that we are not
  simply looking at Ihara's Poisson bracket; we are also keeping track of the
  structure coefficients of the action.

 This leads to Proposition \ref{dcyTw}, which proves that the
  image of $\T W$ under $\dc$ is decomposable in terms of $\T U$ (with $U$ of
  smaller length).
\item In Section \ref{treetocycle} we prove our main Theorem. 
It begins by presenting  some properties of equidimensional
  algebraic cycles over $\ps$ and $\A^1$. Then, we study the relation between the two
  situations and explain how the pull-back by the multiplication (resp. by 
a twisted
  multiplication) gives a homotopy between the identity and the fiber at $0$
  (resp. at $1$) pulled back to a cycle over $\A^1$ by $p:\A^1 \ra \{pt\}$. Finally, the
  above work allows us to inductively construct the desired families of cycles
  $\Lc_W$ and $\Lcu_W$ in Theorem \ref{thm:cycleLcLcu}. 
 We conclude this section by computing the integral attached to the cycle
 $\Lc_{011}$. Its 
  specialization at the point $1$ is $-\zeta(2,1)$.
\item The last section is devoted to some concluding remarks. In particular, we
  show how our construction passes to the setting of 
quasi-finite cycles used in the motivic context provided  by
\cite{LEVTMFG}.
 \end{itemize}
\section{The cycle complex over $\ps$}\label{mainsec:complex}
\subsection{Commutative differential graded algebras}
\label{subsec:1mini}

Let us recall some definitions and properties of commutative differential graded
algebras (cdga's) over $\Q$. 
\begin{defn}[cdga] A commutative differential graded algebra $A$ is a
  commutative graded algebra (with unit) $A= \oplus_n A^n$ over $\Q$ together with a graded
  homomorphism $d=\oplus d^n$, $d^n : A^n \lra A^{n+1}$ such that
\begin{itemize}
\item $d^{n+1}\circ d^n=0$
\item $d$ satisfies the Leibniz rule
\[
d(a\cdot b)= d(a)\cdot b +(-1)^n a \cdot d(b) \qquad \mbox{for }a \in A^n, \, b
\in A^m.
\]
\end{itemize}    
\end{defn}
We recall that a graded algebra is \emph{commutative} if and only if for any homogeneous
elements $a$ and $b$, we have
\[
ab=(-1)^{\deg(a) \deg(b)} ba.
\]
\begin{defn} A cdga $A$ is 
\begin{itemize}
\item \emph{connected} if $A^n=0$ for all $n<0$ and $A^0=\Q \cdot 1$.
\item \emph{cohomologically connected} if $\HH^n(A)=0$ for all $n<0$ and
  $\HH^0(A)=\Q\cdot 1$.
\end{itemize}
\end{defn} 
In our context, the cdga involved are not necessarily connected, but come with
an Adams grading.  
\begin{defn}[Adams grading] An \emph{Adams graded} cdga is a cdga $A$ together
  with a decomposition into subcomplexes $A= \oplus_{p \geqs 0} A(p)$ such that
\begin{itemize}
\item $A(0)=\Q$ is the image of the algebra morphism $\Q \lra A$;
\item The Adams grading is compatible with the product of $A$, i.e.
\[
A^k(p)\cdot A^l(q) \subset A^{k+l}(p+q).
\]
However, no sign is introduced as a consequence of the Adams grading.
\end{itemize}
\end{defn}
For an element $a \in A^k$, we call $k$ the cohomological degree and denote it by
$|a|:=k$. In the case of an Adams graded cdga, for $a\in A^k(p)$, we call $p$
its weight or Adams degree and denote it by $\we(a):=p$.

We assume that all the commutative differential graded algebras are equipped with an augmentation $\ve : A \lra \Q$. Note that an Adams graded cdga $A$ has a canonical
augmentation $A \lra \Q$ with augmentation ideal $A^+=\oplus_ {p \geqs 1} A(p)$. 

\subsection{General construction of the Bloch-Kriz cycle complex}
This subsection is devoted to the construction of the cycle complex as presented
in \cite{BlochACHKT,BlochLMM,BKMTM, LevBHCG}. 

For simplicity we work over $\Spec(\Q)$. For $n \geqs 1$, let $\square^n$ be the
algebraic $n$-cube 
\[
\square^n=(\p^1\sm \{1\})^n
\] 
with the convention that $\square^0=\Spec(\Q)$.
Insertion morphisms $s_{i}^{\ve} :\square^{n-1} \lra \square^{n} $ are given by 
the identification
\[
\square^{n-1}\simeq \square^{i-1}\times \{\ve\}\times \square^{n-i}
\]
for $\ve=0,\infty$. Similarly, for $I \subset \{1, \ldots , n\}$ and $\ve : I
\ra \{0,\infty\}$, we define $s_I^{\ve} : \square^{n-|I|} \lra \square^n$.
\begin{defn} A \emph{face} $F$ of codimension $p$ of $\square^n$ is the image
  $s_I^{\ve}(\square^{n-p})$ for some $I$ and $\ve$  as above such that
  $|I|=p$. 

In other word, a codimension $p$ face of $\square^n$ is given by the equation
$u_{i_k}=\ve_k$ 
for $k$ in $\{1, \ldots, p\}$ and  $\ve_k$ in $\{0,\infty\}$ where $u_1, \ldots,
u_n$ are the usual affine coordinates on $\p^1$. 
\end{defn} 
The permutation group $\Sn$ acts on $\square^n$ by permutation of the factors.
\begin{rem}~
 \begin{itemize}
\item
In some references, \cite{LevBHCG, LEVTMFG} for example, $\square^n$ is
defined to be the usual affine space $\A^n$, and the faces are obtained
by setting various
coordinates equal to $0$ or $1$. This makes the correspondence with the
``usual'' cube more natural. However, the above presentation, which agrees with
\cite{BKMTM} or \cite{GanGonLev05}, makes some comparisons and some formulas
``nicer''. In particular,  the relation between the construction in the setting
$\square^1=\p^1\sm \{1\}$ and the Chow group
$\CH^1(X)_{\Q}$ is simpler.
\item Let $\textbf{Cube}$ be the subcategory of the category of finite  sets whose
  objects are $\ul{n}=\{0,1\}^n$ and whose morphisms are generated by forgetting a
  factor, inserting $0$ or $1$, and permutation of the factors, these morphisms
  being subject to natural relations. Similarly
  to the usual description of a simplicial object,
  $\square^{\bullet}$ is a functor from $\textbf{Cube}$ into the category of
  smooth ${\Q}$-varieties, and the various $\square^n$ are geometric equivalents of
  $\ul n$.
\end{itemize}
\end{rem}

Let $X$ be a smooth quasi-projective variety over ${\Q}$.
\begin{defn}\label{defZc} Let $p$ and $n$ be non negative integers. Let $\Zc^p(X,n)$ be the
  free group generated closed irreducible
  subvarieties of $X \times \square^n$ of codimension $p$ which intersect all
  faces $X \times F$ properly (where $F$ is a face of $\square^n$). That is:
\[
\Z\left< W \subset X \times \square^n \text{ such that} \left\{
\begin{array}{l}
W \text{ is closed and irreducible;} \\
\on{codim}_{X\times F}(W\cap X \times F)=p \\
 \text{or } W \cap (X\times F)= \emptyset
\end{array}
\right.
\right>
\]
\end{defn}
\begin{rem}~ \label{remZc}\begin{itemize}
\item A subvariety $W$ of $X \times \square^n$ as above is called \emph{admissible}.
\item As the projection $p_i : \square^n \ra \square^{n-1}$ forgetting the
  $i$-th factor is smooth,  we have the
  corresponding induced pull-back: $p_i^* : \Zc^p(X,n-1) \ra \Zc^p(X,n)$.
\item  $s_i^{\ve}$ induces a regular closed embedding $X \times \square^{n-1}
  \ra X \times \square^{n}$ which is of local complete intersection. As
  we are considering only admissible cycles, i.e. cycles in ``good position'' with
  respect to the faces, $s_i^{\ve}$ induces  $s_i^{\ve \, *} : \Zc^p(X,n) \ra \Zc^p(X,n-1)$.
\item The morphism $\dN = \sum_{i=1}^{n} (-1)^{i-1}(s_i^{0,*}-s_i^{\infty,*})$ induces a
  differential 
\[
\Zc^p(X,n) \lra \Zc^p(X,n-1).
\]
\end{itemize}
\end{rem}
The action of $\Sn$ on $\square^n$ can be extended to an action of the semi-direct
product $G_n=(\Z / 2\Z)^n \rtimes \Sn$ where each $\Z/2\Z$ acts on $\square^1$ by
sending the usual affine coordinate $u$ to $1/u$.  The sign representation of
$\Sn$ extends to a sign representation $G_n \longmapsto
\{\pm 1\}$. Let $\Alt_n \in \Q[G_n]$ be the corresponding projector. 
\begin{defn} Let $p$ and $k$ be integers with $p>0$. Set
\[
\mc N_X^k(p)=\Alt_{2p-k}(\Zc(X,2p-k)\otimes \Q).
\]
We will refer to $k$ as the \emph{cohomological degree}, and to $p$ as the \emph{weight}.
\end{defn}
\begin{rem} In this presentation, we have not dealt with degeneracies (images in
  $\Zc(X,n)$ of $p_i^*$) because we use an alternative version with
  rational coefficients. For more details, see the first section of
  \cite{LevSM} which presents the general setting of cubical objects. A similar
  remark was made in \cite{BKMTM}[after equation  (4.1.3)]. 
\end{rem}
\begin{defn}[Cycle complex] For $p$ and $k$ as above, the pull-back
\[
s_i^{\ve\, *}: \Zc^p(X,2p-k) \lra \Zc^p(X,2p-k-1)
\]
induces a morphism $\dN_i^{\ve}:\mc N_X^{k}(p) \lra \mc N_X^{k+1}(p) $. Thus,
the differential $\dN$ on $\Zc^p(X,2p-k)$ extends to a  differential
\[
\dN=\sum_{i=1}^{2p-k} (-1)^{i-1} (\dN_{i}^0 - \dN_{i}^{\infty}) : 
 \mc N_X^{k}(p) \st{\dN}{\lra} \mc
N_X^{k+1}(p).
\]
 Let $\mc N_X^{\bullet}(p)$ be the complex
\[
\mc N_X^{\bullet}(p) : \qquad \cdots \lra \mc N_X^{k}(p) \st{\dN}{\lra} \mc
N_X^{k+1}(p) \lra \cdots 
\]

Define the cycle complex as 
\[
\mc N_X^{\bullet}=\Q \oplus \bigoplus_{p \geqs 1}\mc N_X^{\bullet}(p).
\]
\end{defn}
In \cite[\S 5]{LevBHCG} and in \cite{LEVTMFG}[Example 4.3.2], Levine proved
 the following proposition.
\begin{prop} Concatenation of the cube factors and pull-back by the diagonal
\[
X\times \square^n \times X \times \square^m \st{\sim}{\ra} 
X\times X \times \square^n \times  \square^m \st{\sim}{\ra} 
X\times X \times \square^{n+m}\st{\Delta_X}{\longleftarrow}
 X \times \square^{n+m}
\]
induces, after applying the $\Alt$ projector, a well-defined product:
\[
\mc N_X^k(p)\otimes \mc N_X^l(q) \lra \mc N_X^{k+l}(p+q)  
\]
 denoted by $\cdot$
\end{prop}
\begin{rem} The smoothness hypothesis on $X$  allows us to consider the
  pull-back by the diagonal $\Delta_X : X \lra X \times X$ which is, in this case,
  a local complete intersection. 
\end{rem}
We have the following theorem (also stated in \cite{BKMTM, BlochLMM} for
$X=\Sp({\Q})$).
\begin{thm}[{\cite[Theorem 4.7 and \S 5]{LevBHCG}}]
The cycle complex $\mc N_X^{\bullet}$ is a differential graded commutative
algebra. In weight $p$, its cohomology groups are the $p$-th higher Chow group of $X$:
\[
\HH^k(\cNg{X}{\bullet}(p))=\CH^p(X,2p-k)_{\Q},
\]  
where $\CH^p(X,2p-k)_{\Q}$ stands for $\CH^p(X,2p-k)\otimes {\Q}$.
\end{thm} 
Moreover, this construction is functorial in $X$ with respect to flat pull-back  and
proper push-forward (up to the usual shifts in degree and weight). Using Levine's
work \cite{LevBHCG}, we have a more general pull-back functoriality on the level
of  cohomology groups; we could also use Bloch's moving Lemma \cite{BlochMLHCG}.
\subsection{Some properties of Higher Chow groups}
In this section, we present some well-known properties of the higher Chow groups and some
applications that will be used later. Proof of the
different statements can be found in \cite{BlochACHKT} or \cite{LevBHCG}.
\subsubsection{Relation with higher $K$-theory}
Higher Chow groups, in a simplicial version, were first introduced in
\cite{BlochACHKT} in order to achieve a better understanding of
the $K$-groups of higher $K$-theory.  In \cite{LevBHCG}[Theorem 3.1], Levine 
gives a cubical version of the desired isomorphisms.
\begin{thm}[\cite{LevBHCG}]Let $X$ be a smooth quasi-projective ${\Q}$-variety,
 and let  $p$, $k$ be two positive integers.  Then
\[
\CH^p(X, 2p-k)_{\Q}\simeq\Gr^p K_{2p-k}(X)\otimes \Q
\]
\end{thm}
In particular, using the work of Borel \cite{BorCAG}, computing the rank of $K$-groups of
a number field in the case  $k=2$ and $p\geqs 2$, we find
\begin{equation}\label{eqCHQ}
\CH^p(\Q, 2p-2)_{\Q}\simeq\Gr^p K_{2p-2}(\Q)\otimes \Q=0.
\end{equation}
\subsubsection{$A^1$-homotopy invariance}
From Levine \cite{LevBHCG}[Theorem 4.5], we can deduce the following proposition.
\begin{prop}[\cite{LevBHCG}] Let $X$ be as above and let $p_X$ be the projection $p_X :
  X \times \A^1 
  \lra X$. Then the projection $p_X$ induces (by flat pull-back) a
  quasi-isomorphism for any positive   integer $p$
\[
p_X^*: \mc N_{X}^{\bullet}(p) \st{q.i.}{\lra} \mc N_{X\times \A^1}^{\bullet}(p)
\]
\end{prop}
Moreover, an inverse to $p_X^*$ on the cohomology is given by the pull-back by
the zero section $i_0^* : \HH^k(\Nc[X
\times \A^1](p)) \lra \HH^k(\Nc[X](p))$.
\begin{rem}\label{remA1hom}
The proof of Levine's theorem also tells how this quasi-isomorphism arises
using the multiplication map $\A^1 \times \A^1 \lra \A^1$. This leads to
the proof of Proposition \ref{multhomo}.
\end{rem}
We now apply the above result in the case where  $X=\Sp(\Q)$, and use the
relation with $K$-theory via equation \eqref{eqCHQ}.
\begin{coro}\label{H2A1} The second cohomology group
  of $\mc N_{\A^1}^{\bullet}$ 
  vanishes:
\[
\forall\, p\geqs 1 \qquad \HH^2(\Nc[\A^1](p))\simeq\CH^p(\A^1,2p-2)_{\Q}\simeq 
\CH^p(\Q,2p-2)_{\Q}=0.
\]
\end{coro}
\subsubsection{Localization sequence}  
Let $W$ be a smooth closed subvariety of pure codimension $d$ of a smooth
quasi-projective variety $X$. Let $U$ denote the open complement $U=X \sm W$. A
 version adapted to our needs of Theorem 3.4 in \cite{LevBHCG} gives the
 localization sequence for higher Chow groups.  
\begin{thm}[\cite{LevBHCG}]
Let $p$ be a positive integer and $l$ an integer. There is a long exact sequence
\begin{equation}\label{locseq}
\cdots \lra \CH^p(U,l+1)_{\Q} \st{\delta}{\lra} \CH^{p-d}(W,l)_{\Q}  \st{i_{*}}{\lra}
\CH^p(X,l)_{\Q}  \st{j^*}{\lra} \CH^p(U,l)_{\Q} \st{\delta}{\lra} \cdots
\end{equation}
where $i: W \ra X$ denotes the closed immersion and $j:U\ra X$ the open one.
\end{thm}
\begin{rem} Here, $i_{*}$ and $j^*$ denote the usual push-forward for proper
  morphisms and pull-back for flat ones. 
\end{rem}
In order to study the cycle complex over $\ps$, we begin by applying the above
theorem to the case  where $X=\A^1$, $U=\p^1\sm\{0,1,\infty\}$ and $W=\{0,1\}$. 
\begin{coro}\label{locP1}For $p\geqs 0$ and $k \in \Z$, we have the following
description of $\HH^k(\mc N^{\bullet}_{\ps})$: 
\begin{multline*}
\HH^k(\mc N_{\ps}^{\bullet})(p)\simeq \\
 \HH^k(\mc N_{\Q}^{\bullet})(p)\oplus
\left(\HH^{k-1}(\mc N_{\Q}^{\bullet})(p-1)\otimes \Q [L_0] \right) 
\oplus \left(\HH^{k-1}(\mc N_{\Q}^{\bullet})(p-1)\otimes \Q [L_1]\right),
\end{multline*}
where $[L_0]$ and $[L_1]$ are cohomology classes represented by cycles $L_0$ and
$L_1$ respectively which  are in cohomological degree $1$ and weight $1$ (i.e.
of codimension $1$).
\end{coro}
\begin{proof}
For any integer $l$, the long exact sequence above gives 
\begin{multline}\label{seq:CHlocA1}
\cdots\lra  \CH^{p-1}(\{0,1\},l)_{\Q}  \st{i_{*}}{\lra} \CH^p(\A^1,l)_{\Q} \lra
\CH^{p}(X,l)_{\Q}  \st{\delta}{\lra} \\
\CH^{p-1}(\{0,1\},l-1)_{\Q}  \st{i_{*}}{\lra} \CH^p(\A^1,l-1)_{\Q} \lra \cdots
\end{multline}
The map $i_*$ is induced by the inclusions $i_0$ and $i_1$ of $0$ and $1$ into
$\A^1$. The morphism $i_0^* : \HH^{k}(\Nc[\A^1]) \lra \HH^k(\Nc[\{0\}])$, and
more generally $i_x^*$ for 
any $\Q$ point $x$ of $\A^1$, is an isomorphism inverse to
$p_{\Spec(\Q)}^* : \A^1 \ra \Spec(\Q)$. Hence the Cartesian diagram
\[
\begin{tikzpicture}
\matrix (m) [matrix of math nodes,
 row sep=3em, column sep=3em, 
 text height=1.5ex, text depth=0.25ex] 
 {\emptyset & \Sp(\Q) \\
 \Sp(\Q) & \A^1\\};
 \path[->,font=\scriptsize]
 (m-1-1) edge node[auto] {$$} (m-1-2)%;
(m-1-2) edge node[auto] {$i_x$} (m-2-2)%;
(m-2-1) edge node[auto] {$i_0$} (m-2-2)%;
(m-1-1) edge node[auto] {$$} (m-2-1);
%\path[font=\scriptsize]
(m-1-1) edge node[descr] {$\square$} (m-2-2);
%\path[->,font=\scriptsize, bend left]
%(m-1-4) edge node[auto] {$x^*$} (m-1-3);
 \end{tikzpicture}
\]
shows that  $i_{0,*}$ (and respectively $i_{1,*}$) is $0$ on cohomology. 

In particular, the sequence \eqref{seq:CHlocA1} is a collection  of short
exact sequences
\begin{equation*}
0 {\lra}
\CH^p(\A^1,l)_{\Q} \lra 
\CH^{p}(X,l)_{\Q}  \st{\delta}{\lra} 
\CH^{p-1}(\{0,1\},l-1)_{\Q} {\lra}0 .
\end{equation*}

 Thus, for fixed $l$, using the $\A^1$-homotopy property and the fact that
 \[
\CH^p(\{0,1\},l)\simeq \CH^p(\Sp(\Q),l) \oplus \CH^p(\Sp(\Q),l),
\] we obtain the following short exact sequence
\begin{equation*}
0 {\lra}
\CH^p(\Sp(\Q),l)_{\Q} \lra 
\CH^{p}(X,l)_{\Q}  \st{\delta}{\lra} 
\CH^{p-1}(\Sp(\Q),l-1)_{\Q}^{\oplus 2} {\lra}0,
\end{equation*}
 and an isomorphism
\[
\CH^{p}(X,l)_{\Q}  \st{\sim}{\lra} \CH^{p}(\A^1,l)_{\Q}\oplus
\CH^{p-1}(\Sp(\Q),l-1)_{\Q}^{\oplus 2}. 
\]
The relation between the cohomology groups of $\Nc[X](p)$ and the higher Chow
groups concludes the proof.
\end{proof}
\begin{rem} 
The above corollary for $k=1$ gives us a description of
$\HH^1(\cNg{\ps}{\bullet})$ which is the key object of the first step of the
$1$-minimal construction mentioned in the introduction. Hence this corollary is
the starting point of the construction of our algebraic cycle.

The generators $L_0$ and $L_1$ will be given explicitly in terms of
  cycles in $\Nc[\ps]$ at Subsection \ref{subs:poly}. This allows us to
  follow the $1$-minimal model construction in order to build our algebraic cycles. 
\end{rem}
\subsection{The cycle complex over $\p^1 \sm \{0,1,\infty\}$ and mixed Tate motives}   
Levine in \cite{LEVTMFG} makes the link between the category of mixed Tate
motives (in the sense of Levine \cite{KTMMLevine} or Voevodsky \cite{Vo00}) over a
base $X$  and the cycle complex $\mc N_X$. 
 The relation between mixed Tate motives and
the cycle complex was developed earlier for $X=\Sp(\Q)$, the spectrum of
a number field, by Bloch and Kriz in \cite{BKMTM}.

In what follows, $X$ still denotes a smooth,
quasi-projective 
variety over ${\Q}$. We will work with coefficients in $\Q$.

 Under more general conditions, Cisinski and D\'eglise \cite{CiDegTCM} defined
a triangulated category $\DM(X)$ of (effective) motives over a base with the
expected properties. Levine's work
\cite{LevTM,LEVTMFG} shows that when the motive of  $X$ is mixed Tate over
$\Spec(\Q)$ and satisfies the Beilinson-Soul\'e vanishing
conjecture, there exists a Tannakian category $\MTM(X)$ of mixed Tate motives
over $X$.

Together with defining an avatar of $\mc N_X$ in $\DM(X)$, Levine
\cite{LEVTMFG}[Theorem 5.3.2 and beginning of the section 6.6]
shows that when $X$ satisfies the above conditions, we can identify the
Tannakian group associated with $\MTM(X)$ with 
the spectrum of the $\HH^0$ of the bar construction 
over the cdga $\mc N_X$: 
\[
G_{\MTM(X)}\simeq \Sp(\HH^0(B(\mc N_X))).
\]

Then he uses a relative bar-construction in order to relate the natural
morphisms
\[
p^* : \MTM(\Sp({\Q})) \lra \MTM(X), \qquad x^* : \MTM(X) \lra \MTM(\Sp({\K})),
  \]
induced by the structural morphism $p : X \ra \Sp({\Q})$ and a choice of a
${\Q}$-point $x$ to the motivic fundamental group of $X$ at the base point
$x$ defined by Goncharov
and Deligne, $\pi_1^{mot}(X,x)$ (see \cite{GFPLDeli} and \cite{DG}).

In particular, applying this to the  the case $X=\ps$, we have the
following result. 
\begin{thm}[{\cite{LEVTMFG}[Corollary 6.6.2]}]\label{pi1exactseq} Let $x$ be a
  $\Q$-point of   $\ps$. Let $G_{\ps}$ and $G_{\Q}$ denote the spectrum of
  $\HH^0(B(\mc N_{\ps}))$ and $\HH^0(B(\mc N_{\Q}))$ respectively.
 Then there is a split exact sequence:

 \begin{equation}\label{eq:motsespi1}
 \begin{tikzpicture}[baseline=(current bounding box.center)]
\matrix (m) [matrix of math nodes,
 row sep=0.6em, column sep=2.5em, 
 text height=1.5ex, text depth=0.25ex] 
 {1 & {\pi_1^{mot}(\ps,x)} & G_{\ps} &G_{\Q} & 1
 \\};
 \path[->,font=\scriptsize]
 (m-1-1) edge node[auto] {} (m-1-2)%;
(m-1-2) edge node[auto] {} (m-1-3)%;
(m-1-3) edge node[auto] {$p^*$} (m-1-4)%;
(m-1-4) edge node[auto] {} (m-1-5);
\path[->,font=\scriptsize, bend left]
(m-1-4) edge node[auto] {$x^*$} (m-1-3);
 \end{tikzpicture}
 \end{equation}
where $p$ is the structural morphism $p: \ps \lra \Sp(\Q)$.
\end{thm}
%
%----------------------------
Let $\Mm[\Q]$ and $\Mm[\ps]$ denote  the Lie
coalgebra given by the set of indecomposable elements of the $\HH^0$ of the
bar construction of $\Nc[\Sp(\Q)]$ and $\Nc[\ps]$ respectively. Then, taking the
ring of functions and  passing to the set of indecomposable elements, the short
exact sequence 
\eqref{eq:motsespi1}  can be reformulated in terms of Lie coalgebras.

Thus there is a split exact sequence of Lie coalgebras:
 \[
 \begin{tikzpicture}
\matrix (m) [matrix of math nodes,
 row sep=0.6em, column sep=2.5em, 
 text height=1.5ex, text depth=0.25ex] 
 {0 &\Mm[\Q] & \Mm[\ps] &\cLie^{geom}_{\ps} & 0
 \\};
 \path[->,font=\scriptsize]
 (m-1-1) edge node[auto] {} (m-1-2)%;
(m-1-2) edge node[auto] {} (m-1-3)%;
(m-1-3) edge node[auto] {} (m-1-4)%;
(m-1-4) edge node[auto] {} (m-1-5);
 \end{tikzpicture}
 \]  
where $\cLie^{geom}_{\ps}$ is the graded dual of  the Lie algebra associated to
$\pi_1^{mot}(X,x)$.

In particular $\cLie^{geom}_{\ps}$ is the Lie coalgebra which is the graded dual of the free Lie
algebra on two generators $\Lie(X_0,X_1)$. This observation motivated the
study of this Lie coalgebra in Section \ref{sec:Combi} in order to understand
the combinatorial structure of our construction. This Lie coalgebra is presented
in a ``trivalent tree'' version because of the combinatorial construction of
Gangl, Goncharov and Levin \cite{GanGonLev05}, which is based on linear
combination of trivalent trees. Their approach seems to be related to the graded
dual of the universal enveloping algebra of  $\Lie(X_0,X_1)$.

\section{Low weight examples and  cycles corresponding to polylogarithms} 
\label{subs:poly}From now until the end of the article,
 we let $X$ denote $\ps$.

In this section, we give a first presentation of our strategy to construct
general cycles in $\Nc$ 
corresponding to multiple polylogarithms by examining the simple case of the
polylogarithms $Li_n^{\C}$. We will pay special attention to the Totaro cycle,
which is known to correspond to the function $Li^{\C}_2(z)$, and then explain how 
its construction can be generalized to obtain cycles already present in \cite{BKMTM} and
\cite{GanGonLev05} corresponding to the functions
$Li_n^{\C}(z)$. 

\subsection{Two weight $1$ examples of cycles generating the $\HH^1$}
We want to construct cycles in $\Nc$ in order to obtain the inductive
construction of the $1$-minimal model. This means that we will
\begin{enumerate}
\item find in $\mc N_ X^2$ linear combinations
of products of already constructed cycles that are boundaries, i.e. $\dN(c)$ for some
cycle $c$ in $\cN[1]$;
\item explicitly construct the desired $c$. 
\end{enumerate}
The first step begins with a basis of $\HH^1(\cNg{X}{\bullet})$. However, as we only
want a description of the geometric part of the $1$ minimal model (i.e.~relative to the situation over $\Spec(\Q)$),  we do
not need to consider a 
full basis of $\HH^1(\cNg{X}{\bullet})$. We saw that
$\HH^1(\cNg{X}{\bullet})$  (Corollary
\ref{locP1}) is the direct sum of $\HH^1(\cNg{ \Q}{\bullet})$
and two copies of $\HH^{0}(\mc N_{\Q}^{\bullet})$.
\begin{lem} Let $\Gamma_{0}$ and $\Gamma_1$ be 
 the graph of $\rho_0 : X \lra \square^1$ which sends $x$ to $x$
  and the graph of  $\rho_1 : X \lra \square^1$ which sends $x$ to $1-x$
respectively. Then
  $\Gamma_0$ and $\Gamma_1$ define 
admissible algebraic cycles in $X \times \square^1$. Applying the projector
$\Alt$ to the alternating elements gives two elements $L_0$ and $L_1$ in $\mc
N_X^1$, and we have
\[
\HH^1(\mc N_{\ps}^{\bullet})\simeq \HH^1(\cNg{ \Q}{\bullet}) \oplus 
\left( \HH^{0}(\mc N_{\Q}^{\bullet})\otimes \Q L_0\right) 
\oplus
\left( \HH^{0}(\mc N_{\Q}^{\bullet})\otimes \Q L_1\right).
\]
\end{lem}
When we speak about parametrized cycles, we will usually omit the projector $\Alt$
and write
\[
L_{0} =[x; x]\qquad \mx{and} \qquad L_{1}=[x; 1-x] 
\qquad \subset X \times \square^1
\]
where the notation $[x;f(x)]$ denotes the set
\[
\{(x,f(x)) \mbox{ such that } x \in X\}.
\]

\begin{proof}
First of all, note that $L_0$ and $L_1$ are codimension $1$ cycles in
$X \times \square^{1}=X \times \square^{2*1-1}$. Moreover, as 
\[
L_0 \cap \left(X\times\{\ve\}\right)=L_0 \cap \left(\ps \times\{\ve\}\right)=\emptyset,
\]
for $\ve=0, \infty$, $L_0$ is admissible (i.e.~intersects each face
in the right codimension or not at all) and gives an element of
$\cN[1](1)$. Furthermore, the 
above intersection tells us that $\dN(L_0)=0$. Similarly, we obtain that $L_1$
gives an element of $\cN[1](1)$, and that $\dN(L_1)=0$. Thus
$L_0$ and $L_1$ yield well defined classes in $\HH^1(\cN[\bullet](1))$.

In order to show that they are non-trivial, we show that in the localization
sequence \eqref{locseq}, their images under the boundary map 
\[\HH^1(\Nc (1)) \st{\delta}{\lra} 
\HH^0(\Nc[\{0\}](0))\oplus  \HH^0(\Nc[\{1\}](0))
\] are non-zero. It is enough to treat the
case of $L_0$. 
Let $\ol{L_0}$ be the closure of $L_0$ in $\A^1\times \square^1$. 
Indeed, $\ol{L_0}$ is given by the parametrized cycle
\[
\ol{L_0}=[x; x] \subset \A^1\times \square^1,
\]
and the intersection with the face $u_1=0$ is of codimension $1$ in
$\A^1\times\{0\}$ and the intersection with $u_1=\infty$ is empty. Hence
$\ol{L_0}$ is admissible.
%%%%%%%%%%%%%%%

Thus, considering the definition of $\delta$, $\delta(L_0)$ is given by the
intersection of the differential 
of $\ol{L_0}$ with $\{0\}$ and $\{1\}$ on the first and second
factors respectively. The above discussion on the admissibility of $\ol{L_0}$ 
shows that
$\delta(L_0)$ is non-zero on the factor $\HH^0(\Nc[\{0\}](0)) $ and $0$ on the
other factor, as the admissibility condition is trivial in $\HH^0(\Nc[\{0\}](0))
$ and the restriction of $\ol{L_0}$ to $1$ is empty. The situation is reversed
for $L_1$. 
\end{proof}
Later we will consider cycles depending on many parameters, and denote by 
\[
[x;f_1(x ,\mb t), f_2(x,\mb t), \ldots , f_n(x,\mb t)] \subset X\times \square^n
\]
the (image under the projector $\Alt$  of the) restriction to $X \times \square^n$
of the image of 
\[
 \begin{tikzpicture}
\matrix (m) 
[matrix of math nodes,
 row sep=0.6em, column sep=8em, 
 text height=1.5ex, text depth=0.25ex,nodes={rectangle,
   minimum width=8em}] 
{ 
X \times (\p^1)^k & X \times (\p^1)^n \\
 (x, \mb t) &(x,f_1(x ,\mb t), f_2(x,\mb t), \ldots , f_n(x,\mb t)). \\
};
 \path[->,font=\scriptsize]
 (m-1-1) edge node[auto] {} (m-1-2);
 \path[|->,font=\scriptsize]
(m-2-1) edge node[auto] {} (m-2-2);
 \end{tikzpicture}
 \]  

\subsection{A weight $2$ example: the Totaro cycle}
Consider the linear combination
\[
b=L_{0}\cdot L_{1} \in \cN[2](2).
\]
It is given as a parametrized cycle by 
\[
b=[x; x, 1-x] \subset X \times
\square^2
\]
or in terms of defining equations by
\[
X_1V_1-U_1X_2= 0 \qquad \mbox{and} \qquad U_1V_2+U_2V_1=V_1V_2
\]
where $X_1$ and $X_2$ denote the homogeneous coordinates on $X=\ps$ and $U_i$, 
$V_i$ the homogeneous coordinates on each factor $\square^1=\p^1\sm \{1\}$ of $\square^2$.
The intersection of $b$ with any faces ($U_i$ or $V_i=0$ for 
any $i$'s) is empty, because $X_1$ and $X_2$ are different from $0$  in $X$ and
because $U_i$ is different from $V_i$ in $\square^1$. This comment ensures 
that $b$ is admissible. 

Moreover, it implies that $\dN(b)=0$. So $b$ gives a class
\[
[b] \in \HH^2 (\Nc(2)).
\]
Let us show that this class is trivial.

Let $\ol b$ denote the algebraic closure of $b$ in $\A^1\times \square^2$. 
 As before, the intersection with
$\A^1\times F$ for any face $F$ of $\square^2$ is empty, and $\ol b$ (after
applying the projector $\Alt$) gives 
\[
\ol b \in \mc N_{\A^1}^2(2).
\] 
Write $\dN_{\A^1}$ for the differential in $\mc N_{\A^1}$. Then
$\dN_{\A^1}(\ol b)=0$ and $\ol b$ defines a class
 \[
[\bar b] \in \HH^2 (\Nc[\A^1](2)).
\]
As Corollary \ref{H2A1} ensures that $\HH^2 (\Nc[\A^1](2))=\CH^2(\A^1,2)=0$,
there exists $\ol c \in \mc N_{\A^1}^1(2)$ such that 
\[
\dN_{\A^1}(\ol c)=\ol b.
\]

Moreover, note that $\ol b|_0=\ol b|_1=\emptyset$. The multiplication map

\[ \begin{tikzpicture}[baseline={([yshift=0.5em] current bounding box.south)}]
 \matrix (m) [matrix of math nodes,
 row sep=0.6em, column sep=2em, 
text height=1.5ex, text depth=0.25ex] 
 {\A^1 \times \square^1 \times \square^{2} 
& \A^1 \times  \square^{2}  \\
};
\path[->,font=\scriptsize]
(m-1-1) edge node[auto] {$\mu$} (m-1-2);
%(m-2-1) edge node[auto] {} (m-2-2)
%\path[|->,font=\scriptsize]
%(m-2-1) edge node[auto] {$$} (m-2-2);
\end{tikzpicture}\!,
\quad
\begin{tikzpicture}[baseline={([yshift=0.5em] current bounding box.south)}]
 \matrix (m) [matrix of math nodes,
 row sep=0.6em, column sep=2em, 
text height=1.5ex, text depth=0.25ex] 
 {
{[x ; u_1, u_2, u_{3}]}  &
{[\frac{x}{1-u_1} ; u_2, u_{3}]}   \\
 };
\path[|->,font=\scriptsize]
(m-1-1) edge node[auto] {$$} (m-1-2);
 \end{tikzpicture} \]
is flat.  Hence, we can consider the pull-back of the cycle $\ol
b$ by $\mu$. This pull-back is 
given explicitly (after reparametrization) by
\[
\mu^*(\ol b)=[x; 1 -\frac{x}{t_1} ,t_1,1-t_1 ] \subset \A^1\times \square^3.
\]

This is nothing but an $\A^1$-based variant of Totaro's cycle \cite{Totaro},
already described in 
\cite{BKMTM, BlochLie}, and it gives a well-defined element in $\mc N_{\A^1}^1(2)$.
\begin{defn} Let $L_{01}=\on{Li}^{cy}_2$ denote the cycle 
\[
L_{01}=[x; 1 -\frac{x}{t_1} ,t_1,1-t_1 ] \subset X \times \square^3
\]
in $\mc N_X^1(2)$.
\end{defn}

\begin{rem}
The cycle  $L_{01}$ corresponds to the function $x \mapsto \on{Li_2}(x)$, as shown in
\cite{BKMTM} or in \cite{GanGonLev05}.
\end{rem}
From the parametrized expression above, we obtain the following result. 
\begin{lem} The cycle $L_{01}$ satisfies the following properties
\begin{enumerate}
\item $\dN(L_{01})=b$.
\item $L_{01}$ extends to $\A^1$, i.e.~its closure $\ol{L_{01}}$ in $\A^1 \times \square^3$
  gives a well-defined element in  $\mc N_{\A^1}^1(2)$. 
\item $\ol{L_{01}}|_{x=0}=\emptyset$ and $\ol{L_{01}}|_{x=1}$ are well-defined.
\end{enumerate}
\end{lem}

\subsection{Polylogarithm cycles}
We construct the cycles $\on{Li}^{cy}_n=L_{0\cdots 01}$ for $n>2$ inductively. 
Define
$\on{Li}^{cy}_1$ to be equal to $L_1$.
\begin{lem}\label{lem:cycleLi_k}
For any integer $n \geqs 2$, there exist cycles  $\on{Li}^{cy}_{n}$ in
$\mc N_{X}^1(n)$ satisfying
\begin{enumerate}
\item $\dN(\on{Li}^{cy}_{n})=L_0\cdot \on{Li}^{cy}_{n-1}$
\item $\on{Li}^{cy}_{n}$ extends to $\A^1$, i.e.~its closure
  $\ol{\on{Li}^{cy}_{n}}$ in $\A^1 \times \square^{2n-1}$
  yields a well-defined element in  $\mc N_{\A^1}^1(n)$.
\item $\ol{\on{Li}^{cy}_{n}}|_{x=0}=\emptyset$ and $\ol{\on{Li}^{cy}_{n}}|_{x=1}$ are 
well-defined.   
\item $\on{Li}^{cy}_{n}$ is explicitly given as a parametrized cycle by
\[
[x; 1-\frac{x}{t_{n-1}}, t_{n-1},1-\frac{t_{n-1}}{t_{n-2}}, t_{n-2}, 
\ldots, 1-\frac{t_2}{t_{1}}, t_1, 1-t_1 ] 
\subset X \times \square^{2n-1}
\] 
\end{enumerate}
\end{lem}
\begin{proof}
For $n=2$, we already defined $\on{Li}^{cy}_2=L_{01}$ satisfying the
expected properties.

Assume now that we have constructed the cycles 
$\on{Li}^{cy}_k$ for $2\leqs k\leqs n$.
As before, let $b$ be the product 
\[
b=L_{0}\cdot \on{Li}^{cy}_{n-1}=
[x; x,  1-\frac{x}{t_{n-2}},
t_{n-2},1-\frac{t_{n-2}}{t_{n-3}}, t_{n-3},  
\ldots, 1-\frac{t_2}{t_{1}}, t_1, 1-t_1].
\]
and $\ol b$ its algebraic closure in $\A^1\times \square^{2n-2}$.

Computing the differential with the Leibniz rule, we obtain $\dN(b)=-L_{0}\cdot L_{0}\cdot
\on{Li}^{cy}_{n-2}=0$, and $b$ gives a class in $\HH^2(\Nc(n))$. 

Using its expression as a parametrized cycle, we compute the differential
of $\ol b$ in $\Nc[\A^1]$
\begin{equation}\label{eqdA1Lin}
\dN_{\A^1}(\ol b)=\sum_{i= 1}^{2n-2}(\dN_{\A^1, i}^0(\ol b)- \dN_{\A^1, i}^{\infty}(\ol b))=0,
\end{equation} 
since many terms are the empty cycle because intersecting with a face $u_i=0,\infty$ on a
factor $\square^1$ leads to a $1$ appearing on another $\square^1$, while the 
remaining terms cancel after applying the projector $\Alt$.

Just as it does in the case of $\on{Li}^{cy}_{2}$, $\ol b$ gives a class in $\HH^2
(\Nc[\A^1](2))=0$ by Corollary \ref{H2A1}, and there exists 
\[
\ol{\on{Li}^{cy}_{n-1}}=\ol c \in \cNg{\A^1}{1}
\]
 such that 
\[
\dN_{\A^1}(\ol c)=\ol b.
\] 
Just as $\ol{\on{Li}^{cy}_{n-1}}|_{x=0}=\emptyset$, we also have $\ol b|_{x=0}=\emptyset$, and the
element $\ol c$ is given by the pull-back by the multiplication 
\[
 \begin{tikzpicture}[baseline={([yshift=0.5em] current bounding box.south)}]
 \matrix (m) [matrix of math nodes,
 row sep=0.6em, column sep=2em, 
text height=1.5ex, text depth=0.25ex] 
 {\A^1 \times \square^1 \times \square^{2n-2} 
& \A^1 \times  \square^{2n-2}  \\
};
\path[->,font=\scriptsize]
(m-1-1) edge node[auto] {$\mu$} (m-1-2);
\end{tikzpicture} \!,
\]
given in coordinates by
\[
\begin{tikzpicture}[baseline={([yshift=0.5em] current bounding box.south)}]
 \matrix (m) [matrix of math nodes,
 row sep=0.6em, column sep=2em, 
text height=1.5ex, text depth=0.25ex] 
 {
{[x ; u_1,u_2,\ldots, u_{2n-1}]}  &
{[\frac{x}{1-u_1} ; u_2, \ldots, u_{2n-1}]}   \\
 };
\path[|->,font=\scriptsize]
(m-1-1) edge node[auto] {$ $} (m-1-2);
 \end{tikzpicture}. 
\]

Reparametrizing the factor $\A^1$ and the first $\square^1$ factor, we write $\ol
c=\mu^*(\ol b)$ explicitly as a parametrized cycle
\[
\ol{\on{Li}^{cy}_{n}}=\ol c=
[x; 1-\frac{x}{t_{n-1}},t_{n-1},1- \frac{t_{n-1}}{t_{n-2}}, t_{n-2},\ldots, 1
-\frac{t_2}{t_1} ,t_1,1-t_1 ] \subset \A^1 \times \square^{2n-1}.
\]

Now, let $\on{Li}^{cy}_{n}$ be the restriction of $\ol c$ to $\cN(n)$, i.e.~the
parametrized cycle
\[
\on{Li}^{cy}_{n}=
[x; 1-\frac{x}{t_{n-1}},t_{n-1},1- \frac{t_{n-1}}{t_{n-2}}, t_{n-2},\ldots, 1
-\frac{t_2}{t_1} ,t_1,1-t_1 ] \subset X \times \square^{2n-1}.
\]

The different properties, $d(\on{Li}^{cy}_n)=L_{0}\cdot \on{Li}^{cy}_{n-1}$,
extension to $\A^1$, $\on{Li}^{cy}_n|_{x=0}=\emptyset$, can now be derived easily
either with the explicit parametric representation or using the properties of
$\ol c$.
\end{proof}
\begin{rem}
In equation \eqref{eqdA1Lin}, the fact that $\dN_{\A^1,1}^0(\ol b)=0$ is related
to the induction hypothesis $\ol{\on{Li}^{cy}_{n-1}}|_{x=0}=\emptyset$; in terms of
cycles, the above part of the differential is given by 
\[
\ol b\cap \left(\A^1\times \{0\}\times \square^{2n-3} \right)=\on{Li}^{cy}_{n-1}|_{x=0}
.
\]

The other terms in the differential are related to the equation satisfied by
$\on{Li}^{cy}_{n-1}$ in $\cN(n-1)$, giving 
\[
\dN(b)=-L_{0}\cdot L_{0}\cdot
\on{Li}^{cy}_{n-2}=0.
\]
Even if $L_0$ is not defined in $\cNg{\A^1}{\bullet}$, the fact that
$\ol{\on{Li}^{cy}_{n-1}}|_{x=0}=0$ ensures that the product really corresponds to an
element in $\cNg{\A^1}{\bullet}$.
\end{rem}
\begin{rem}
\begin{itemize}
\item The expression of $\on{Li}^{cy}_n$ as a parametrized cycle was already
  given fiberwise in \cite{BKMTM} and in 
\cite{GanGonLev05}.
\item Moreover, $\on{Li}^{cy}_n$ corresponds to the function $z \mapsto
  Li_n^{\C}(z)$, as shown in 
\cite{BKMTM}.
\item The construction is given in full detail for more general cycles in
  Section \ref{treetocycle}, and is nothing but a direct application of Theorem
  \ref{thm:cycleLcLcu} to the word $0 \cdots 01$ (with $n-1$ zero). 

The case of cycles $\on{Li}^{cy}_n$ is, however, simple enough to be treated 
separately as the ``good'' case. 
\item It is a general fact that pulling back by the multiplication
preserves the property of having empty fiber at $x=0$, as proved in Proposition
\ref{multequi}. 
\end{itemize}
\end{rem}

\subsection{Admissibility problem at $x=1$ in weight $3$}
\label{subsec:L011pb}
It seems that the first attempt to define algebraic cycles
corresponding not only to polylogarithms but also to multiple polylogarithms
was made by Gangl, Goncharov and Levin in \cite{GanGonLev05}. In their work,
they succeeded in constructing cycles corresponding to the value
$\on{Li}_{n_1,\ldots, n_k}(z_1, \ldots, z_k)$ for fixed parameters $z_i$ in a
number field $F$, with the conditions $z_i \neq 1$ and $z_i\neq z_j$ for $i\neq
j$. However, their cycles are not admissible if one removes the conditions on
the $z_i$.  In this section, we consider the first example where this problem appears.

In the example of the $\on{Li}^{cy}_n$ cycles, we repeatedly
multiply by the cycle $L_0$ :
\[
\dN(\on{Li}^{cy}_{n})=L_0\cdot \on{Li}^{cy}_{n-1}.
\]
 Below we explain the first example where  a multiplication by the cycle $L_1$
 arises and where both geometric and combinatorial difficulties arise.

The cycle $L_{01}$ was defined previously, as was the cycle
$L_{001}=\on{Li}^{cy}_3$, by considering the product
\[
b=L_{0}\cdot L_{01}.
\]

Now we would also like to consider the product
\[
 b=L_{01}\cdot L_{1} \quad \in \cN[2](3),
 \]
given as a parametrized cycle by
\[
b=[x;1-\frac{x}{t_1}, t_1, 1-t_1, 1-x] \quad \subset X\times \square^4.
\]
From this expression, we see that $\dN(b)=0$, since $x\in X$ cannot be equal
to $1$.

Let $\ol b$ be the closure of the defining cycle of $b$ in $\A^1 \times \square^4$, i.e.
\[
\ol b=\left\{(x,1-\frac{x}{t_1}, t_1, 1-t_1, 1-x) \mbox{ such that } x \in \A^1, \,
t_1 \in \p^1 \right\}.
\]
Let $F$ be a face of $\square^4$, and let $u_i$ denote the coordinates on each
factor $\square^1$. Then $u_i \neq 1$. If $F$ is contained in a
hyperplane defined by the equation $u_2=\infty$ or $u_3=\infty$, then, since
$u_1 \neq 1$, we have
\[
\ol b \cap \A^1 \times F =\emptyset.
\]  
Similarly, the intersection of $\ol b$ with a face contained in
$\{u_4=\infty\}$ is empty, because $x \in \A^1$ is different from $\infty$. This remark
reduces the case where $F$ is contained in $\{u_1=\infty\}$ to the case 
where $F$ is contained in
$\{u_2=0\}$, which gives an empty intersection since $u_3 \neq 1$. By symmetry, the
intersection with an $F$ contained in $\{u_3=0\}$ is also empty.

In order to prove that $\ol b$ is admissible and gives an element in
$\cNg{\A^1}{2}$, it remains to check the (co)dimension condition on the three
remaining faces: the face defined by the equation $u_1=0$, the one defined by
the equation $u_4=0$ and the one defined by the equations $u_1=u_4=0$. 
The intersection of $\ol b$ with the face $\{u_1=u_4=0\}$ is empty since
$u_2\neq 1$. The intersection $\ol b$ with the face defined by the equation
by $u_1=0$ or $u_4=0$ is $1$-dimensional, so of codimension $3$ in
$\A^1\times F$.

\begin{rem}
Let $F_4^0$ denote the face of $\square^4$ defined by $u_4=0$. The intersection
of $b$ with  
$X \times F_4^0$ is empty since $x\neq 1$ in $X=\ps$.     
\end{rem}

From the above discussion we obtain a well-defined element 
in $\cNg{\A^1}{2}(3)$, which we again denote by $\ol b$. 
Since the intersection with $u_1=0$ is killed by the projector $\Alt$,
computing the differential in $\cNg{\A^1}{\bullet}$ gives
\[
\dN_{\A^1}(\ol b)=-\ol{L_{01}}|_{x=1}\neq 0
\]
and $\ol b$ does not give a class in $\HH^2(\Nc[\A^1])$.

In order to overcome this obstacle, we introduce the constant 
cycle $L_{01}(1)$ in $\cN(2)$ defined by
\[
\ds L_{01}(1)=[x;1-\frac{1}{t_1}, t_1, 1-t_1] 
\subset X \times \square^3.
\] 
The cycle $L_{01}(1)$ satisfies
\[
\forall a \in X \qquad L_{01}(1)|_{x=a}=L_{01}|_{x=1}.
\]
and extends to a well-defined cycle in $\cNg{\A^1}{1}(2)$.

Instead of considering the product $L_{01}\cdot L_{1}$, we consider
the linear combination of products
\begin{equation}\label{bL011}
b=(L_{01}-L_{01}(1))\cdot L_{1}=L_{01}\cdot L_{1}-L_{01}(1)\cdot L_{1} 
\quad  \in  \cN[2](3).
\end{equation}

As above, we check that $b$ extends to a well-defined element $\ol b$ in
$\cNg{\A^1}{2}(3)$. The correction by $-L_{01}(1)\cdot L_1$ ensures that   
\[ \dN(b)=0, \qquad 
\dN_{\A^1}(\ol b)=0, \qquad
\ol b|_{x=0}=\emptyset.
\]
Considering the previous computation of the pull-back by the multiplication $\mu
: \A^1 \times \square^1 \lra 
\A^1$, we define $L_{011}$ in $\Nc(3)$ as the parametrized cycle
\begin{multline} \label{L011defpa}
L_{011}=[x; 1-\frac{x}{t_2},1-\frac{t_2}{t_1},t_1,1-t_1,1-t_2]\\
+[x; 1-\frac{x}{t_2},1-t_2,1-\frac{1}{t_1},t_1,1-t_1] \qquad \subset X \times \square^5
\end{multline}
Since $x\neq 1$ in $X=\ps$, it is easy to check 
that $L_{011}$ is admissible on $ X
\times \square^5$ and gives a well-defined element in $\Nc(3)$. An explicit
computation also yields
\begin{equation}\label{Labb}
\dN(L_{011})=b= (L_{01}-L_{01}(1))\cdot L_{1}.
\end{equation}
\begin{rem}
However, we point out that
\begin{itemize}
\item The cycle $L_{011}$ is defined as parametrized cycle, not using the
  pull-back by the multiplication, which serves here as a support to ``guess''
  the parametrized expression.
\item The fiber at $x=1$ of the algebraic closure $\ol{L_{011}}$ of $L_{011}$ in
  $\A^1 \times \square^5$ is not admissible.
\item This non-admissibility problem was also encountered by Gangl, Goncharov and
  Levin in  \cite{GanGonLev05}.
\end{itemize}
\end{rem}

In section \ref{treetocycle}, we will explain how to obtain general cycles
admissible at $x=1$. The particular example of a cycle $\Lc_{011}$ related
to $L_{011}$ above will be detailed in the next subsection
\ref{Lcex3}.
\begin{rem} Even if $L_{011}$ is not admissible at $x=1$, we \emph{could} still go on
  looking for ``good'' linear combinations of products in $\cNg{X}{2}$ . In
  particular, in weight   $4$, we \emph{could} consider 
\begin{equation}\label{bLaabb}
b=L_{0}\cdot L_{011}
+L_{001}\cdot L_{1}
-L_{001}(1)\cdot L_{1}
+L_{01}\cdot L_{01}(1)
\end{equation}
and observe that 
\begin{itemize}
\item The terms $L_{0}\cdot L_{011} +L_{001}\cdot L_{1} $ correspond to the
  expression of the cobracket of the Lyndon word $0011$ in the Lie coalgebra
  which is the graded dual to the free Lie algebra $\Lie(X_0,X_1)$ (cf. equation
  \eqref{eq:exmdLieT0011}. This is  explained in
  generality at   Section \ref{subsec:triTdual}.  
\item When computing the differential $\dN(b)$, the term $\dN(L_{01}\cdot
  L_{01}(1))$ gives 
\[
\dN(L_{01}\cdot  L_{01}(1)) = L_0\cdot L_1 \cdot L_{01}(1).
\]
From the differential of $L_0\cdot L_{011}$ arises a term in  $L_0\cdot
L_{01}(1)\cdot L_{1}$. Hence
the differential of $L_{01}\cdot  L_{01}(1)$ cancels with part of the
differential of $L_0\cdot L_{011}$ and can be thought as  ``a propagation'' of the
correction introduced  for $L_{011}$.  
\item The term $-L_{001}(1)\cdot L_{1}$ is similar to the correction
  $-L_{01}(1)\cdot L_{1}$ introduced earlier for $L_{011}$ and ensures that
  $\dN_{\A^1}(\ol b)=0$.
\end{itemize}
The above remarks are the motivation for the introduction of the 
cobracket $\dc$ on the Lie coalgebra dual to the Lie algebra representing the
action of $Lie(X_0,X_1)$ by Ihara's special derivations. In particular it
was the the similarity between equation \eqref{bLaabb} above and equation 
\eqref{eq:dcyTw0011} that led the author to uncover
the relation between Ihara's action and the geometry of the above cycles. 
The goal of section \ref{sec:Combi} is to understand the global combinatorics
underlying the correction terms and to make  their relation to Ihara's
special derivations explicit. 

Note however that these combinatorics can not be applied directly. The reason is
that  the lack of admissibility of the fiber at $x=1$ of  $\ol{L_{011}}$ 
 ``propagates'' to higher weight, leading to closed
subvarieties which are not admissible anymore. In the next section we explain
how to overcome this issue.
\end{rem} 

\subsection{Twisted multiplication and admissibility in weight $3$}
\label{Lcex3}
All the examples of cycles given above are obtained by the general
construction, \emph{with the exception} of $L_{011}$. Namely 
\begin{equation*}
L_0=\Lcz, \quad L_1=\Lco, \quad L_{01}=\Lc_{01}, \quad 
\on{Li}^{cy}_n=L_{W_n}
=\Lc_{W_n},
\end{equation*}
with $W_n=\underbrace{0 \cdots 0}_{n-1 \mbox{ times}}1$.

In the previous section, the cycle $L_{011}$ was constructed by considering the
product 
\[
b=(L_{01}-L_{01}(1))\cdot L_1 \qquad \in \cNg{X}{2}(3).
\]

However, the fiber at $1$ of $\ol{L_{011}}$ is not an admissible
cycle. This comes from the lack of admissibility of the fiber at $1$ of the
cycle $\ol{L_1}$; the closure in $\A^1\times \square^1$ of $L_1$.

When building the cycles $\on{Li}^{cy}_n$, the lack of admissibility of the fiber at
$x=0$ of $\ol{L_0}$ is counterbalanced by an empty fiber at $x=0$ of
$\ol{\on{Li}^{cy}_{n-1}}$. This phenomenon allowed the induction to go through
in the construction of the cycles $\on{Li}^{cy}_n$.

Thus, we want replace  the factor $L_{01}-L_{01}(1)$ in the expression of $b$
by a cycle $\Lcu_{01}$ whose closure $\ol{\Lcu_{01}}$ in $\A^1\times \square^3$
is admissible with an empty fiber at $x=1$. We define a twisted multiplication
map $\nu$ as follows:

\[ \begin{tikzpicture}[baseline={([yshift=0.5em] current bounding box.south)}]
 \matrix (m) [matrix of math nodes,
 row sep=0.6em, column sep=2em, 
text height=1.5ex, text depth=0.25ex] 
 {\A^1 \times \square^1 \times \square^{2} 
& \A^1 \times  \square^{2}  \\
};
\path[->,font=\scriptsize]
(m-1-1) edge node[auto] {$\nu$} (m-1-2);
\end{tikzpicture}\!,
\quad
\begin{tikzpicture}[baseline={([yshift=0.5em] current bounding box.south)}]
 \matrix (m) [matrix of math nodes,
 row sep=0.6em, column sep=2em, 
text height=1.5ex, text depth=0.25ex] 
 {
{[x ; u_1, u_2, u_{3}]}  &
{[\frac{x-u_1}{1-u_1} ; u_2, u_{3}]}   \\
 };
\path[|->,font=\scriptsize]
(m-1-1) edge node[auto] {$$} (m-1-2);
 \end{tikzpicture} \]
which exchanges the role of $0$ and $1$ with respect to the multiplication
$\mu$. As in the case of $L_{01}$ we obtain
\[
\ol{\Lcu_{01}}=\nu^*(\ol{L_0\cdot L_1}) \quad \subset \A^1\times \square^3,
\]
and $\Lcu_{01}$ is defined as its restriction to $X\times \square^3$. The cycle
$\ol{\Lcu_{01}}$ can also be described as a parametrized cycle by 
\[
\ol{\Lcu_{01}}=[x; \frac{t_1-x}{t_1-1} ,t_1,1-t_1 ] \subset \A^1 \times \square^3.
\]
The cycle $\ol{\Lcu_{01}}$ is a well-defined element of $\cNg{\A^1}{1}(2)$ and its fiber
at $x$ is empty. $\Lcu_{01}$ is an element of  $\cNg{X}{1}(2)$ and a direct
computation shows that 
\[
\Lcu_{01}=L_{01}-L_{01}(1)+\dN\left(
[x; s, \frac{
s - \frac{t_1 -x}{t_1}
}{
s -\frac{t_1-1}{t_1} }
, t_1 , 1 - t_1 ]
 \right)
\]

Then the product 
\[
b'=\Lcu_{01}L_1=[x; \frac{t_1-x}{t_1-1} ,t_1,1-t_1;1-x]
\]
can be extended to an admissible cycle $\ol{b'}$ in $\A^1\times
\square^4$ with empty fiber at $x=0$ and zero under $\da[\A^1]$. Its pull-back
by the multiplication gives 
\[
\ol{\Lc_{011}}=\mu^{*}(\ol{b'})=
[x,1-\frac{x}{t_2},\frac{t_1-t_2}{t_1-1}, t_1,1-t_1,1-t_2] \in \cNg{\A^1}{1}(3),
\]
which has an empty fiber at $0$. By proposition \ref{multhomo} (or direct
computation), it satisfies
\[
\da[\A^1](\ol{\Lc_{011}})=\ol{b'}=\ol{\Lcu_{01}L_1}.
\] 
By definition $\Lc_{011}$ is its restriction to $X\times \square^5$ and satisfies
\[
\dN(\Lc_{011})={\Lcu_{01}L_1}.
\]

This example shows that in order to construct a well-defined family of cycles whose
extension to $\A^1$ admits a well-defined fiber at $x=1$ (as elements of
$\cNg{\{1\}}{1}$) and an empty fiber at $0$, we need to simultaneously construct
a family of cycles satisfying the same properties but with $0$ and $1$ exchanged.

\section{Combinatorial settings}\label{sec:Combi}
A \emph{plane} or  \emph{planar} tree is a finite tree whose
internal vertices are of valency $\geqs 3$, and on which a cyclic
ordering is given on the edges coming out of each vertex. All other vertices
are of valency $1$; we call them \emph{external vertices}. 

A \emph{rooted} tree
is a planar tree as above with one distinguished external vertex of valency $1$,
called its root. In particular a rooted tree has at least one edge. The external
vertices which are not the root are called \emph{leaves}.

We will draw trees so that the root vertex is at the top and so that the cyclic
order around the vertices is counterclockwise.

The following combinatorial section is  organized as follows. Subsection
\ref{subsec:LieXY-Lyndon} reviews some properties of the free Lie algebra on
two generators $\Lie(X_0,X_1)$. In particular it presents the basis of Lyndon
brackets and a presentation of the brackets as trivalent trees. Subsection
\ref{subsec:Iharacobr} introduces Ihara's special derivations \cite{IharaAPBG,
  IharaSDABG} and their actions on $\Lie(X_0,X_1)$. Using the tree
representation for $\Lie(X_0,X_1)$,  we keep track of the part of Ihara's bracket
(or Poisson bracket) on  $\Lie(X_0,X_1)$ which comes from the special
derivations. This is done by taking the semi-direct sum of $\Lie(X_0,X_1)$ by
itself.
 
Section \ref{subsec:triTdual} dualizes the above situation. Ihara's coaction is
then written down in terms of the basis dual to the Lyndon brackets. The
structure coefficients of this coaction give us the ``system of differential
equations'' after a change of basis.   

As a last comment to this introduction to this combinatorial part, we observe
that the usages of trivalent tree is not actually necessary. However, 
the presentation with trees is somehow more visual, and more importantly, 
it sheds a new light on the relation between our construction and the 
work of Gangl, Goncharov and Levin \cite{GanGonLev05}.
\subsection{Lyndon words and the free Lie algebra $\Lie(X_0,X_1)$}
\label{subsec:LieXY-Lyndon}
The material developed in this section is detailed in full generality in 
\cite{ReuFLA93,ReuFLAH} and recalls the basic definitions and some properties
of the free Lie algebra on
two generators 
and its relations to trivalent trees and Lyndon words.
\subsubsection{Trees and free Lie algebra}
Recall that a \emph{Lie algebra} over $\Q$ is a $\Q$ vector space $L$, equipped
with a bilinear mapping $[\, ,\,  ] : L \otimes L \lra L$, satisfying the two
following 
properties for any $x,y,z$ in $L$:
\begin{align}
& [x,x]=0 \\
& [[x,y],z]+[[y,z],x]+[[z,x],y]=0. \tag{Jacobi}
\end{align}
\begin{rem} Note that applying the first relation to $[x+y,x+y]$ yields
  the   \emph{antisymmetry} relation
\[
[x,y]=-[y,x].
\]
Thus, we may rewrite the Jacobi identity as 
\[
[[x,y],z]=[x,[y,z]]+[[x,z],y].
\]
\end{rem}
\begin{defn} Given a set $S$, a free Lie algebra on $S$ over $\Q$ is a Lie
  algebra $L$ over $\Q$ together with a mapping $i : S \ra L$ with the following
  \emph{universal property}:

For each Lie algebra $K$ and each mapping $f : S \ra K$, $f$ factors uniquely
through $L$. 
\end{defn}
In what follows, we will only consider $S$ to be a set with two elements, either
$S=\{0,1\}$ or $S=\{X_0,X_1\}$. 

It is usual to consider the free Lie algebra on
$\{X_0,X_1\}$ as a subspace of $\Q<X_0,X_1>$ (its enveloping algebra), the
space of polynomials in two 
non commuting variables $X_0$ and $X_1$. Let $\on{Lie}(X_0,X_1)$ denote this free
Lie algebra.

In order to show the existence of free Lie algebras, a tree representation
is often used.
\begin{defn} Let $\Ttq$ denote the $\Q$ vector space generated by the set $\Ttr$ of
  rooted, planar, trivalent trees with leaves decorated (i.e. labeled) 
by $0$'s and $1$'s.

For two trees $T_1$, $T_2$ in $\Ttr$, define $T_1 \prac T_2$ to be the tree
obtained by joining the root (marked by a circle around the vertex) of $T_1$ and
$T_2$ and adding a new root: 
\[
\roota[1.5]{T_1} \prac \roota[1.5]{T_2}:=\rootaa{T_1}{T_2}
\]
\end{defn}

The internal law $\prac$ is standard in the the study of binary operations and
is usually called grafting. 

$\Ttr$ is isomorphic to the free magma on $\{0,1\}$; a branch $\prac$
in a tree corresponds to a bracketing in a well-formed expression.

The composition law $\prac$ extends by bilinearity to $\Ttq$ making it into a
ring. Let $I_{Jac}$ 
denote the ideal of $\Ttq$ generated by elements of the form 

\[
T \prac T 
\quad \mx{and} \quad
(T_1 \prac T_2)\prac T_3 + (T_2 \prac T_3)\prac T_1 +(T_3 \prac T_1)\prac T_2.
\]
The quotient $\Ttq / I_{Jac}$ is a Lie algebra with bracket $[\, , \,]$ given by
$\prac$; in fact it is the free Lie algebra on $\{0,1\}$. 
Identifying $\{0,1\}$ with $\{X_0,X_1\}$ by the obvious morphism and using the
correspondence $\prac \leftrightarrow [\, , \,]$, we obtain
\begin{lem} The quotient $\TL=\Ttq/ I_{Jac}$ is isomorphic to $\Lie(X_0,X_1)$.
\end{lem}
For $T$ in $\Ttr$, let $[T]$ denote its image in $\TL$.
%%%%%%%%%%%%%%%%%%%%%%%%%%%%%%%%%%%%%%%%%%%%%%%%%%%%%%%%%%
%%%%%%%%%%%%%%%%%%%%%%%%%%%%%%%%%%%%%%%%%%%%%%%%%%%%%%%%%%%%%%
%%%%%%%%%%%%%%%%%%%%%%%%%%%%%%%%%%%%%%%%%%%%%%%%%%%%%%%%%%%%%%
\subsubsection{Lyndon words}
In this section we will recall a particular basis of the vector space $\TL$, 
the one induced by the Lyndon words.

Let $S$ be the set $\{0,1\}$, and let $S^*$ denote the set of finite words in the
letters $0,1$. Let $<$ be the lexicographic order on $S^*$ with $0<1$. 
\begin{defn}[Lyndon words]
A \emph{Lyndon word} $W$ in $S^*$ is a  nonempty word which is smaller than all 
its non-trivial proper right factors, i.e. $W \neq \emptyset$ and 
\[
W=U V \mbox{ with } U,V \neq \emptyset \quad \Rightarrow \quad W<V.
\]
\end{defn}
Note that $0$ and $1$ are Lyndon words by convention.

\begin{exm} The Lyndon words of length $\leqs 4$ are
\[
0, \, 1 ,\, 01,\, 001,\, 011,\, 0001,\, 0011,\, 0111.
\] 
They are ordered by the lexicographic order, which gives
\[
0<0001<001<0011 <01<011<0111<1.
\] 
\end{exm}

In order to associate a tree to a Lyndon word, we need the following definition.
\begin{defn}[Standard factorization] Let $W$ be a word in $S^*$ of length $\geqs
  2$.  The standard factorization of $W$ is the decomposition 
\[W=UV \mbox{ with } \left\{
\begin{array}{l}U,V \in S^*\sm \emptyset  \\
\mbox{ and } V 
\mbox{ is the smallest
  non-trivial proper right factor of }W.  
\end{array}
\right.
\] 
\end{defn}

To any Lyndon word $W$ we associate a tree $\tw W$ in $\Ttr$. If $W=0$ or $W=1$,we set 
\[
\tw{0}=\roota{0} \qquad \tw 1 = \roota 1.
\]
For a Lyndon word $W$ of length $\geqs 2$, let $W=UV$ be its standard
factorization and set
\[
\tw W=\tw U \prac \tw V.
\]

Let $H_L$ be the set $\{\tw W \}$ where $W$ runs through the Lyndon
  words in the letters $0,1$. 
\begin{rem}\label{orderHL}
The set  $H_L$ is endowed with the total order $<$ induced by
  the ordering of the Lyndon words $W$ given by the lexicographic order on
  $S^*$. 
\end{rem} 
\begin{defn}Let $Lyn$ be the set of the Lyndon words.
For any Lyndon word $W$, let $[\tau_W]$, or simply $[W]$, be the image of $\tw
W$ in $\TL$.  

We say  that $\tw W$ is a 
\emph{Lyndon tree} and that $[W]$ is a \emph{Lyndon bracket}.
\end{defn}

\begin{thm}[{\cite{ReuFLA93}[Theorem 5.1]}]
The family $([W])_{W \in Lyn}=([\tau_W])_{W \in Lyn}=$   forms a basis of $\TL$.  
\end{thm}

\begin{exm}In length $\leqs 3$, the Lyndon trees are given by:
\[
\tw{0}=\roota{0} ,\quad \tw 1 = \roota 1 ,\quad \tw{01}=\TLxy{} ,\quad
\tw{001}=\TLxxy{},\quad 
\tw{011}=\TLxyy{},
\]
and in length $4$ by
\[
\tw{0001}=\TLxxxy{}
,\quad 
\tw{0011}=\TLxxyya{}
,\quad
\tw{0111}=\TLxyyy{}.
\]
\end{exm}

Moreover, a basis of $\TL \w \TL$ is then given by the family $([U]\w [V])$
for $U, V$ Lyndon words such that $U<V$. Writing the Lie bracket  in
this basis yields the structure coefficients of $\TL=\Lie(X_0,X_1)$.

\begin{defn}The \emph{structure coefficients} $\alpha_{U,V}^W$ of
  $\TL=\Lie(X_0,X_1)$ are given for any Lyndon words $W$ and $U<V$ by the
  family of relations
\begin{equation}\label{brac}
[U]\w[V] \st{[\,,\,]}{\longmapsto} [[U],[V]]=\sum_{W \in
  Lyn}\alpha_{U,V}^W [W].
\end{equation}
The $\alpha_{U,V}^W $ are integers
\end{defn}

\subsection{Special derivation and Ihara's cobracket and coaction}
\label{subsec:Iharacobr}
We now review the Ihara bracket \cite{IharaAPBG, IharaSDABG}, denoted
$\{ \, , \, \}$, which provides $\Lie(X_0,X_1)$ with another Lie 
algebra structure.  We also explain how the Ihara bracket is represented 
in terms of trivalent trees.

A derivation of $\Lie(X_0,X_1)$ is a linear endomorphism $D$ of $\Lie(X_0,X_1)$
compatible with the bracket $[ \, , \, ]$ in the following way:
\[
D([f,g])=[D(f),g]+[f,D(g)] \qquad \forall\, f,g \in \Lie(X_0,X_1).
\] 
The commutator of two derivations $D$ and $D'$, given by 
\[
[D,D']_{Der}=D \circ D' -D' \circ D,
\]
places a Lie algebra structure on the set $Der(\Lie(X_0,X_1))$ of derivations.
\begin{defn}[Special derivations]\label{spederi}
For any $f$ in $\Lie(X_0,X_1)$ we define the \emph{special derivation} $D_f$ by
\begin{equation}\label{eq:defDf}
D_f(X_0)=0, \qquad D_f(X_1)=[X_1,f]
\end{equation}
\end{defn}

The map $ \Lie(X_0,X_1) \lra Der(\Lie(X_0,X_1))$ sending $f \mapsto D_f$ is
linear, and its image is a Lie subalgebra of $Der(\Lie(X_0,X_1))$ such that for
any $f$ and $g$ in $\Lie(X_0,X_1)$ we have 
\[
[D_f,D_g]_{Der}=D_h \qquad \mx{with } h=[f,g]+D_f(g)-D_g(f).
\]
\begin{defn}[{Ihara's bracket, \cite{IharaAPBG,
      IharaSDABG}}]\label{def:Iharabracket} The \emph{Ihara's bracket} 
$\{\, ,\,   \}$ on $\Lie(X_0,X_1)$ is defined by
\[
\{f,g\}=[f,g]+D_f(g)-D_g(f).
\]
\end{defn}
\begin{rem}\label{remIhara0and1}
Note that 
%even if Ihara's bracket is defined consistently on $\Lie(X_0,X_1)$, 
the derivation $D_{X_1}$ is identically zero. Moreover $D_{X_0}$ is the adjoint derivation 
\[
D_{X_0}(g)=\on{ad}_{X_0}(g)=[g,X_0];
\]
In particular, $\{X_0,g\}=[D_{X_0},D_g]_{Der}=0$.
\end{rem}
As we just saw, $\Lie(X_0,X_1)$ acts on itself by the non-inner derivations
$D_f$; that is $D_f$ is not an adjoint derivation (with the exception of
$D_{X_0}$ and $D_{X_1}$). The Ihara bracket controls the relation between the
usual bracket, the derivation bracket and the action. However it loses the
description of the action. Using the tree representation for $\Lie(X_0,X_1)$ 
allows us to track of the action. We begin by adding a root decoration to the trivalent
trees. Trees corresponding to an element in the Lie algebra $\Lie(X_0,X_1)$ have
a root decorated by a generic parameter $x$. Trees corresponding to derivation have
a root decorated by $1$ (as reminder that they act on $X_1$).

More formally, let $\Ttg$ (resp. $\Ttg[1]$) denote the $\Q$ vector space
generated by the set $\Ttrg$ (resp. $\Ttrg[1]$) of
  rooted, planar, trivalent trees with leaves decorated (i.e. labeled) 
by $0$'s and $1$'s and a root decorated by $x$ (resp. $1$).

The internal law $\prac$ is defined as above on each set $\Ttrg[x]$ and
$\Ttrg[1]$, i.e.~it joins the two trees and adds a new root redecorated by $x$
and $1$ respectively. It is then extended to the disjoint union 
\[
\{0\}\cup \Ttrg[1] \cup \Ttrg[x] 
\] 
by $0$ whenever the two trees do not have the same root decoration, and extended
by bilinearity to 
\[
\Ttg \oplus \Ttg[1].
\] 
The ideal $I_{Jac}$ is defined in the obvious way, separately on $\Ttg$ and on
$\Ttg[1]$, and we define the Lie algebras:
\[
\TL[x]=\Ttg[x]/I_{Jac}\qquad \mx{and}\qquad
\TL[1]=\Ttg[1]/I_{Jac}
\]
We identify $\TL[x]$ with the Lie algebra $\Lie(X_0,X_1)$.  Again, we will write
$[T(a)]$ for the image in $\TL[a]$ of the tree $T$ in $\Ttrg[a]$ for $a \in \{1, x\}$.

Similarly, a generic element of $\TL$ is denoted by $[F]$, while its image in 
$\TL[x]$ (resp. $\TL[1]$), i.e.~with root decorated by $x$ (resp. $1$), 
is denoted by $[F(x)]$ (resp.$[F(1)]$).

To any element $[F(1)]$  in $\TL[1]$ we associate a
derivation $D_{F(1)}$ on the direct sum 
\[
\TL[1;x]=\TL[1]\oplus \TL[x]
\]
via
\[
D_{F(1)}(\Big[\TLx{1}\Big])=D_{F(1)}(\Big[\TLx{x}\Big])=0, \qquad
D_{F(1)}(\Big[\TLy{1}\Big])= \Big[\TLy{1}\Big] \prac  [F(1)] \]%\\
and
\[
D_{F(1)}(\Big[\TLy{x}\Big])=\Big[\TLy{x}\Big] \prac [ F(x)]
=\Big[\TLy{x}\Big] \prac \Big[
\begin{tikzpicture}[baseline={([yshift=-0.5ex]current bounding box.center)},scale=0.6]
\tikzstyle{every child node}=[minimum size=0pt, inner sep=0pt]%[intvertex]
\node[roots](root) {}
%[deftree]
[level distance=2.0em,sibling distance=3ex]
child {node[mathsc,leaf](1){F} 
}
;%%%
\fill (root.center) circle (1/0.6*\lbullet) ;
\node[mathsc, xshift=-1ex] at (root.west) {x};
%\node[labf] at (1.south){F};
\end{tikzpicture} %
\Big]
=\left[
\begin{tikzpicture}[baseline=(current bounding box.center),scale=1]
\tikzstyle{every child node}=[mathscript mode,minimum size=0pt, inner sep=0pt]
%[intvertex]%[fill,circle,minimum size=3pt, inner sep=0pt]
\node[roots](root) {}
[deftree]
%[mathscript mode]%[fill,circle,minimum size=3pt, inner sep=0pt]
%\node[draw,circle,inner sep=0.5pt](root) 
%{$\bullet$}
%[level distance=1.5em,sibling distance=3ex]
child {node%[intvertex]%[fill, circle, minimum size=2pt,inner sep=0pt]
[fill,circle,minimum size=0.6ex, inner sep=0pt]
           {}[level distance=1.5em]
  child{ node[fill,circle,minimum size=0.6ex, inner sep=0pt](1){}}
  child{ node[leaf]{F}}
};
\fill (root.center) circle (1/1*\lbullet) ;
\node[mathsc, xshift=-1ex] at (root.west) {x};
\node[labf] at (1.south){1};
\end{tikzpicture}\right]
\]
%\end{multline*}
where in the above equation, the tree $F(x)$ has been pictured as $\Big[
\begin{tikzpicture}[baseline={([yshift=-0.5ex]current bounding box.center)},scale=0.6]
\tikzstyle{every child node}=[minimum size=0pt, inner sep=0pt]%[intvertex]
\node[roots](root) {}
%[deftree]
[level distance=2.0em,sibling distance=3ex]
child {node[mathsc,leaf](1){F} 
}
;%%%
\fill (root.center) circle (1/0.6*\lbullet) ;
\node[mathsc, xshift=-1ex] at (root.west) {x};
%\node[labf] at (1.south){F};
\end{tikzpicture} %
\Big]$. The last picture means that the two trees $\TLy{x}$ and $F(x)$ are 
joined at the root and a new root is added with root decoration $x$.  The above operation is linear in $F$.

For an element $[F(x)]$ of $\TL[x]$, we define   the derivation $D_{F(x)}$ on
$\TL[1,x]$  to be $0$. 
With the above definition, $\TL[1]$,
identified with $\Lie(X_0,X_1)$, acts on
$\TL[1]$ and on $\TL[x]$ by special derivations while $\TL[x]$ acts by $0$.

Now, we endow $\TL[1;x]$ with the Lie algebra structure of the semi-direct sum (see
\cite{GOVFLTLTG})  of
$\TL[x]\simeq \Lie(X_0,X_1)$ by $\TL[1]\simeq \Lie(X_0,X_1)$ acting by the above
derivations. More precisely, we define on $\TL[1;x]$ the bilinear map  
\[
\{[F(a)] ,[G(b)] \}=[F(a)]\prac [G(b)] + D_{F(a)}([G(b)])-D_{G(b)}([F(a)])
\]
where $F(a)$ and $G(b)$ denote two generic elements of $\TL[1;x]$, i.e.~$a$
and $b$ lie in $\{1,x\}$.

\begin{lem}The direct sum $\TL[1;x]=\TL[1]\oplus \TL[x]$ endowed with 
$\{\, , \, \}$ is a Lie algebra.
\end{lem}
\begin{proof} The proof takes place in $\TL[1;x]$ and its
  subspaces $\TL[x]$ and $\TL[1]$. Thus we simply write $F(a)$ to denote the
  element $[F(a)]$ in $\TL[1;x]$. It is enough to check that $\{\, , \, \}$
  satisfies Jacobi identity. The definition of 
\[
\{F(a) ,G(b) \}
\]
ensures that $\{\, , \, \}$ is the usual bracket $\prac$ on $\TL[x]$, when 
  $a=b=x$. When
  $a=b=1$, then $\{\, , \, \}$ is the Ihara bracket on $\TL[1]\simeq
  \Lie(X_0,X_1)$. Hence the Jacobi identity holds when the three terms are all in
  $\TL[x]$ or all in $\TL[1]$. When $a=x$ and $b=1$, the bracket $\{F(a),G(b)\}$
  reduces to $-D_{G(1)}(F(x))$.   We have to show that 
\[
\{F(a),\{G(b),H(c)\}\}+\{G(b),\{H(c),F(a)\}\}+\{H(c),\{F(a),G(b)\}\}=0
\]
when two out of the three elements $a,b,c \in \{1,x\}$ are equal and the other
is different. We can assume that $a=b$. When $a=b=x$ the Jacobi identity reduces to
the fact that $H(1)$ act as a derivation on $\{F(x),G(x)\}$. When $a=b=1$,
the Jacobi identity reduces to the definition of the bracket of two derivations.
\end{proof}

Note that a basis of $\TL[1;x]\w \TL[1;x]$ is given by the union of the 
following families:
\begin{align*}
[\tau_U(x)]\w[\tau_V(x)] & \mx{ for any Lyndon word } U<V\\
[\tau_U(x)]\w[\tau_V(1)]  & \mx{ for any Lyndon word } U\neq V\\
[\tau_U(1)]\w[\tau_V(1)] & \mx{ for any Lyndon word } U<V. 
\end{align*}

\begin{defn}\label{def:brac1t}The structure coefficients $\alpha_{U,V}^W$,
  $\beta_{U,V}^W$ and $\gamma_{U,V}^W$ of
  $\TL[1;x]$ are given for any Lyndon words $W$ by the
  family of relations
\begin{equation}
\begin{aligned}\label{brac1t}
\{[\tau_{U}(x)],[\tau_{V}(x)]\}&=\sum_{W \in
  Lyn}\alpha_{U,V}^W [\tau_{W}(x)]. & \mx{for any Lyndon word } U<V\\
\{[\tau_{U}(x)],[\tau_{V}(1)]\}&=\sum_{W \in
  Lyn}\beta_{U,V}^W [\tau_{W}(x)]  & \mx{for any Lyndon word } U\neq V\\
\{[\tau_{U}(1)],[\tau_{V}(1)]\}&=\sum_{W \in
  Lyn}\gamma_{U,V}^W [\tau_{W}(1)] & \mx{for any Lyndon word } U<V. 
\end{aligned}
\end{equation}
Note that the $\alpha_{U,V}^W $ are the $\alpha$'s of equation \eqref{brac}
because $\{\, , \, \}$ restricted to $\TL[x]$ is the usual Lie bracket. All
coefficients above are integers.
\end{defn}
We will need the following property of the above coefficients later,
for our geometric application. 
\begin{lem}\label{relV01} 
Let $W$ be a Lyndon word of length greater than or equal to
  $2$.
Then the following holds for any Lyndon words $U,V$ :
\begin{itemize}
\item $\beta_{0,V}^W=0$,
\item $\beta_{V,0}^W=\alpha_{0,V}^W$
\item $\beta_{U,1}=0$,
\item $\beta_{1,U}^W=\alpha_{U,1}^W$.
\item $\gamma_{U,V}^W=\alpha_{U,V}^W+\beta_{U,V}^W-\beta_{V,U}^W$.
\end{itemize}
In particular, $\beta_{0,0}^W=\beta_{1,1}^W=0$. We also have 
\[
\alpha_{U,V}^{\ve}=\beta_{U,V}^{\ve}=\gamma_{U,V}^{\ve}=0
\]
for $\ve \in\{0,1\}$.
\end{lem}
\begin{proof}
The coefficient $\beta_{0,V}^W$ arises by decomposing the bracket
\[
\{[\tau_{0}(x)],[\tau_{V}(1)] \}=-D_{\tau_V(1)}(\big[
\TLx{x}\big])=0
\]
into the basis. Hence $\beta_{0,V}^W=0$. Similarly, $\beta_{1,U}^W$ arises from 
\[
-D_{\tau_U(1)}(\Big[\TLy{x}\Big])=-\Big[\TLy{x}\Big]\prac [\tau_U(x)]=
[\tau_U(x)] \prac \Big[\TLy{x}\Big]
\]
which shows that  $\beta_{1,U}^W=\alpha_{U,1}^W$.
In the same way, $\beta_{U,1}=0$ arises from $-D_{\tau_1(1)}(\big[
\tau_U(x)\big])=0$ because $D_{\tau_1(1)}=D_{X_1}=0$ after identifying $\TL[1]$
and $\Lie(X_0,X_1)$. The same identification shows that 
\[
-D_{\tau_0}=-ad_{\tau_0} : [F] \longmapsto -[F] \prac [\tau_0]
\]
which proves $\beta_{V,0}^W=\alpha_{0,V}^W$. The relation 
\[
\gamma_{U,V}^W=\alpha_{U,V}^W+\beta_{U,V}^W-\beta_{V,U}^W
\] 
comes from the relation between the commutator of two  derivations, the action
and the usual bracket.  It is given in terms of the Lie algebra
$\Lie(X_0,X_1)$ by: 
\[
\{D_f,D_g\}=D_h \qquad \mx{with } h=[f,g]+D_f(g)-D_g(f)
\]
for any $f,g$ in $\Lie(X_0,X_1)$.
\end{proof}

\subsection{Trivalent trees and duality}\label{subsec:triTdual}
\newcommand{\sca}{1.5}
\newcommand{\scont}{1.9}
\renewcommand{\scont}{1.9}
Here, we dualize equation \eqref{brac1t} by considering the vector space that is
the graded dual of $\TL[1;x]$. The dual $\dc$ of the bracket $\{ \, , \,\}$ coming from the
semi-direct sum is a cobracket, i.e.~essentially a differential where the
relation $d^2=0$ is dual to the Jacobi identity. Hence the coefficients
$\alpha$'s and $\beta$'s from equation \eqref{brac1t} give us a ``differential
system'' (cf. equation \eqref{eq:dcyTw}) which will lead us, after a change of
basis,  to the differential system satisfied by our algebraic cycle.

We begin by making the construction of the graded dual of $\TL[1;x]$
explicit.  
\begin{rem}
It is not actually necessary to give an explicit construction of the dual 
vector space of $\TL[1;x]$, with an explicit basis dual to one above. 
However, this explicit construction
allows us to work with concrete objects. Moreover, we use it to
relate our work to the combinatorial construction of Gangl
Goncharov and Levin in \cite{GanGonLev05}(see Remark
\ref{relationGGL}. 
\end{rem}

The construction of the  vector spaces dual to $\TL[x]$ and $\TL[1]$ are
parallel. Hence $a$ will denote an element of $\{1,x\}$. 

Let $\Tts[a]$ be the quotient of $\Ttg[a]$ by the ideal (for $\prac$) $I_s$ generated by 
\[
T_1\prac T_2 + T_2 \prac T_1.
\] 
 Let $T$ be a tree in
$\Ttrg[a]$ with subtree $T_1 \prac T_2$, and let $T'$ be the tree $T$ in which 
$T_1 \prac T_2$ has been replaced by $T_2 \prac T_1$. The following relation
holds in $\Tts[a]$:
\[
T=-T'.
\]

The total order on $H_L$ (Remark \ref{orderHL}) induces a total order $<$ on
$\Ttrg[a]$. Let $\Bs[a]$ 
be the set of trees $T$ in $\Ttrg[a]$ such that 
\[
 T'=T_1 \prac T_2 \mbox{ is subtree of } T \quad \Rightarrow \quad T_1<T_2.
\]
Writing $T(a) \in \Bs[a]$ also for the image of a  tree $T(a)$ in $\Tts[a]$, we see that
\begin{lem} The set $\Bs[a]$ induces a basis of $\Tts[a]$, also denoted by $\Bs[a]$.
\end{lem}
From now on we identify $\Tts[a]$ with its dual, via the basis $\Bs[a]$.

Let $I_{Jac,a}^<$ denote the image of the ideal $I_{Jac}$ in $\Tts[a]$. The Lie
algebra $\TL[a]$ is then isomorphic to the quotient $\Tts[a]/I_{Jac,a}^{<}$ and, using the
identification between $\Tts[a]$ and its graded dual (the grading coming from
the number of leaves), we can
identify the graded dual of $\TL[a]$ with a subspace of $ \Tts[a]$.
\begin{defn} Let $\Tcl[a] \subset \Tts[a]$ denote the vector subspace of
  $\Tts[a]$ which is the graded dual of $\TL[a]\simeq \Lie(X_0,X_1)$.

Let $(\T W(a))_{W \in Lyn}$ in $\Tcl[a]$ denote the dual basis of
  the basis $([W])_{W \in Lyn}$ of the free Lie algebra $\TL[a]$.
\end{defn}

The $\T W(a)$ are linear combinations of trees in $\Bs[a]$. Observe that any
Lyndon tree $\tw W(a)$ is in $\Bs[a]$ and that by definition its coefficient in $\T W$
is $1$. 
\begin{exm}\label{exmTw}Up to length $\leqs 3$, $\T W(a)=\tw W(a)$, i.e.
\[
\T{0}(a)=\TLx{a} ,\quad \T{1} = \TLy a ,\quad \T{01}=\TLxy{a} ,\quad
\T{001}=\TLxxy{a},\quad 
\T{011}=\TLxyy{a}.
\]
In length $4$, the first linear combination appears:
\[
\T{0001}=\TLxxxy{a}
,\quad 
\T{0011}=\TLxxyy{a}
,\quad
\T{0111}=\TLxyyy{a}.
\]

\end{exm}
\begin{rem}
Using $\Tts[a]$ instead of $\Ttrg[a]$ to explicitly construct the graded dual of
$\TL[a]$ makes it possible so shrink the size of the linear combinations involved in the
basis dual to $([\tau_W])_{W \in Lyn}$. As an example, the linear combination of
trees corresponding to $\tau_{01}$ in the dual basis working in $\Ttrg[a]$ should be 
\[
\frac 1 2 \left( \TLxy{a} -
\begin{tikzpicture}[baseline=(current bounding box.center),scale=0.6]
\tikzstyle{every child node}=[intvertex]
\node[roots](root) {}
[deftree]
child {node{} 
    child{ node(1){}}
    child{ node(2){}} 
}
;%%%
\fill (root.center) circle (1/0.6*\lbullet) ;
\node[mathsc, xshift=-1ex] at (root.west) {a};
\node[labf] at (1.south){1};
\node[labf] at (2.south){0};
\end{tikzpicture} 
\right)
\] 
\end{rem}

As the Lie bracket on $\TL[a]$ (for the free Lie algebra structure) is induced
by $\prac : \Ttr[a] \w \Ttr[a] \ra \Ttr[a]$; it is also 
induced by $\prac $ on $\Tts[a]$. By duality, one obtains a cobracket
\[
\dL : \Tcl[a] \lra \Tcl[a] \w \Tcl[a]  
\]
dual to the Lie bracket and induced by the map $\Ttq[a] \lra \Ttq[a] \w \Ttq[a]$ also
denoted by $\dL$ : 
\begin{equation}\label{dLieTtq}
\dL : \rootaar[1]{T_1}{T_2}{a}\longmapsto 
\rootar[1.3]{T_1}{a} \w \rootar[1.3]{T_2}{a}.
\end{equation}
The property that $\dL\circ \dL =0$ on $\Tcl[a]$ is dual to the Jacobi identity on
$\TL[a]$.
\begin{exm} As example in weight $4$, we have
\begin{equation*}\label{eq:exmdLieT0011}
\dL(\T{0011}(a))=\T 0(a) \w \T{011}(a) + \T{001}(a) \w \T 1(a).
\end{equation*}
\end{exm}
\begin{prop}\label{dLTw} By duality, the following hold in $\Tcl[a]$:
\begin{itemize}
\item $\T 0(a) = \TLx a$, $\T 1(a) = \TLy a$;
\item $\dL(\T 0(a))=\dL(\T 1(a))=0$;
\item for all Lyndon words $W$ of length $\geqs 2$,  
\begin{equation}\label{eq:dLTw}
\dL(\T W(a))=\sum_{\substack{U<V\\U,V \in Lyn}}\alpha_{U,V}^W\T{U}(a) \w \T{V}(a)
\end{equation}
where the $\alpha_{U,V}^W$ are the structure coefficients of $\TL\simeq
\Lie(X_0,X_1)$ defined by 
equation \eqref{brac}.
\end{itemize} 
Moreover, we can construct the linear combinations $\T W(a)$ inductively by
\begin{equation}\label{eq:indconsTw}
\T W(a)=\sum_{\substack{U<V\\U,V \in Lyn}}\alpha_{U,V}^W\T{U}(a) \prac \T{V}(a)
\end{equation}
for $W$ of length greater than or equal to $2$. Here $\prac$ denotes the
bilinear map $\Tcl[a] \otimes \Tcl[a] \lra \Tcl[a]$  induced by $\prac$.
\end{prop}
Note that between equation \eqref{brac} and equation \eqref{eq:dLTw} the
summation is ``reversed'' due to the duality; equation \eqref{eq:dLTw} computes
the transpose of the matrix representation of the Lie bracket given in basis by
equation \eqref{brac}. Equation \eqref{eq:indconsTw} provides an inductive
constructions of trees $\T W(a)$. 

Having constructed the graded dual vector spaces $\Tcl[1]$ and $\Tcl[x]$ above, we give the following definition.
\begin{defn}Let $\Tcl[1;x]$ be the graded dual of $\TL[1;x]$; as a vector space it is
  the direct sum 
\[
\Tcl[1;x] = \Tcl[1] \oplus \Tcl[x].
\]
\end{defn}
A basis of $\Tcl[1;x]$ is given by the union of the two families:
\begin{align*}
\T W (x)& \mx{ for any Lyndon word } W \\ 
\T W (1) & \mx{ for any Lyndon word } W. \\
\end{align*}
Similarly a basis of $\Tcl[1;x] \w \Tcl[1;x]$ is given by the union of the
following families:
\begin{align*}
\T U(x)\w\T V(x) & \mx{ for any Lyndon word } U<V\\
\T U(x)\w\T V(1)  & \mx{ for any Lyndon word } U\neq V\\
\T U(1)\w\T V(1) & \mx{ for any Lyndon word } U<V. 
\end{align*}

\begin{prop} \label{dcyTw}The bracket $\{,\}$ on $\TL[1,x]$ induces a cobracket on
  $\Tcl[1;x]$
\[
\dc : \Tcl[1;x] \lra \Tcl[1;x] \w \Tcl[1;x] 
\]
which is written in terms of the above basis as
\begin{equation}\label{eq:dcyTw}
\dc(\T W(x))=\sum_{U<V} \alpha_{U,V}^W\T U(x) \w \T V(x) + 
\sum_{U,V} \beta_{U,V}^W\T U(x) \w \T V (1)
\end{equation}
and 
\begin{equation}\label{eq:dcyTw1}
\dc(\T W(1))=\sum_{U<V} \gamma_{U,V}^W\T U(1) \w \T V(1) 
\end{equation}
where $U$ and $V$ are Lyndon words, the $\alpha_{U,V}^W$, $\beta_{U,V}^W$ and
$\gamma_{U,V}^W$ being those defined in equation \eqref{brac1t}.

In particular $\dc^2=0$. 
\end{prop}
\begin{proof}The proposition follows by duality. In particular equations
  \eqref{eq:dcyTw} and \eqref{eq:dcyTw1} are just the transpose of
  \eqref{brac1t}.
\end{proof}
Let us give some examples of computation of $\dc(\T W(x))$.
\begin{exm}\begin{itemize}
\item For Lyndon words of length $1$ and $2$, we have
\[
\dc(T_{0}(x))=\dc(T_{1}(x))=0, \quad \mx{ and }
\dc(T_{01}(x))=\T 0 (x)\w \T 1(x)+ \T 1(x)\w \T 0(1). 
\]
\item For Lyndon words of weight $3$, we have
\[
\dc(T_{001}(x))=\T 0(x)\w \T{01}(x)+\T{01}(x)\w \T{0}(1)
\]
\begin{equation}\label{eq:dcyTW011}
\dc(T_{011}(x))=\T{01}(x)\w \T{1}(x)+\T{1}(x)\w \T{01}(1)
\end{equation}
\item In weight $4$, we find
 \begin{multline} \label{eq:dcyTw0011}
\dc(\T{0011}(x))=\T{0}(x)\w \T{011}(x)+ \T{011}(x)\w \T{0}(1)\\
+\T{001}(x)\w \T{1}(x)+
\T{1}(x)\w \T{001}(1)+\T{01}\w \T{01}(1)
\end{multline}
\end{itemize}
In particular, equations  \eqref{eq:dcyTW011} and \eqref{eq:dcyTw0011} should be
compared with equations \eqref{Labb} and \eqref{bLaabb}.
\end{exm}
 Note that the only difference between \eqref{eq:dcyTw0011} and \eqref{bLaabb}
 lies in the terms with a 
factor of $\T{0}(1)$. This difference also appears when comparing $\dc(\T{001})$
above to $\dN(L_{001})=\dN(\on{Li}^{cy}_3)$ presented at Lemma
\ref{lem:cycleLi_k}. We could simply kill these terms in the expression $\dc(\T
W)$ by taking the appropriate quotient. However, Lemma \ref{relV01} ensures that
terms of the form $\T W(x) \w \T 0 (1)$ can always be regrouped with a unique
term $\T 0(x) \w \T W(x)$ giving a term in 
\[
(\T 0(x) - \T 0(1))\w \T W(x). 
\]
Hence, terms with a factor $\T 0(1)$ do not really change the combinatorial
situation.

However, as presented in section \ref{subsec:L011pb}, the geometric situation 
does not exactly fit the above combinatorial setting, which needs to be rewritten in a
suitable way. Before doing so, we would like to comment on the relation between
equation \eqref{eq:dcyTw} and the combinatorial approach of Gangl, Goncharov and
Levin in \cite{GanGonLev05}.
\begin{rem}\label{relationGGL} In \cite{GanGonLev05}, the authors  constructed
  parametrized algebraic cycles in $\cNg{\Spec(\Q)}{1}$ from linear 
  combinations of trivalent trees  and a \emph{forest cycling map}. They worked
  in cdga setting where:  
\begin{itemize}
\item the product is induced by the disjoint union of trees (hence the name ``forest'');
\item the graded commutativity is induced by an ordering of the edges of the
  trees and of the forests and an alternating relation;
\item the differential $\dggl$ consists in the appropriate alternating sums of the
  following operation : (a) contracting internal edges of trees and (b)
  contracting and splitting root and external edges as pictured below (figures
  \ref{rootcont} and \ref{leafcont} : 
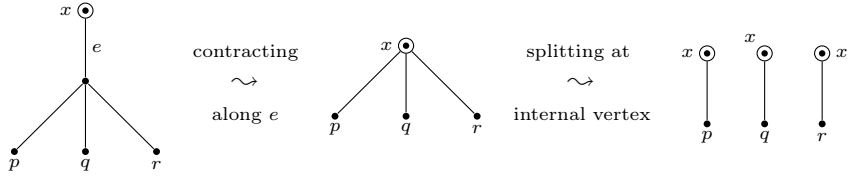
\begin{figure}[!htb]
\begin{tikzpicture}[baseline=(current bounding box.center),scale=\scont]
\tikzstyle{every child node}=[intvertex]%[fill,circle,minimum size=3pt, inner sep=0pt]
\node[roottest=2](root) {}
[deftree]
child {node{}
  child{node(1){}}
  child{node(2){}}
  child{node(3){}}
  edge from parent node[xshift=1ex, mathsc]{e}
};
\node[mathsc, xshift=-1ex] at (root.west) {x};
\node[labf] at (1.south) {p};
\node[labf] at (2.south) {q};
\node[labf] at (3.south) {r};
%\node[labfs] at (1) {0}; 
\fill (root.center) circle (1/\scont*\lbullet) ;
\end{tikzpicture} 
$\begin{array}{c}
\scriptstyle \textrm{contracting}\\ \leadsto \\\scriptstyle   \textrm{along }e\end{array}$
\begin{tikzpicture}[baseline=(current bounding box.center),scale=\scont]
\tikzstyle{every child node}=[intvertex]%[fill,circle,minimum size=3pt, inner sep=0pt]
\node[roottest=2](root) {}
[deftree]
%child {node{}
  child{node(1){}}
  child{node(2){}}
  child{node(3){}} ;
%};
\node[mathsc, xshift=-1ex] at (root.west) {x};
\node[labf] at (1.south) {p};
\node[labf] at (2.south) {q};
\node[labf] at (3.south) {r};
%\node[labfs] at (1) {0}; 
\fill (root.center) circle (1/\scont*\lbullet) ;
\end{tikzpicture} 
$\begin{array}{c}\scriptstyle{\textrm{splitting at }}\\
{\leadsto}\\
\scriptstyle \textrm{internal vertex}\end{array}$
\begin{tikzpicture}[baseline=(current bounding box.center),scale=\scont]
\tikzstyle{every child node}=[intvertex]%[fill,circle,minimum size=3pt, inner sep=0pt]
\node(ancre){};
\node[roottest=2,xshift=-5ex] (root) at (ancre){}
[deftree]
  child{node(1){}
};
%%%
\node[roottest=2,xshift=0ex] (roott) at (ancre){}
[deftree]
  child{node(2){}
};
\node[roottest=2,xshift=5ex] (roottt) at (ancre){}
[deftree]
  child{node(3){}
};
\node[mathsc, xshift=-1ex] at (root.west) {x};
\node[mathsc, yshift=0.8ex, xshift=-0.8ex] at (roott.130) {x};
\node[mathsc, xshift=1ex] at (roottt.east) {x};
\node[labf] at (1.south) {p};
\node[labf] at (2.south) {q};
\node[labf] at (3.south) {r};
%\node[labfs] at (1) {0}; 
\fill (root.center) circle  (1/\scont*\lbullet) ;
\fill (roott.center) circle  (1/\scont*\lbullet) ;
\fill (roottt.center) circle  (1/\scont*\lbullet) ;
\end{tikzpicture}
\caption{Contracting the root}  
\label{rootcont}
\end{figure}
\begin{figure}[!htb]
\begin{tikzpicture}[baseline=(current bounding box.center),scale=\scont]
\tikzstyle{every child node}=[intvertex]%[fill,circle,minimum size=3pt, inner sep=0pt]
\node[roottest=2](root) {}
[deftree]
child {node{}
  child{node(1){}}
  child{node(2){}}
  child{node(3){}
        edge from parent node[xshift=1.5ex, mathsc]{e}}
};
\node[mathsc, xshift=-1ex] at (root.west) {x};
\node[labf] at (1.south) {p};
\node[labf] at (2.south) {q};
\node[labf] at (3.south) {r};
%\node[labfs] at (1) {0}; 
\fill (root.center) circle  (1/\scont*\lbullet) ;
\end{tikzpicture}$\begin{array}{c}\scriptstyle{\textrm{contracting }}\\
{\leadsto}\\
\scriptstyle \textrm{along }e\end{array}$ 
\begin{tikzpicture}[baseline=(current bounding box.center),scale=\scont]
\tikzstyle{every child node}=[intvertex]%[fill,circle,minimum size=3pt, inner sep=0pt]
\node[roottest=2](root) {}
[deftree]
child {node(3){}
  child{node(1){}}
  child{node(2){}}
%  child{node(3){}} ;
};
\node[mathsc, xshift=-1ex] at (root.west) {x};
\node[labf] at (1.south) {p};
\node[labf] at (2.south) {q};
\node[mathsc,xshift=1ex] at (3.east) {r};
%\node[labfs] at (1) {0}; 
\fill (root.center) circle  (1/\scont*\lbullet) ;
\end{tikzpicture} $\begin{array}{c}\scriptstyle{\textrm{splitting at }}\\
{\leadsto}\\
\scriptstyle \textrm{internal vertex}\end{array}$
\begin{tikzpicture}[baseline=(current bounding box.center),scale=\scont]
\tikzstyle{every child node}=[intvertex]%[fill,circle,minimum size=3pt, inner sep=0pt]
\node(ancre){};
\node[roottest=2,xshift=-5ex] (root) at (ancre){}
[deftree]
  child{node(1){}
};
%%%
\node[roottest=2,xshift=5ex] (roott) at (ancre){}
[deftree]
  child{node(2){}
};
\node[roottest=2,yshift=5ex] (roottt) at (ancre){}
[deftree]
  child{node(3){}
};
\node[mathsc, xshift=-1ex] at (root.west) {r};
\node[mathsc, xshift=1ex] at (roott.east) {r};
\node[mathsc, yshift=1ex] at (roottt.north) {x};
\node[labf] at (1.south) {p};
\node[labf] at (2.south) {q};
\node[labf] at (3.south) {r};
%\node[labfs] at (1) {0}; 
\fill (root.center) circle  (1/\scont*\lbullet) ;
\fill (roott.center) circle  (1/\scont*\lbullet) ;
\fill (roottt.center) circle  (1/\scont*\lbullet) ;
\end{tikzpicture}  
\caption{Contracting a leaf} 
\label{leafcont} 
\end{figure}
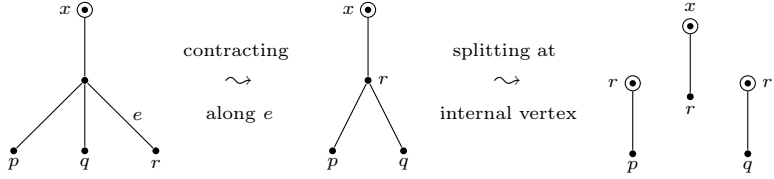
\end{itemize}
Their linear combination of trees differ from ours by their decorations, but also
by their structures. More precisely, with our decoration ($x$ decorates the root
and $0$ and $1$ the leaves) their linear combinations of trees are dual to the
standard basis of the universal enveloping algebra of $\Lie(X_0,X_1)$ presented
as a quotient of $\Ttg[x]$. 

Moreover, the authors of \cite{GanGonLev05} are not very precise about the case
where the leaves or the root are decorated by $0$. 
The \emph{forest cycling map} in loc. cit. 
sends any tree with root decorated by $0$ to the empty cycle. Hence, taking the
quotient by the ideal they generate,  we kill these
trees with root decorated by $0$.  

We can endow the linear combinations of trees $\T W$ with a canonical
ordering of each tree (root edge is the first edge, then down and
left). Then we can compute
$\dggl(\T W(x))$ in the \cite{GanGonLev05} setting and observe that it  satisfies
precisely equation \eqref{eq:dcyTw} after killing trees with $0$ as root decoration:
\begin{equation}\label{eq:ggl:dcyTw}
\dggl(\T W(x))=\sum_{U<V} \alpha_{U,V}^W\T U(x) \cdot \T V(x) + 
\sum_{U,V} \beta_{U,V}^W\T U(x) \cdot \T V (1)
\end{equation}
We will not give a proof of the above claim, which involves long computations on the
decompositions of Lie brackets  into the Lyndon bracket basis. However, let
us explain informally why it is true.
\begin{itemize}
\item Because of the ordering of edges and its alternating relation, working in
  the vector space 
\[
\Tcl[1;x]=\Tcl[1]\oplus \Tcl[x] \subset \Tts[1] \oplus \Tts[x]
\] is possible. The graded commutativity in \cite{GanGonLev05}
corresponds to the exterior product $\Tcl[1;x]\w \Tcl[1;x]$.
\item In computing $\dggl(\T W(x))$, the internal edges do not contribute. To see
  this, proceed by induction and use the fact that $\dL^2=0$.
\item The part of $\dggl$ corresponding to the root edge is nothing but $\dL$,
  hence this part equals 
\[
\sum_{U<V} \alpha_{U,V}^W\T U(x) \cdot \T V(x)
\] 
in $\dggl(\T W(x))$.
\item The part of $\dggl$ corresponding to external edges with $0$ as decoration
  gives a factor with a tree having $0$ as root decoration. Hence this part is
  $0$ because we have killed these trees.
\item The part of $\dggl$ corresponding to an external edge $e$  with $1$ as
  decoration (given by ``contract and split'') is dual to the action by special
  derivations (given by ``attach at $1$'') as drawn below:

%\[
%\begin{figure}[!htb]
%%%%%%%%%%%%%%%%%%%%%
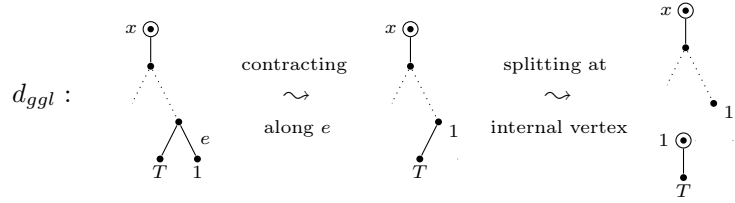
\begin{center}
$\dggl :$ \hspace{1em}
  \begin{tikzpicture}[baseline=(current bounding box.center),scale=1]
 \tikzstyle{every child node}=[intvertex]
 \node[roots](root) {}
 [deftree]
 child {node{}
   child[edgesp=1,dotted] {node[minimum size=0pt, inner sep=0pt](1){}}
   child[edgesp=1.5,dotted] {node{} 
     child[deftree,solid]{ node(2){}}  
     child[deftree,solid]{ node(3){}
                           edge from parent node[xshift=1.5ex, mathsc]{e}} 
   }
}
;
\fill (root.center) circle (1/1*\lbullet) ;
\node[mathsc, xshift=-1ex] at (root.west) {x};
%\node[labf] at (1.south){0};
\node[labf] at (2.south){T};
\node[labf] at (3.south){1};
 \end{tikzpicture}
$\begin{array}{c}\scriptstyle{\textrm{contracting }}\\
{\leadsto}\\
\scriptstyle \textrm{along }e\end{array}$ 
%%%%%%%%%%%%%%%%%%%%%%%%%%%%%%%%%%%%%%%%%%%%%%
%
  \begin{tikzpicture}[baseline=(current bounding box.center),scale=1]
 \tikzstyle{every child node}=[intvertex]
 \node[roots](root) {}
 [deftree]
 child {node{}
   child[edgesp=1,dotted] {node[minimum size=0pt, inner sep=0pt](1){}}
   child[edgesp=1.5,dotted] {node(a){} 
     child[deftree,solid]{ node(2){}}  
     child[deftree,draw=white]{ 
                       node[minimum size=0pt, inner sep=0pt](3){}
                               } 
   }
}
;
\fill (root.center) circle (1/1*\lbullet) ;
\node[mathsc, xshift=-1ex] at (root.west) {x};
%\node[labf] at (1.south){0};
\node[labf] at (2.south){T};
\node[labf,, xshift=1ex] at (a.east){1};
 \end{tikzpicture}
%%%%%%%%%%%%%%%%%%%%%%%%%%%%%%%%%%%%%%%%%%
 $\begin{array}{c}\scriptstyle{\textrm{splitting at }}\\
{\leadsto}\\
\scriptstyle \textrm{internal vertex}\end{array}$
  \begin{tikzpicture}[baseline=(current bounding box.center),scale=1]
 \tikzstyle{every child node}=[intvertex]
 \node[roots](root) {}
 [deftree]
 child {node{}
   child[edgesp=1,dotted] {node[minimum size=0pt, inner sep=0pt](1){}}
   child[edgesp=1.5,dotted] {node(a){} 
     child[deftree,draw=white]{ 
                           node[minimum size=0pt, inner sep=0pt](2){}
                          }  
     child[deftree,draw=white]{ 
                       node[minimum size=0pt, inner sep=0pt](3){}
                               } 
   }
}
;
\fill (root.center) circle (1/1*\lbullet) ;
\node[mathsc, xshift=-1ex] at (root.west) {x};
%\node[labf] at (1.south){0};
%
\node[labf,, xshift=1ex] at (a.east){1};
\node[roots,xshift=-1ex](roott) at (2)  {}
[deftree]
  child{node(b){}
};
\fill (roott.center) circle (1/1*\lbullet) ;
\node[mathsc, xshift=-1ex] at (roott.west) {1};
\node[labf] at (b.south){T};
\end{tikzpicture}
\end{center}
%\caption{Contracting a leaf} 
%\label{leafcont} 
%\end{figure}
%\]
where $T$ denotes a subtree; and 
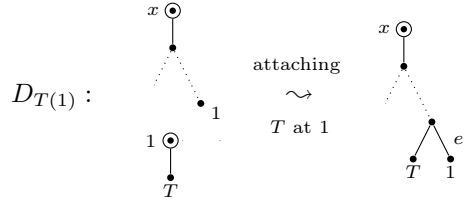
\begin{center}
$D_{T(1)} :$ \hspace{1em}
  \begin{tikzpicture}[baseline=(current bounding box.center),scale=1]
 \tikzstyle{every child node}=[intvertex]
 \node[roots](root) {}
 [deftree]
 child {node{}
   child[edgesp=1,dotted] {node[minimum size=0pt, inner sep=0pt](1){}}
   child[edgesp=1.5,dotted] {node(a){} 
     child[deftree,draw=white]{ 
                           node[minimum size=0pt, inner sep=0pt](2){}
                          }  
     child[deftree,draw=white]{ 
                       node[minimum size=0pt, inner sep=0pt](3){}
                               } 
   }
}
;
\fill (root.center) circle (1/1*\lbullet) ;
\node[mathsc, xshift=-1ex] at (root.west) {x};
%\node[labf] at (1.south){0};
%
\node[labf,, xshift=1ex] at (a.east){1};
\node[roots,xshift=-1ex](roott) at (2)  {}
[deftree]
  child{node(b){}
};
\fill (roott.center) circle (1/1*\lbullet) ;
\node[mathsc, xshift=-1ex] at (roott.west) {1};
\node[labf] at (b.south){T};
\end{tikzpicture}
$\begin{array}{c}\scriptstyle{\textrm{attaching }}\\
{\leadsto}\\
\scriptstyle T \textrm{ at  }1\end{array}$ 
  \begin{tikzpicture}[baseline=(current bounding box.center),scale=1]
 \tikzstyle{every child node}=[intvertex]
 \node[roots](root) {}
 [deftree]
 child {node{}
   child[edgesp=1,dotted] {node[minimum size=0pt, inner sep=0pt](1){}}
   child[edgesp=1.5,dotted] {node{} 
     child[deftree,solid]{ node(2){}}  
     child[deftree,solid]{ node(3){}
                           edge from parent node[xshift=1.5ex, mathsc]{e}} 
   }
}
;
\fill (root.center) circle (1/1*\lbullet) ;
\node[mathsc, xshift=-1ex] at (root.west) {x};
%\node[labf] at (1.south){0};
\node[labf] at (2.south){T};
\node[labf] at (3.south){1};
 \end{tikzpicture}
\end{center}

\end{itemize}
\end{rem} 
\subsection{A combinatorial statement}\label{combist}
Because of the admissibility issue at $1$ for algebraic cycles explained in
Section \ref{subsec:L011pb}, equation \eqref{eq:dcyTw} cannot be used directly.
The solution found in Section \ref{Lcex3} leads us to express
\eqref{eq:dcyTw} using  products as $\T U(x)\w (\T V(x)-\T V(1))$ rather than
$\T U(x)\w \T V (1)$. 

\begin{defn} For any Lyndon word $W$, let $\Tcu W$ be the difference
\[
\Tcu W=(\T W(x)-\T W(1)).
\]  
In order to use a consistent notation, we set $\Tc W=\T W(x)$, which can be
thought as $\Tc W=\T W(x) -\T W(0)$ where $\T W(0)=0$. 
\end{defn}
Because 
\[
\T W(1) = \Tcu W -\Tc W,
\]
a basis of $\Tcl[1;x]$ is given by the union of the two families:
\begin{align*}
\Tc W & \mx{ for any Lyndon word } W \\ 
\Tcu W & \mx{ for any Lyndon word } W. \\
\end{align*}
Hence a basis of $\Tcl[1;x] \w \Tcl[1;x]$ is given by the union of the
following families:
\begin{align*}
\Tc U\w\Tc V & \mx{ for any Lyndon word } U<V\\
\Tcu U\w \Tc V  & \mx{ for any Lyndon word } U\neq V\\
\Tcu U\w \Tcu V & \mx{ for any Lyndon word } U<V. 
\end{align*}
We can now rewrite equation \eqref{eq:dcyTw} in terms of the above basis. We
also write the cobracket $\dc(\Tcu W)$ in terms of this basis.
\begin{defn} \label{def:aapbbpcoef} Let $W$ be a Lyndon word. We define coefficients
  $a_{U,V}^W$, $\ap_{U,V}^W$ for any Lyndon words $U<V$ and coefficient
  $b_{U,V}^W$ and $\bp_{U,V}$ for any Lyndon words $U,V$ as follows:
\begin{equation}\label{ED-Tc}
\dc(\Tc W)=\sum_{U<V} a_{U,V}^W\Tc U \w \Tc V +
 \sum_{U,V}b_{U,V}^W\Tcu{U}\w \Tc V, 
\tag{ED-$\Tcg^0$}
\end{equation}
and 
\begin{equation}\label{ED-Tcu}
\dc(\Tcu W)=\sum_{U<V} \ap_{U,V}^W\Tcu U \w \Tcu V +
\sum_{U,V}\bp_{U,V}^W\Tcu{U}\w\Tc V
%+%\sum_{V}\ap_{0,V}\Lcz\Lc_V, 
\tag{ED-$\Tcg^1$}
\end{equation}
\end{defn}
Note that no term of the form $\Tcu U \w \Tcu V $ appears in equation \eqref{ED-Tc} because
there is no term of the form  $\T U (1) \w \T V(1)$ appearing in equation
\eqref{eq:dcyTw}. More precisely, we express the $a$'s, $b$'s, $\ap$'s and
$\bp$'s in terms of the $\alpha$'s and $\beta$'s of equation \eqref{eq:dcyTw}.
\begin{lem}\label{aapbbpcoef} 
For any Lyndon words $W$, the following holds
\begin{equation}\label{abcoef}
\begin{array}{ll}
a_{U,V}^W=\alpha_{U,V}^W+\beta_{U,V}^W -\beta_{V,U}^{W} & \mx{ for } U<V\\[0.5em]
b_{U,V}^W=\beta_{V,U}^W& \mx{ for any }U,\, V
\end{array}
\end{equation}
and
\begin{equation}\label{apbpcoef}
\begin{array}{ll}
\ap_{U,V}^W=-a_{U,V}^W&\mx{ for } U<V,\\[0.5em]
\bp_{U,V}^W=a_{U,V}^W+b_{U,V}^W&\mx{ for } U<V,\\[0.5em]
\bp_{V,U}^W=-a_{U,V}^W+b_{V,U}^W&\mx{ for } U<V,\\[0.5em]
%\ap_{0,V}^W=a_{0,V}&\mx{ for any }V, \\[0.5em]
\bp_{U,U}^W=b_{U,U}^W&\mx{ for any } U.
\end{array}
\end{equation}
\end{lem}
Note that coefficients $\ap$'s and $\bp$'s are defined in terms of $a$'s and
$b$'s (from equation \eqref{ED-Tc}), and not in term of coefficients $\alpha$'s
and $\beta$'s (from equation \eqref{eq:dcyTw}). In particular this makes  the proof
of the above Lemma and the comparison between $\dc(\Tc W)$ and $\dc(\Tcu W)$ easier:
\begin{multline}\label{eq:dcyTcTcu}
\dc(\Tc W -\Tcu W)= \\
\sum_{U<V} a_{U,V}^W
\left(\Tc U \w \Tc V +\Tcu U \w \Tcu V + \Tcu V \w \Tc U - \Tcu U \w \Tc V   \right)
\end{multline}
\begin{proof} Beginning with equation \eqref{eq:dcyTw} for a Lyndon word $W$
\[
\dc(\T W(x))=\sum_{U<V} \alpha_{U,V}^W\T U(x) \w \T V(x) + 
\sum_{U,V} \beta_{U,V}^W\T U(x) \w \T V (1),
\]
we write terms of the form  $\T U(x) \w \T V (1)$ as 
\begin{align*}
\T U(x) \w \T V (1)=&\T U(x)\w (-(\T V(x)-\T V(1))+\T V(x)) \\
&=-\Tc U\w \Tcu V +\Tc U \w \Tc V\\
&=\Tcu V \w \Tc U +\Tc U \w \Tc V.\\
\end{align*}
Then equation \eqref{abcoef} follows by reordering the terms of the sum. The
second sum $\sum_{U,V} \beta_{U,V}^W\T U(x) \w \T V (1)$ is, in this reordering,
cut in three pieces corresponding to $U<V$, $U>V$ and $U=V$.  Note that when
$U=V$, the term $\Tc U \w \Tc V$ vanishes leaving only the term 
$\Tcu V \w \Tc U$.  The inversion of letters $U$ and $V$ in the equality 
\[
b_{U,V}^W=\beta_{V,U}^W
\]
is induced by the choice of terms $\Tcu U \w \Tc V$ as ``cross-terms'' of
the basis instead of $\Tc U \w \Tcu V$. This choice is motivated by the position
 of the tree $\Tcu 0$ in the products. The tree $\Tcu 0$  will correspond to the
 algebraic cycle $\Lcz$ which usually arises as the first term of a product.  

In order to obtain equation \eqref{apbpcoef}, we compute
$ \dc(\Tcu W)=\dc(\T W (x))-\dc(\T W(1))$
 as 
\begin{multline}
\dc(\Tcu W)= \sum_{U<V}a_{U,V}^W \Tc U \w \Tc V + \sum_U b_{U,U}^W\Tcu U \w \Tc U\\
+\sum_{U<V}  b_{U,V}^W\Tcu U \w \Tc V + \sum_{V<U}  b_{U,V}^W\Tcu U \w \Tc V\\
-\sum_{U<V}\gamma_{U,V}^W \T U (1) \w \T V(1)
\end{multline}
where $U,V$ are Lyndon words. 
We observe that 
\[
\Tc U \w \Tc V=-\Tcu U \w \Tcu V +\Tcu U \w \Tc V - \Tcu V \w \Tc U + \T U(1) \w
\T V(1).
\]
Substituting this expression for $\Tc U \w \Tc V$ in the above equation for
$\dc(\Tcu W)$, we obtain
\begin{multline*}
\dc(\Tcu W)= \sum_{U<V}-a_{U,V}^W \Tcu U \w \Tcu V +  \sum_{U<V}a_{U,V}^W \Tcu U
\w \Tc V +  \sum_{U<V}-a_{U,V}^W \Tcu V \w \Tc U
\\ \sum_U b_{U,U}^W\Tcu U \w \Tc U+
\sum_{U<V}  b_{U,V}^W\Tcu U \w \Tc V + \sum_{V<U}  b_{U,V}^W\Tcu U \w \Tc V%\\
%-\sum_{U<V}\gamma_{U,V}^W \T U (1) \w \T V(1)
\end{multline*} 
because by Lemma \ref{relV01}, we have
\[
\gamma_{U,V}^W=\alpha_{U,V}^W+\beta_{U,V}^W -\beta_{V,U}^{W} =a_{U,V}^W.
\]
Collecting terms in the last expression of $\dc(\Tcu W)$ yields equation \eqref{apbpcoef}.
\end{proof}

Lemma \ref{relV01} gives us some extra information about the coefficients
$\alpha^{W}_{U,V}$ and $\beta^W_{U,V}$ when $W$, $U$ or $V$ is equal to the
letter $0$ or $1$. This translates for coefficients $a$'s, $b$'s, $\ap$'s and
$\bp$'s as:
\begin{lem}\label{rem:coefAL}
\begin{itemize}
\item If $W$ is the Lyndon word $0$ or $1$, then:
\[
a_{U,V}^0=b_{U,V}^0=\ap_{U,V}^0=\bp_{U,V}^0=0,
\quad
a_{U,V}^1=b_{U,V}^1=\ap_{U,V}^1=\bp_{U,V}^1=0
\]
for any Lyndon words $U$ and $V$.
\item  For any Lyndon word $W$, $U$ and $V$ of length at least $2$, one has
\[
a_{0,V}^W=\ap_{0,V}^W=0 
\quad \mx{ and } \quad
b_{U,0}^W=\bp_{U,0}^W= 0.
\] 
which says that there is no term in $\Tc 0 \Tc V$, $\Tcu 0 \Tcu V$ or $\Tcu U
\Tc 0$. Moreover, we have
\[
a_{U,1}^W=\ap_{U,1}^W=0, 
\quad \mx{ and } \quad  
b_{1,V}^W=\bp_{1,V}^W=0  \]
which says that there is no term in $\Tc U \Tc 1$, $\Tcu U \Tcu 1$ or $\Tcu 1
\Tc V$.
\end{itemize}
We also note that for  $W$ a Lyndon word, we have
\[
a_{U,V}^W=b_{U,V}^W=\ap_{U,V}^W=\bp_{U,V}^W=0
\]
whenever the length of $U$ plus the length of  $V$ is not equal to the length of $W$.

 In particular, equation   \eqref{ED-Tc} and equation \eqref{ED-Tcu} involve
 only Lyndon words of length   smaller than the length of $W$.
\end{lem} 
\begin{proof}
This is a direct consequence of Lemma \ref{relV01} and Lemma \ref{aapbbpcoef},
\end{proof}

The two equations \eqref{ED-Tc} and \eqref{ED-Tcu} provide the combinatorial
situation which will allow the construction of the algebraic cycles. However, we
cannot directly relate the above equation to what happens in the cycle algebra
 $\cNg{\ps}{\bullet}$, %
because
\begin{itemize}
\item the structures are not the same : $\Tcl[1;x]$ is a Lie coalgebra while
  $\cNg{\A^1}{\bullet}$ is a commutative differential graded algebra;
\item we can associated a cycle $\Lc_W$ to $\Tc W$ (resp. $\Lcu_W$  to $\Tcu W$)
  only after that the cycle has been constructed.
\end{itemize}
%%%%%%%%%%%%
Later, we will apply the following Proposition \ref{dALdALu0} to the complex
$\Nge{\ps}{\bullet}$ of equidimensional cycles over $\ps$. The result and the
proof are algebraic and do not involve any geometry. Proposition
\ref{dALdALu0} concludes this combinatorial portion of the paper. It relates  the structure
of equations \eqref{ED-Tc} and \eqref{ED-Tcu} to the ``differential system''
arising in the construction of the algebraic cycles in the next section.

Let $ (\mc A^{\bullet}, \da[\mc A])$ be a cdga and  $p$ an integer $\geqs 2$. We assume the
following:
\begin{itemize}
\item There exist two degree $1$ elements  $\Acz$ and $\Aco$ in $\mc A^1$  such that
\[
\da[\mc A](\Acz)=\da[\mc A](\Aco)=0;
\]
\item For any Lyndon words $W$ of length $k$ with $2 \leqs k \leqs p-1$, there
  exist two degree $1$ elements $\Ac W$ and $\Acu W$ in $\mc A^1$ satisfying
  \eqref{ED-Tc} and  \eqref{ED-Tcu} respectively:
\begin{align*}
\da[\mc A](\Ac W)=\sum_{U<V} a_{U,V}^W\Ac U \w \Ac V +
 \sum_{U,V}b_{U,V}^W\Acu{U}\w \Ac V,\\
\da[\mc A](\Acu W)=\sum_{U<V} \ap_{U,V}^W\Acu U \w \Acu V +
\sum_{U,V}\bp_{U,V}^W\Acu{U}\w\Ac V
\end{align*}
\end{itemize}
 
\begin{prop}\label{dALdALu0} 
Let $W$ be a Lyndon word of length $p$. Let $\RAL$ and $\RALu$ be the degree $2$
elements  defined by
\[
\RAL=\sum_{U<V} a_{U,V}^W\Ac U \Ac V + \sum_{U,V}b_{U,V}^W\Acu{U}\Ac V, 
\]
and  
\[
%\begin{equation*}
\RALu=\sum_{U<V} \ap_{U,V}^W\Acu U \Acu V + \sum_{U,V}\bp_{U,V}^W\Acu {U}\Ac V, 
\]
where the coefficients $a$, $b$ $\ap$, $\bp$ are, as above, those defined by equations
 \eqref{ED-Tc} and  \eqref{ED-Tcu}.  

Then
\[
\da[\mc A](\RAL)=\da[\mc A](\RALu)=0.
\]
\end{prop}
\begin{proof}First we recall that the symmetric algebra $S^g(V^{\bullet})$ over
  a graded vector space $V$ is the tensor algebra modulo the ideal generated by
\[
ab-(-1)^{\deg(a)\deg(b)}ba
\]
whenever $a$ and $b$ are homogeneous. 

Let $S^{\bullet}_{\mc T}=S^g(\Tcl[1;x])$ be a symmetric graded algebra over the vector space
$\Tcl[1;x]$ concentrated purely in degree $1$. We shall use the same notation
for an element in $\Tcl[1;x]$ and its image in $S^{\bullet}_{\mc T}$. The
cobracket $\dc$ and the Leibniz rule induces a differential on $S^{\bullet}_{\mc
  T}$ denote by $\da[cy]$. This makes $S^{\bullet}_{\mc T}$ into a cdga. In
particular  equations
\eqref{ED-Tc} and  \eqref{ED-Tcu} hold in $S^{\bullet}_{\mc T}$ after
replacing $\dc$ by $\da[cy]$ and the wedge product in $\Tcl[1;x] \w \Tcl[1;x]$
by the product in $S^{\bullet}_{\mc T}$.  

Let $S^{\bullet}_{\mc T, \, \leqs p-1}$ be the subalgebra of $S^{\bullet}_{\mc T}$
generated by elements $\Tc U$ and $\Tcu V$ for $U$ and $V$ Lyndon words of
length $k \leqs p-1$. It is a sub-cdga of $S^{\bullet}_{\mc T}$ because
equations \eqref{ED-Tc} and \eqref{ED-Tcu} involve only words of smaller length
on the right-hand side. 

Let $W$ be a Lyndon word of length $p$. Note that the degree $2$ elements
$\da[cy](\Tc W)$ and $\da[cy](\Tcu W)$ of $S^{\bullet}_{\mc T}$ also lie in
$S^{\bullet}_{\mc T, \, \leqs p-1}$:
\[
\da[cy](\Tc W), \, \da[cy](\Tcu W) \in S^{\bullet}_{\mc T, \, \leqs p-1}.
\]
However that they are not boundary in $S^{\bullet}_{\mc T, \, \leqs p-1}$.

We define an algebra morphism $\vp : S^{\bullet}_{\mc T, \, \leqs p-1} \lra \mc
A^{\bullet}$ on the degree one elements by 
\[
\vp(\Tc U)=\Ac U, \qquad \mx{and} \qquad \vp(\Tcu U)=\Acu U
\]
for any Lyndon word $U$ of length $k$ with $2 \leqs k \leqs p-1$ and by
$\vp(\Tcu 0)=\Acz$ and $\vp(\Tc 1)=\Aco$. The morphism $\vp$ is a cdga morphism
due to the assumption on $\da[\mc A](\Ac U)$ and $\da[\mc A](\Acu U)$ for
Lyndon words $U$ of length $\leqs p-1$. Hence $\RAL=\vp(\da[cy](\Tc W))$ and
$\RALu=\vp(\da[cy](\Tcu W))$ satisfy
\[
\da[\mc A](\RAL)=\vp(\da[cy]\circ \da[cy](\Tc W))=0,
\quad \mx{and} \quad  
\da[\mc A](\RALu)=\vp(\da[cy]\circ \da[cy](\Tcu W))=0.
\]

\end{proof}

\newcommand{\lenqt}{0.3}
\newcommand{\arbdot}[1]{\begin{tikzpicture}[%
baseline={([yshift=-0.6ex,xshift=-1ex]current bounding box.center)},scale=\lenqt]
\tikzstyle{every child node}=[mathscript mode,minimum size=0pt, inner sep=0pt]
\node[inner sep=0pt](root) {}
[level distance=\ledge,sibling distance=\sibdis]
%[mathscript mode]%[fill,circle,minimum size=3pt, inner sep=0pt]
%\node[draw,circle,inner sep=0.5pt](root) 
%{$\bullet$}
%[level distance=1.5em,sibling distance=3ex]
child{node%[fill, circle, minimum size=2pt,inner sep=0pt]
           {}[level distance=\ledge]
  child{ node[inner sep=0pt]{}}
  child{ node[fill, circle, minimum size=1.5pt,inner sep=0pt](1){}
[edge from parent path={[solid]
(\tikzparentnode) to (\tikzchildnode)
 }]}
[edge from parent path={[dotted]
(\tikzparentnode) to (\tikzchildnode)
 }]
};
%\fill (root.center) circle (1/\lenqt*\lbullet) ;
\node[mathsc,yshift=0.5em] at (1.east) {#1};
\end{tikzpicture}
}
\newcommand{\insertion}[3][0.5]{%
\begin{tikzpicture}[%
baseline={([yshift=0ex]current bounding box.center)},scale=#2]
\tikzstyle{every child node}=[mathscript mode,minimum size=0pt, inner sep=0pt]
\node[minimum size=0pt, inner sep=0pt]%[draw,circle,inner sep=0.5pt](root) 
{}
[level distance=1.5em,sibling distance=3ex]
child {node[fill, circle, minimum size=2pt,inner sep=0pt]{}[level distance=1.5em]
  child{ node(a){}}
  child{ node(b){}}
};
\node[yshift=-0.5ex,xshift=-0.5ex](a1) at (a){};
\node[yshift=-0.5ex,xshift=0.5ex](b1) at (b){};
\node[xshift=#1ex,yshift=#1ex](c) at (b.east){$\scriptstyle{#3}$};
\draw[<-] (a1.center)--(b1.center) node[below]{};
\end{tikzpicture}
}

\newcommand{\subt}[2]{\left(#2/{\arbdot{#1}}\right)^r}
\newcommand{\inse}[2][0.5]{\insertion[#1]{0.4}{#2}}
\newcommand{\ins}[2][0.7]{\insertion[#1]{0.3}{#2}}
\newcommand{\qtl}[2]{ \vphantom{\left(#2/\arbdot{a}\right)}^l\left(#2/{\arbdot{#1}}\right)
} 
\newcommand{\qts}[3]{ #1\,\sm \hspace{-1ex}\st{#2}{} _{#3}
}     
\newcommand{\Le}[2][]{Le_{#2}^{#1}}
\newcommand{\Lep}[2][]{Le_{pos.}^{#1}(#2)}

\section{Construction of algebraic cycles}\label{treetocycle}

In this section we define two ``differential systems'' for algebraic cycles,
one corresponding to cycles with empty fiber at $x=0$ and another corresponding to
cycles with empty fiber at $x=1$. Then, we show that there exist two families
of cycles in $\Nge{X}{1}$
satisfying these systems induced by two families of cycles in $\Nge{\A^1}{1}$.

\renewcommand{\Lc}{\Lcal^{0}}
\renewcommand{\Lcu}{\Lcal^{1}}
\renewcommand{\AL}{A_{\Lc}}
\renewcommand{\ALu}{A_{\Lcu}}
\renewcommand{\dN}{\partial}
\subsection{Equidimensional cycles}
We recall that the base field is $\Q$ and that all varieties considered below
are $\Q$-varieties. We also recall that $\square^1=\p^1\sm \{1\}$, that
$\square^n=(\square^1)^n$, and that  $X=\ps$.
\begin{defn}[Equidimensionality]\label{def:equi}Let $Y$ be an irreducible smooth variety
\begin{itemize}
\item Let $\Ze[p](Y,n)$ denote the free abelian group
  generated by irreducible closed subvarieties $Z \subset Y\times \square^n$ such
  that for any face
  $F$ of $\square^n$, the intersection $Z\cap (Y\times F) $ is empty or the
  restriction of $p_1 : Y\times \square^n \lra Y$ to  
\[
 Z\cap (Y\times F) \lra Y
\]
is dominant and equidimensional of pure relative dimension  $\dim(F)-p$,i.e.~the non empty fibers have the same required dimension.
\item We say that elements of $\Ze[p](Y,n)$ are \emph{equidimensional over $Y$ with
    respect to any  face} or simply \emph{equidimensional}.
\item Following the definition of $\cNg{Y}{k}(p)$, let $\Nge{Y}{k}(p)$ denote 
\[
\Nge{Y}{k}(p)=\Alt\left(\Ze[p](Y,2p-k )\otimes \Q\right).
\]
\end{itemize}
\end{defn}
\begin{defn} Let $C$ be an element of $\cNg{Y}{\bullet}$ decomposed in terms of cycles as 
\[
C=\sum_{i\in I} q_i Z_i, \qquad q_i \in \Q, 
\]
where $I$ is a finite set and the $Z_i$ are irreducible closed subvarieties
of $Y \times \square^{n_i}$ 
 intersecting all the faces of $\square^{n_i}$ properly (i.e. in codimension
 $p_i$).
\begin{itemize}
\item The support of $C$ is defined as $\Supp(C)=\bigcup_i Z_i$.
\item For $C$ in $\Nge{Y}{k}(p)$, we will say that $C$ has empty fiber at a
  point $y$ in $Y$ if for any $i$ in $I$ the fiber of $Z_i \lra Y$ at $y$ is
  empty.
\end{itemize} 
\end{defn}
\begin{prop}\label{equistr}Let $Y$ be an irreducible smooth variety.
\begin{enumerate}
\item The differential $\da[Y]$ on $\cNg{Y}{\bullet}$ induces a differential:
\[
\Nge{Y}{k}(p) \st{\da[Y]}{\lra}  \Nge{Y}{k+1}(p)
\]
which makes $\Nge{Y}{\bullet}(p)$ into a sub-complex of $\cNg{Y}{\bullet}(p)$. 
\item $\ds \Nge{Y}{\bullet}=\oplus_{p\geqs 0}\Nge{Y}{\bullet}(p)$ is a
  subalgebra (sub-cdga) of $\cNg{Y}{\bullet}$. 
\item Assume that $Z$ or $Z'$ has an empty fiber at a point $y$ in $Y$. Then the
  fiber at $y$ of $\ol{Z}\cdot \ol{Z'}$ is empty. 
\end{enumerate}
\end{prop}
\begin{proof}
As the generators of $\Ze[p](Y,2p-k)$ are equidimensional over $Y$ when
intersected with any face,  they stay 
equidimensional over $Y$ with respect to any face  when intersected with a
codimension $1$ face because a face intersected with a codimension $1$ face is
either another
face or the intersection is empty. This gives the first point.

Let $Z$ and  $Z'$ be two generators of $\Ze[p](Y, 2p-k)$ and  $\Ze[q](Y,2q-l)$ respectively
for $p$, $q$, $k$ and $l$ integers. By definition, for any face $F \subset
\square^{2p-k}$, either the restriction of the projection $p_1 : Y \times F \lra Y$  to
\[
 p_1 : Z\cap(X \times F) \lra Y 
\]
 is equidimensional of relative dimension $\dim(F)-p$, or the above intersections are empty. Similarly for $Z'$. 

Let $F$ and $F'$ be two faces as above, and assume that none of the 
intersections $ Z\cap(X \times F)$ and $Z'\cap(X \times F')$ is empty. Then
\[
Z \times Z' \cap(Y\times Y \times F \times F') 
\subset Y \times Y \times \square^{2(p+q)-k-l}
\]
is equidimensional over $Y\times Y$ of relative dimension
$\dim(F)+\dim(F')-p-q$.  
For any point $x$ in the image of the 
diagonal $\Delta : Y \lra Y \times Y$, we denote the fiber over $x$ with the
subscript $x$. 
In particular we have
\begin{multline*}
\dim\left(
\big((Z \times Z') \cap(Y\times Y \times F \times F')
\big)_x
\right)=\\
\dim\left(Z \times Z' \cap(\{x\} \times F \times F')\right) = \\
\dim(\{x\})+\dim(F)+\dim(F')-p-q
\end{multline*}
and $Z \times Z' \cap(\im(\Delta) \times F \times F')$ is equidimensional over
$Y$ of relative dimension $\dim(F)+\dim(F')-p-q$ by either of the two projections $Y
\times Y \lra Y$. If either $Z \cap (Y \times F)$ or $Z' \cap (Y \times F')$ is
empty, then the intersection 
\[
Z \times Z' \cap \left(Y\times Y \times F \times F'\right)
\]
is empty and so is $Z \times Z' \cap(\im(\Delta) \times F \times F')$.

From this, we deduce that 
\[
(\Delta\times \id)^{- 1}(Z \times Z')\simeq Z \times Z' \cap\left(\im(\Delta) \times 
\square^{2(p+q)-k-l}\right)
\] 
is equidimensional over $Y$ with respect to any face.  Hence,
\[
Z \cdot Z' = \Alt((\Delta \times \id)^{- 1}(Z \times Z')) \in \cNg{Y}{\bullet}
\]
and the product in $\cNg{Y}{\bullet}$ induces a cdga structure on
$\Nge{Y}{\bullet}$ which makes it into a sub-cdga. 

Moreover, from the above computation, we see that if the fiber of $Z$ is empty
at a point $y$, then, denoting the various fibers at $y$ with the subscript
$y$, we have
\[
\Big(
(\Delta \times \id)^{- 1}(Z \times Z')
\Big)_y=Z \times Z'\cap(\{(y,y)\}\times
\square^{2(p+q)-k-l})=Z_y\times Z'_y=\emptyset. 
\]  
The same holds if $Z'$ is empty at $y$, which gives the last point of the proposition. 
\end{proof}
In order to compare the situation in $\cNg{X}{\bullet}$ and in
$\cNg{\A^1}{\bullet}$, we will use the following proposition.
\begin{prop}\label{equiopen}
Let $Y_0$ be an open dense subset of an irreducible smooth variety $Y$ 
and let $j : Y_0 \lra Y$ denote the inclusion. 
\begin{enumerate}
\item The restriction of cycles from $Y$ to $Y_0$
induces a morphism of cdga's
\[
j^* : \Nge{Y}{\bullet} \lra \Nge{Y_0}{\bullet}.
\]
\item The morphism $j^*$ is injective. 
\item Let $C$ be in $\cNg{Y_0}{\bullet}$ be  decomposed in terms of cycles as 
\[
C=\sum_{i\in I} q_i Z_i, \qquad q_i \in \Q
\] 
where $I$ is a finite set. 
Assume that  for any $i$, the Zariski
closure $\ol{Z_i}$ of $Z_i$ is in $\Ze[p_i](Y,n_i)$. 
Define $\bar C$ as 
\[
\bar C=\sum_{i\in I} q_i \ol{Z_i};
\]
then
\[
\bar C \in \Nge{Y}{\bullet} \qquad \mx{and}\qquad  C=j^*(\bar C)\in \Nge{Y_0}{\bullet}.
\]
\item In particular $\ol{j^*(D)}=D$ for any $D$ in $\Nge{Y}{\bullet}$.
\end{enumerate}
\end{prop}
\begin{proof}
It is enough to prove the first part of the proposition for generators of
$\Nge{Y}{\bullet}$.

Let $Z$  be an irreducible, closed   subvariety of codimension
$p$  of $Y \times \square^{2p-k}$  such that for any face $F$   of $\square^{2p-k}$,  the
intersection 
\[
Z \cap (Y \times F) \qquad \left(\mx{resp. } Z' \times (Y\times F') \right)
\]
is either dominant equidimensional over $Y$ of relative dimension $\dim(F)-p$
or empty.

Let $Z_0$  be the intersections $Z\cap Y_0$. As for
any face $F$ of $\square^{2p-k}$ we have
\[
Z_0\cap(Y_0\times F)=\left(Z\cap(Y\times F)\right)\cap Y_0\times \square^{2p-k},
\]
$Z_0$ is equidimensional with respect to any face over $Y_0$ with
relative dimension $\dim(F)-p$. This also shows that $j^*$
commutes with the differential on $\Nge{Y}{\bullet}$ and on
$\Nge{Y_0}{\bullet}$.

Moreover if $\ol{Z_0}$ denotes the Zariski closure of $Z_0$ in $Y$, we have
$\ol{Z_0}=Z$ because $Z$ is closed and irreducible. This gives part (4) of the
proposition and the injectivity of $j^*$.

Let $Z'$ be  an irreducible, closed   subvariety of codimension
$q$  of $Y \times \square^{2q-l}$ providing a generator of $\Ze[q](Y,2q-l)$. In
order to show that $j^*$ commutes with the product 
structure, it suffices to remark that 
\[
Z_0\times Z_0' =(Z\times Z')\cap \left(Y_0\times Y'_0 \times
  \square^{2(p+p')-k-k'}\right) 
\subset Y \times Y \times \square^{2(p+p')-k-k'}.
\]

Let $C$ and $ \bar C$ be as in the proposition. The fact that $\bar C$ is in
$\Nge{Y}{\bullet}$ follows directly from the definition. To prove that
\[
C=j^*(C') \in \Nge{Y_0}{\bullet},
\]
we can assume that $I$ contains only one element  and that $q_1=1$. Then it
follows from the fact that $Z_1=\ol{Z_1}\cap Y_0 \subset Y$. 

\end{proof}
The main geometric tool of our construction comes from the usual multiplication
on $\A^1$ which induces a homotopy on $\Nge{\A^1}{\bullet}$ between the identity
and the constant cycle given by the fiber at $0$.

Let $m : \A^1 \times \A^1 \lra \A^1$ be the multiplication map sending $(x,y)$
to $x y$, and let $\tau : \square^1=\p^1\sm\{1\}\lra \A^1$ be the isomorphism
sending the affine coordinate $u$ to $\frac{1}{1-u}$. The map $\tau$ sends
$\infty$ to $0$, $0$ to $1$ and extends as a map from $\p^1$ to $\p^1$ sending
$1$ to $\infty$.

The maps $m$ and $\tau$ are in particular flat and equidimensional of relative
dimension $1$ and $0$ respectively.

Consider the following commutative diagram for a positive integer $n$;
\[
\begin{tikzpicture}
\matrix (m) [matrix of math nodes,
row sep=2.5em, column sep=9.5em, 
 text height=1.5ex, text depth=0.25ex]
{\A^1\times \square^1 \times \square^n & \A^1 \times\square^n \\
\A^1\times \square^1 & \A^1 \\
\A^1 & \\};
\path[->,font=\scriptsize]
(m-1-1) edge node[mathsc,auto]{(m\circ(\id_{\A^1}\times \tau))\times \id_{\square^n}} (m-1-2)
(m-1-1) edge node[mathsc,left]{p_{\A^1\times \square^1}} (m-2-1)
(m-1-2) edge node[mathsc,right]{p_{\A^1}} (m-2-2)
(m-2-1) edge node[mathsc,auto]{m\circ(\id_{\A^1}\times \tau)} (m-2-2)
(m-2-1) edge node[mathsc,left]{p_{\A^1}} (m-3-1);
\end{tikzpicture}
\]

\begin{prop}[multiplication and equimensionality]\label{multequi}
In the following statement,  $p$, $k$ and $n$ will denote positive integers
subject to the relation $n=2p-k$.
\begin{itemize}
\item The composition $ \wt{m}=(m \circ (\id_{\A^1}\tau ))\times
  \id_{\square^n}$ induces  a group morphism
\[
\Ze[p](\A^1,n)\st{\wt{m}^*}{\lra} \Ze[p](\A^1\times \square^1,n)
\] 
which extends into a morphism of complexes for any $p$,
\[
\Nge{\A^1}{\bullet}(p)\st{\wt{m}^*}{\lra} \Nge{\A^1\times\square^1}{\bullet}(p).
\]
\item Moreover, there is a natural morphism
\[
h_{\A^1,n}^p : \Ze[p](\A^1\times \square^1,n) \lra  \Ze[p](\A^1,n+1)
\] 
given by regrouping the $\square$'s factors. 
\item The composition $\mu^* = h_{\A^1,n}^p \circ \wt{m}^*$ gives a
  morphism
\[
\mu^* : \Nge{\A^1}{k}(p) \lra \Nge{\A^1}{k-1}(p)
\]
sending equidimensional cycles with  empty fiber at $0$ to equidimensional cycles
with empty fiber at $0$.
\item Let $\theta : \A^1 \lra \A^1$ be the involution sending the natural affine
  coordinate $x$ to $1-x$. Twisting the multiplication $\wt{m}$ by $\theta$ via
\[
\begin{tikzpicture}
\matrix (m) [matrix of math nodes,
row sep=2.5em, column sep=2.5em, 
 text height=1.5ex, text depth=0.25ex]
{\A^1\times \square^1 \times \square^n & \A^1 \times\square^n \\
\A^1\times \square^1 \times \square^n & \A^1 \times\square^n \\};
\path[->,font=\scriptsize]
(m-2-1) edge node[mathsc,auto]{\wt{m}} (m-2-2)
(m-1-1) edge node[mathsc,left]{\theta \times\id _{\square^{n+1}}} (m-2-1)
(m-2-2) edge node[mathsc,right]{\theta \times\id _{\square^{n}}} (m-1-2)
(m-1-1) edge node[mathsc,auto]{} (m-1-2);
\end{tikzpicture}
\]
gives a   morphism
\[
\nu^* : \Nge{\A^1}{k}(p) \lra \Nge{\A^1}{k-1}(p)
\]
sending equidimensional cycles with  empty fiber at $1$ to equidimensional cycles
with empty fiber at $1$.
\end{itemize}
\end{prop}
\begin{proof}
It is enough to work with generators of $\Ze[p](\A^1,n)$. Let $Z$ be an
irreducible subvariety of $\A^1\times \square^n$ such that for any face $F$ of
$\square^n$, the first projection 
\[
p_{\A^1} : Z \cap (\A^1 \times F) \lra \A^1
\]
is dominant and equidimensional of relative dimension $\dim(F)-p$ or empty. Let $F$ be a face
of 
$\square^n$. First we want to show that 
under the projection $\A^1 \times \square^1 \times \square^n \lra \A^1 \times
\square^1$, 
\[
\wt{m}^{-1}(Z)\cap (\A^1\times \square^1 \times F) \lra \A^1 \times
\square^1 \]
 is dominant and equidimensional of relative dimension $\dim(F)-p$ or empty. This follows
 from the  
 fact that $Z \cap (\A^1 \times F)$ is dominant equidimensional over $\A^1$ and $m$ is
 flat and equidimensional of relative 
 dimension $1$ (hence so are $m\circ (\id_{\A^1}\times \tau)$ and $\wt{m}$). 

The map
 $\wt{m}$ is 
 the identity on the $\square^n$ factor, 
 thus for $Z\subset \A^1\times \square^n$ as above and a codimension $1$
 face $F$ of $\square^n$, $\wt{m}^{-1}(Z)$ satisfies
\[
\wt{m}^{-1}(Z)\cap (\A^1 \times \square^1\times F)=\wt{m}^{-1}(Z\cap(\A^1 \times 
F)).
\]
This shows that $\wt{m}^*$ is a morphism of complexes.

Moreover, assuming  that the fiber of $Z$ at $0$ is empty,   the intersection 
\[
\wt{m}^{-1}(Z)\cap (\{0\}\times \square^1 \times \square^n)
\]
is empty because $\wt{m}$ restricted to  
\[
\{0\} \times \square^1 \times \square^n
\]
factors through the inclusion $\{0\}\times \square^n \lra \A^1 \times
\square^n$. Hence the fiber of $\wt{m}^{-1}(Z)$ over $\{0\}\times \square^1$ 
(resp. over $\{0\}$) by $p_{\A^1\times \square^1}$ (resp. $p_{\A^1}\circ
p_{\A^1\times \square^1} $) is empty.

Now, let $Z$ be an irreducible subvariety of $\A^1\times\square^1\times
\square^n$ such that for any face $F$ of $\square^n$
\[
Z\cap (\A^1\times \square^1 \times F) \lra \A^1 \times
\square^1 \]
is dominant and equidimensional of relative dimension $\dim(F)-p$ when the
intersection is not empty. Let $F'$ be a face of 
\[
\square^{n+1}=\square^1\times \square^n.
\]
The face $F'$ is either of the form $\square^1\times F$ or of the form
$\{\ve\}\times F$ with $F$ a face of $\square^n$  and $\ve \in \{0,\infty\}$. We
can assume that $Z\cap (\A^1\times \square^1 \times F)$ is not empty. When
$F'$ is of the first type, we observe  that 
\[
Z\cap (\A^1\times \square^1 \times F) \lra \A^1 \times \square^1
\] is dominant and equidimensional and that $\A^1
\times \square^1 \lra \A^1$  is equidimensional of relative dimension $1$. Hence the
projection 
\[
Z\cap (\A^1\times \square^1 \times F) \lra \A^1
\]
is equidimensional of relative dimension
\[
\dim(F)-p+1=\dim(F')-p.
\]

When $F'$ is of the second type, by symmetry of the role of $0$ and $\infty$, we can
assume that $\ve=0$. Then, the intersection 
\[
Z\cap (\A^1\times \{0\} \times F)
\] 
is nothing but the fiber of $Z\cap (\A^1\times \square^1 \times F)$
over $\A^1\times \{0\}$. Hence, it has pure dimension $\dim(F)-p+1$.

Moreover, denoting the fiber with a subscript, the composition 
\[
Z\cap (\A^1\times \{0\} \times F)=
\left(Z\cap(\A^1\times \square^1 \times F)\right)_{\A^1\times\{0\}} \lra 
\A^1\times \{0\} \lra \A^1
\]
is equidimensional 
of relative dimension  
\[
\dim(F)-p=\dim(F')-p.
\] 
This shows that $h_{\A^1,n}^p$ gives a well defined morphism and that it 
preserves the fiber at a point $x$ in $\A^1$; in particular if $Z$ has an empty
fiber at $0$, so does $h_{\A^1,n}^p (Z)$. 

Finally, the last part of the proposition follows from the fact that
$\theta$ exchanges the roles of $0$ and $1$.
\end{proof}
\begin{rem}\label{emptyinfty} We saw that $\wt{m}$ sends cycles with
  empty fiber at $0$ 
  to cycles with empty fiber at any point in $\{0\}\times \square^1$. Similarly,
  $\wt{m}$ sends cycles with empty fiber at $0$ to cycles that also
  have an empty fiber at any point in $\A^1\times \{\infty\}$.
\end{rem}
From the proof of Proposition 4.2 in \cite{LevBHCG}, we deduce that
$\mu^*$ gives a homotopy between $p_{0}^*\circ i_0^*$ and $ 
\id  $ where $i_0$ is the zero section $\{0\} \ra \A^1$ and $p_0$ the projection
onto the point $\{0\}$.
\begin{prop}\label{multhomo}Let notation be as in Proposition
  \ref{multequi} above. For $\ve=0,1$, let $i_{\ve}$ be  the inclusion of $\ve$
  into $\A^1$
\[
i_0 : \{\ve\} \lra \A^1,
\]  
and let $p_{\ve}$  be the corresponding projections $p_{\ve} : \A^1 \lra
\{\ve\}$.

Then $\mu^*$ provides a homotopy between 
\[
p_{0}^*\circ i_0^* \mx{ and } 
\id : \Nge{\A^1}{\bullet} \lra \Nge{\A^1}{\bullet},
\] 
and similarly $\nu^*$ provides a homotopy between
\[
p_{1}^*\circ i_1^* \mx{ and } 
\id : \Nge{\A^1}{\bullet} \lra \Nge{\A^1}{\bullet}.
\] 
In other words, 
\[
\da[\A^1]\circ \mu^* + \mu^*\circ \da[\A^1] = 
\id - p_{0}^*\circ i_0^* 
\quad \mx{and} \quad 
\da[\A^1]\circ \nu^* + \nu^*\circ \da[\A^1] = 
\id-  p_{1}^*\circ i_1^*
\]
\end{prop}
The proposition follows from commuting the different compositions
  involved, and from the relation between the differential on $\Nge{\A^1\times
    \square^1}{\bullet}$ and that on $\Nge{\A^1}{\bullet}$ via the map
  $h_{\A^1,n}^p$.

\begin{proof} 
Let $i_{0,\square}$ and $i_{\infty, \square}$ denote the zero section and
the infinity section $\A^1 \lra \A^1 \times \square^1$.
The action of $\theta$ only exchanges the role of $0$ and $1$ in $\A^1$, hence it
is enough to prove the statement for $\mu^*$.
As before, in order to obtain the proposition for $\Nge{\A^1}{k}(p)$, it is
enough to work with the generators of $\Ze[p](\A^1,n)$ with $n=2p-k$.

By the previous proposition \ref{multequi}, $\wt{m}^*$ commutes with the
 differential on $\Ze[p](\A^1,\bullet)$ and on $\Ze[p](\A^1\times \square^1,
 \bullet)$.  
 As the morphism $\mu^*$ is defined by  $\mu^*=h_{\A^1,n}^p\circ \wt{m}^*$, the
 proof relies on computing $\da[\A^1]  \circ h_{\A^1, n}^p$. Let $Z$ be
 a generator of $\Ze[p](\A^1\times \square^1,n)$. In particular,
\[
Z \subset \A^1\times \square^1 \times \square^n
\]  
and $h_{\A^ 1,n}^p(Z)$ is also given by $Z$ but viewed in 
\[
\A^1\times \square^{n+1}.
\] 
The differentials denoted by $\da[\A^1]^{n+1}$  on
$\Ze[p](\A^1,n+1)$ and  $\da[\A^1 \times
\square^1]^n$ on $\Ze[p](\A^1\times \square^1, n)$ 
are both given by intersections with the codimension $1$ faces, but the first
$\square^1$ factor in $\square^{n+1}$ gives  two more faces and introduces a
change of sign. Namely, using an extra subscript to indicate in which cycle
groups the intersections take place, we have
\begin{align*}
\da[\A^1]^{n+1}(h_{\A^1,n}^p(Z))=&
\sum_{i=1}^{n+1} (-1)^{i-1}\left(\dN_{i,\A^1}^0(Z)-\dN_{i,\A^1}^{\infty}(Z)\right) \\ 
=&\dN_{1,A^1}^0(Z)-\dN_{1, \A^1}^{\infty}(Z)
-\sum_{i=2}^{n+1}(-1)^{i-2} \left(\dN_{i,\A^1}^0(Z)-\dN_{i,\A^1}^{\infty}(Z)\right) \\
=&i_{0,\square}^*(Z)-i_{\infty, \square }^*(Z) 
-\sum_{i=1}^{n}(-1)^{i-1} 
\left(\dN_{i+1,\A^1}^0(Z)-\dN_{i+1,\A^1}^{\infty}(Z)\right) \\
=&i_{0,\square}^*(Z)-i_{\infty, \square }^*(Z) \\
&\hspace{3em}~-\sum_{i=1}^{n}(-1)^{i-1} 
\left(h_{\A^1, n-1}^p \circ \dN_{i,\A^1\times \square^1}^0(Z)-
h_{\A^1, n-1}^p \circ \dN_{i,\A^1\times \square^1}^{\infty}(Z)\right) \\
=&i_{0,\square}^*(Z)-i_{\infty, \square }^*(Z) - 
h_{\A^1, n-1}^p\circ\da[\A^1 \times \square^1]^n(Z).
\end{align*} 

Thus, we can compute 
 $
\da[\A^1]\circ \mu^* +\mu^*\circ \da[\A^1] 
$ on $\Ze[p](\A^1, n) $ as
\begin{align*}
\da[\A^1]\circ \mu^*+\mu^*\circ \da[\A^1] =&
\da[\A^1]\circ h_{\A^1,n} \circ \wt{m}^*+
h_{\A^1,n-1} \circ \wt{m}^* \circ \da[\A^1] \\
=&i_{0,\square}^*\circ \wt{m}^*-i_{\infty, \square }^*\circ \wt{m}^*
-h_{\A^1,n-1} \circ\da[\A^1]\circ \wt{m}^* \\
&
\hphantom{i_{0,\square}^*\circ \wt{m}^*-i_{\infty, \square }^*\circ \wt{m}^*}
+h_{\A^1,n-1} \circ \da[\A^1] \circ \wt{m}^* \\
=&i_{0,\square}^*\circ \wt{m}^*-i_{\infty, \square }^*\circ \wt{m}^*.
\end{align*}

The morphism $i_{\infty,\square}^*\circ \wt{m}^*$ is induced by
\[
\begin{tikzpicture}
\matrix (m) [matrix of math nodes,
 row sep=0.6em, column sep=2.5em, 
 text height=1.5ex, text depth=0.25ex] 
{\A^1 & \A^1 \times \square^1 & \A^1 \times \A^1 & \A^1 \\
x & (x, \infty) & (x, 0) & 0 \\
};
\path[->]
(m-1-1) edge node[mathsc,auto] {i_{\infty,\square}} (m-1-2)
(m-1-2) edge node[mathsc,auto] {\tau} (m-1-3)
(m-1-3) edge node[mathsc,auto] {m} (m-1-4);
\path[|->]
(m-2-1) edge node[mathsc,auto] {} (m-2-2)
(m-2-2) edge node[mathsc,auto] {} (m-2-3)
(m-2-3) edge node[mathsc,auto] {} (m-2-4);
\end{tikzpicture}
\]
which factors through
\[
\begin{tikzpicture}
\matrix (m) [matrix of math nodes,
 row sep=1.5em, column sep=2.5em, 
 text height=1.5ex, text depth=0.25ex] 
{\A^1 & \A^1 \times \square^1 & \A^1 \times \A^1 & \A^1 \\
\A^1 &  &  & \A^1 \\
};
\path[->]
(m-1-1) edge node[mathsc,auto] {i_{\infty,\square}} (m-1-2)
(m-1-2) edge node[mathsc,auto] {\tau} (m-1-3)
(m-1-3) edge node[mathsc,auto] {m} (m-1-4)
(m-1-1) edge node[mathsc,auto] {p_0} (m-2-1)
(m-2-1) edge node[mathsc,auto] {i_0} (m-2-4)
(m-1-4) edge node[mathsc,auto] {\id_{\A^1}} (m-2-4);
\end{tikzpicture}
\]
Thus, 
\[
i_{\infty,\square}^*\circ \wt{m}^*=(i_0\circ p_0)^*=p_0^*\circ i_0^*.
\]
 Similarly  $i_{0, \square }^*\circ \wt{m}^*$ is induced by 
\[
\begin{tikzpicture}
\matrix (m) [matrix of math nodes,
 row sep=0.6em, column sep=2.5em, 
 text height=1.5ex, text depth=0.25ex] 
{\A^1 & \A^1 \times \square^1 & \A^1 \times \A^1 & \A^1 \\
x & (x, 0) & (x, 1) & x \\
};
\path[->]
(m-1-1) edge node[mathsc,auto] {i_{\infty,\square}} (m-1-2)
(m-1-2) edge node[mathsc,auto] {\tau} (m-1-3)
(m-1-3) edge node[mathsc,auto] {m} (m-1-4);
\path[|->]
(m-2-1) edge node[mathsc,auto] {} (m-2-2)
(m-2-2) edge node[mathsc,auto] {} (m-2-3)
(m-2-3) edge node[mathsc,auto] {} (m-2-4);
\end{tikzpicture}
\]
which factors through $\id_{\A^1} : \A^1 \lra \A^1$, and we have 
\[
i_{0, \square }^*\circ \wt{m}^*=\id.
\]
This concludes the proof of the proposition.
\end{proof}

\subsection{Cycles over $X=\ps$ corresponding to multiple polylogarithms}

Set $\Lcz=L_{0}$ and $\Lco=L_{1}$, where
$L_0$ and $L_1$ are the cycles in $\cNg{X}{1}(1)$ 
defined in Section \ref{subs:poly} induced by the graph of $x \mapsto x$ and $x
\mapsto 1-x$ from $X \lra \p^1$. Note that the superscript $1$ in $\Lcz$ refers
to the fact that this cycle has an empty fiber at $1$.

Consider the two following differential systems

\begin{equation}\label{ED-Lc}
\dN(\Lc_W)=\sum_{U<V} a_{U,V}^W\Lc_U \Lc_V + \sum_{U,V}b_{U,V}^W\Lcu_{U}\Lc_V
\tag{ED-$\Lc$}
\end{equation}
and 
\begin{equation}\label{ED-Lcu}
\dN(\Lcu_W)=\sum_{U<V} \ap_{U,V}^W\Lcu_U \Lcu_V +
\sum_{U,V}\bp_{U,V}^W\Lcu_{U}\Lc_V%+
%\sum_{V}\ap_{0,V}\Lcz\Lc_V 
\tag{ED-$\Lcu$}
\end{equation}
where the coefficients $a_{U,V}^W$, $b_{U,V}^W$, $\ap_{U,V}^W$ and $\bp_{U,V}^W$ are
those defined in Definition \ref{def:aapbbpcoef}.  

These differential equations are exactly the differential system considered in
section \ref{combist}.

\begin{thm}\label{thm:cycleLcLcu} Let $j$ be the inclusion $X \hookrightarrow
  \A^1$. For any Lyndon word $W$ of length $p\geqs 2$, there exist
  two cycles $\Lc_W$ and $\Lcu_W$ in $\cN (p)$ such that:
\begin{itemize}
\item $\Lc_W$, $\Lcu_W$ are elements of $\Nge{X}{1}(p)$.
\item There exist cycles $\ol{\Lc_W}$, $\ol{\Lcu_W}$ in $\Nge{\A^1} 1 (p)$ such
  that
\[
\Lc_W=j^*(\ol{\Lc_W}) \qquad \mx{and} \qquad \Lcu_W=j^*(\ol{\Lcu_W}). 
\]
\item The restriction of $\ol{\Lc_W}$ (resp. $\ol{\Lcu}$) to the fiber $x=0$ (resp. $x=1$)
  is empty. 
\item The cycle $\Lc_W$ (resp. $\Lcu_W$) satisfies the equation \eqref{ED-Lc}
  (resp \eqref{ED-Lcu}) in $\Nc$ and the same holds for its extension
  $\ol{\Lc_W}$ (resp. $\ol{\Lcu_W}$) to $\Nge{\A^1}{\bullet}$.
\end{itemize}  
\end{thm}

The remainder of this section is devoted to proving the above
theorem. Let $\AL$ and $\ALu$ denote the right-hand side of \eqref{ED-Lc} and
\eqref{ED-Lcu} respectively. The proof works by induction, and will be 
developed by the following steps:
 \begin{itemize}
\item Reviewing the cycles $\Lc_{01}$ and $\Lcu_{01}$ presented in subsection
  \ref{subs:poly} in order to show that they give the desired cycles for $W=01$.
\item Proving that $\AL$ and $\ALu$ have
  differential $0$ in $\Nc$. This was essentially proved in  Proposition \ref{dALdALu0}.
\item Extending $\AL$ and $\ALu$ to $\A^1$ and proving in Lemma \ref{d(EDLA)} that the
  differential stays $0$ in $\Nc[\A^1]$.
\item Constructing $\Lc_W$ and $\Lcu_W$ by pull-back by the
  multiplication and pull-back by the twisted multiplication in Lemma \ref{buildLc}. 
\item Proving that the pull-back by the (twisted) multiplication preserves the 
  equidimensionality property and has empty fiber at $x=0$ (resp. $x=1$), as a
  direct consequence of Proposition \ref{multequi}. 
\item Showing that $\Lc_W$ and $\Lcu_W$ satisfy the expected differential equations, which
  follows from the homotopy property of the (twisted) multiplication given in
  Proposition \ref{multhomo}.  
\end{itemize}
\begin{proof}
We start the induction with the only Lyndon word of length $2$: $W=01$. 
\begin{exm}\label{LcLcu01} In Section \ref{subs:poly}, we already considered
  the product 
\[
b=\Lcz \Lco=[x;x,1-x]. \subset X\times \square^2.
\]
In other words, up to projection onto the alternating elements, $b$ is nothing but the
graph of the function $X \lra (\p^1)^2$ sending $x$ to $ (x,1-x)$. Its closure
$\ol{b}$ in $\A^1 \times \square^2$ is induced by the graph of $x \mapsto (x,1-x)$
viewed as a function from $\A^1$ to $(\p^1)^2$:
\[
\ol{b}= [x;x,1-x] \qquad \subset \A^1\times \square^2.
\] 
From this expression, we see that $\da[\A^1](\ol{b})=0$.

Proposition  \ref{equistr}
already ensures that $b$ is equidimensional over $X$, as this is the case for both
$\Lcz$ and $\Lco$. Then, in order to show that $\ol{b}$ is equidimensional over
$\A^1$, it is enough to look at the fibers over $0$ and $1$. In both cases, the
fiber is empty and $\ol{b}$ is equidimensional over $\A^1$.  Now, set
\[
\ol{\Lc_{01}}=\mu^*(\ol{b}) \quad \mx{ and } \quad \ol{\Lc_{01}}=\nu^*(\ol{b}) 
\] 
where $\mu^*$ and $\nu^*$ are defined in Proposition \ref{multequi}. The same
proposition shows that $\ol{\Lc_{01}}$ and $\ol{\Lcu_{01}}$ are
equidimensional over $\A^1$; more precisely they are elements of
$\Nge{\A^1}{1}(2)$.  The same proposition shows that $\ol{\Lc_{01}}$ and
$\ol{\Lcu_{01}}$ have empty fiber at $0$ and $1$ respectively because $\ol b$ has. 

Since the fibers at $0$ and $1$ of $\ol{b}$ are empty and 
$\da[\A^1](\ol{b})=0$, we
conclude from Proposition \ref{multhomo} that 
\[
\da[\A^1](\ol{\Lc_{01}})=\da[\A^1](\ol{\Lcu_{01}})=\ol{b}.
\]

Finally, we define 
\[
\Lc_{01}=j^*(\ol{\Lc_{01}}) \quad \mx{ and } \quad
 \Lcu_{01}=j^*(\ol{\Lcu_{01}}), 
\]
where $j $ is the inclusion $X \lra \A^1$, and conclude using Proposition
\ref{equiopen}.

We can explicitly compute the two pull-backs, and obtain a parametric
representation 
\[
\Lc_{01} =[x; 1-\frac{x}{t_1}, t_1, 1-t_1], \qquad 
\Lcu_{01} =[x; \frac{t_1-x}{t_1-1}, t_1, 1-t_1].
\]
In order to compute the pull-back, observe that if $u=1-x/t_1$ then 
\[
\frac{x}{1-u} = t_1.
\]
Computing the pull-back by $\mu^*$ then comes down to simply 
rescaling the new $\square^1$
factor which arrives in first position. The case of $\nu^*$ is similar, but using
the fact that for $u=\frac{t_1-x}{t_1-1}$ we have
\[
\frac{x-u}{1-u}=t_1.
\]
\end{exm}

Let $W$ be a Lyndon word of length $p$ greater than or equal to $3$. From now on, we
assume that Theorem \ref{thm:cycleLcLcu} holds for any Lyndon word of length strictly 
less than $p$. We set 
\[
\AL=\sum_{U<V} a_{U,V}^W\Lc_U \Lc_V + \sum_{U,V}b_{U,V}^W\Lcu_{U}\Lc_V, 
\]
and 
\[
\ALu=\sum_{U<V} \ap_{U,V}^W\Lcu_U \Lcu_V + \sum_{U,V}\bp_{U,V}^W\Lcu_{U}\Lc_V, 
\]

Lemma \ref{rem:coefAL} shows that $\AL$ and $\ALu$ only involve Lyndon words
$U$ and $V$ such that the sum of the length of $U$ and the length of $V$ is
equal to the length of
$W$. In particular the various coefficients are $0$ as soon as
$U$ or $V$ has length greater than or equal to $W$.

The induction hypothesis gives the existence of $\Lc_U$ and $\Lcu_V$
for any $U$ and $V$ of smaller length, and by definition
$\dN(\Lcz)=\dN(\Lc_1)=0$. So the combinatorial Proposition
\ref{dALdALu0}, with $\Ac U=\Lc_U$ and $\Acu U=\Lcu_U$, shows that
\begin{equation}\label{dNAL0}
\dN(\AL)=\dN(\ALu)=0.
\end{equation}

\begin{lem}[extension to $\A^1$]\label{d(EDLA)}
 Let $\ol{\AL}$ (resp. $\ol{\ALu}$) denote the
  algebraic cycles in $\Zc(\A^1\times \square^{2p-2})$ obtained by
  taking the  Zariski closure in $\A^1\times \square^{2p-2}$ of each term in the formal sum
  defining $\AL$ (resp. $\ALu$). Then 
\begin{itemize}
\item $\ol{\AL}$ and $\ol{\ALu}$ are equidimensional over $\A^1$ with respect to
  any face of $\square^{2p-2}$; i.e. $\ol{\AL}$ and $\ol{\ALu}$ are in
  $\Nge{\A^1}{2}(p)$. 
\item $\AL$ has empty fiber at $0$ and $\ALu$ has empty fiber at $1$.
\item $\da[\A^1](\ol{\AL})=\da[\A^1](\ol{\ALu})=0$
\end{itemize} 
\end{lem}

\begin{proof}
The cases of $\AL$ and $\ALu$ are very similar, so we only discuss the case
of $\AL$.

Let $U$ and $V$ be Lyndon words different from $0$ and $1$, of respective length
$q$ and $q'$ strictly smaller than the length $p$ of $W$.

The induction hypothesis and Proposition \ref{equiopen} show that the
equidimensional cycles over $\A^1$  $\ol{\Lc_{U}}$ (resp. $\ol{\Lcu_{U}}$) and
$\ol{\Lc_{V}}$ 
(resp. $\ol{\Lcu_{V}}$) are the Zariski closure of ${\Lc_{U}}$
(resp. ${\Lcu_{U}}$) and ${\Lc_{V}}$ (resp. ${\Lcu_{V}}$) respectively.

Thus, Proposition \ref{equiopen} ensures that
\[
\ol{\Lc_{U}\cdot \Lc_{V}}=\ol{\Lc_{U}}\cdot \ol{\Lc_{V}}\in \Nge{\A^1}{2}(p)
\qquad \mx{and} \qquad  
\ol{\Lcu_{U}\cdot\Lc_{V}}=\ol{\Lcu_{U}}\cdot \ol{\Lc_{V}}\in \Nge{\A^1}{2}(p).
\]
and that the above products have empty fiber at $0$, since this is the case 
for $\ol{\Lc_U}$ and $\Lc_V$ (Proposition \ref{equistr}).

In order to show that $\AL$ extends to an equidimensional cycle over $\A^1$, it
is now enough to study the products $\Lcz \cdot \Lc_V$ and $\Lcu_U\cdot \Lco$,
since Lemma \ref{rem:coefAL} shows that these are the only types of 
product in $\AL$ with a non-zero coefficient involving  $\Lcz$ and $\Lco$ which are not
equidimensional over $\A^1$.  

We  show below that $\ol{\Lcz \Lc_V}$
is equidimensional over $\A^1$ and has empty fiber at $0$. 

We first observe that $V$ has length $p-1$ and that $\Lc_V$ is in
$\Nge{X}{1}(p-1)$. Let $Z$ be an irreducible component of $\Supp(\Lc_V)$, and
let $\ol{Z}$ denote its Zariski closure 
\[
\A^1\times \square^{2p-3} .
\]
Then $\ol{Z}$ %(resp. $\ol{Z}^{\sm 0}$)
is an equidimensional cycle over $\A^1$ because it is an irreducible component
of $\Supp(\ol{\Lc_V})$. Moreover, $\ol Z$ has an empty fiber at $0$.

Let $\Gamma$ denote the graph of $\id : \p^1\lra \p^1$. Then we have
\[
\Lcz=%\Alt(\Gamma|_{X\times X})=
\Alt\left(\Gamma\cap (X\times \square^1) \right)%, \quad 
\qquad \mx{and} \qquad 
\ol{\Lcz}=\Alt\left(\Gamma\cap (\A^1\times \square^1) \right). 
\]     

We simply write $\Gamma_{X}$ and $\ol{\Gamma_{X}}$
for 
\[
\Gamma_X=\Gamma\cap (X\times \square^1)
\quad \mx{ and } \quad 
\ol{\Gamma_X}=\Gamma\cap (\A^1\times \square^1) 
\]

It is enough to show that $\ol{\Gamma_X \cdot Z} $ is equidimensional over
$\A^1$ (here $\cdot$ denotes the product in $\Nge{X}{\bullet}$). As the projection
\[
\Gamma_X \times Z \lra X \times X
\]
is equidimensional, we have
 \[
\ol{\Gamma_X \cdot Z} \simeq (\ol{\Gamma_X} \times \ol{Z})\cap \im(\Delta_{\A^1})
\]
where $\Delta_{\A^1} : \A^1 \times \square^{2p-2} \lra \A^1 \times  \A^1 \times
\square^{2p-2}$. 

Hence it is enough to show that for any face $F$ of $\square^{2p-2}$ the projection
\[
(\ol{\Gamma_X} \times \ol{Z})\cap \big(\im(\Delta_{\A^1})\cap \times F\big)
\lra \im(\Delta_{\A^1})
\]
is either empty or dominant and equidimensional of relative dimension $\dim(F)-p$.

Restricting the above situation to $X \times X\subset \A^1 \times \A^1$, we see
that it is enough to check that the fibers at $(0,0)$ and $(1,1)$ are empty,
since
\[
(\Gamma_X \times Z) \cap (X \times X \times F)
\] 
is either empty or dominant and equidimensional of the right relative dimension over
$X\times X$.

We write the face $F$  as $F_1 \times F'$ with $F^1$ a face of $\square^1$ and
$F'$ a face of $\square^{2p-3}$. Using the fact that 
\[
\A^1\times \A^1 \times \square^{2p-2} \simeq 
( \A^1 \times \square^{1})\times ( \A^1 \times \square^{2p-3}),
\]
the fiber at $(1,1)$ is given by 
\[
(\ol{\Gamma_X} \times \ol{Z})\cap (\{(1,1)\}\times F)\simeq
(\ol{\Gamma_X}\cap \{1\}\times F_1)\times (\ol Z \cap \{1\}\times F')=\emptyset
\] 
because $\ol{\Gamma_X}\cap \{1\}\times F_1$ is empty in $\A^1 \times \square^1$
($\ol{\Gamma_X}$ is the restriction of the graph of $\id$).

Similarly, the fiber at $(0,0)$ is given by 
\[
(\ol{\Gamma_X} \times \ol{Z})\cap (\{(0,0)\}\times F)\simeq
(\ol{\Gamma_X}\cap \{0\}\times F_1)\times (\ol Z \cap \{0\}\times F')=\emptyset
\] 
because $\ol{Z}$ has empty fiber at $0$ (by induction hypothesis $\ol{\Lc_{V}}$ has
empty fiber at $0$).

Thus,  $\ol{\Gamma_X Z}$ is equidimensional over $\A^1$ with empty fiber at
$0$ (and also at $1$) for any irreducible component $Z$ of $\Supp(\Lc_V)$. Hence  
$\ol{\Lcz \Lc_V}$ is equidimensional with respect to any face and
has empty fiber at $0$ (and at $1$). Exchanging the role of $0$ and $1$, a
similar argument  shows that $\ol{\Lcu_U\Lco}$ 
is equidimensional over $\A^1$ and has empty fiber at $1$ (and at $0$).

The above discussion also shows that $\ol{\AL}$ is equidimensional over $\A^1$
and has an empty fiber at $0$.  

Now, we need to show that 
\[
\da[\A^1](\ol{\AL})=0.
\]

We compute
\[
j^*(\da[\A^1](\ol{\AL}))=\dN(j^*(\ol{\AL}))=\dN(\AL)=0
\]
as explained above. The injectivity of $j^*$ on 
equidimensional cycles (Proposition \ref{equiopen}) ensures that
\[
\da[\A^1](\ol{\AL})=0.
\]

\end{proof}
The equality
\[
\da[\A^1](\ol{\AL})=0 \qquad \left(\mx{resp. }\da[\A^1](\ol{\ALu})=0\right)
\]
shows  that $\ol{\AL}$ (resp. $\ol{\ALu}$) gives a class in $\HH^2(\cNg{\A^1}{\bullet})$.  
As Corollary \ref{H2A1} ensures that this cohomology group is $0$,  $\ol{\AL}$ (resp.
$\ol{\ALu}$) is the
boundary of some cycle $c$ (resp. $c'$) in $\cNg{\A^1}{1}$. Lemma
\ref{buildLc} below gives this $c$ (resp. $c'$) explicitly and, after restriction
to $X$, concludes the proof of
Theorem \ref{thm:cycleLcLcu}.

\end{proof}
\begin{lem}\label{buildLc}
Define $\ol{\Lc_W}$ and $\ol{\Lcu_W}$ in $\Nge{\A^1}{1}(p)$ by 
\[
\ol{\Lc_W}=\mu^*(\ol{\AL}) \quad \mx{and} \quad \ol{\Lcu_W}=\nu^*(\ol{\ALu})
\]
where $\mu^*$ and $\nu^*$ are  the morphisms defined in Proposition
 \ref{multequi}.

Let $j: X \lra \A^1$ be the natural inclusion of $\p^1\sm \{0,1,\infty\}$
into $\A^1$ and define $\Lc_W$ and $\Lcu_W$ by
\[
\Lc_W=j^*(\ol{\Lc_W}) \quad \mx{and} \quad \Lcu_W=j^*(\ol{\Lcu_W}).
\]

Then $\Lc_W$ and $\Lcu_W$ satisfy the conditions of Theorem
\ref{thm:cycleLcLcu}.
\end{lem}
\begin{proof} As in Proposition \ref{multhomo},  let $i_0$ (resp. $i_1$) be  the
  inclusion of $0$ (resp. $1$) 
  in $\A^1$:
\[
i_0 : \{0\} \lra \A^1 \quad  \quad i_{1} : \{1\} \lra \A^1,
\]  
and let $p_0$ and $p_1$ be the corresponding projection $p_{\ve} : \A^1 \lra
\{\ve\}$ for $\ve=0,1$.

Proposition \ref{multequi} ensures that $\ol{\Lc_W}$ (resp. $\ol{\Lcu_W}$ ) is 
equidimensional over $\A^1$ with respect to faces, and  has an empty 
fiber at $x=0$
(resp. $x=1$); in particular $i_0^*(\ol{\AL})=i_1^*(\ol{\ALu})=0$.
 Moreover, Proposition \ref{multhomo} enables us to compute
$\da[\A^1](\ol{\Lc_W})$  as 
\begin{align*}
\da[\A^1](\ol{\Lc_W})=&\da[\A^1]\circ \mu^*(\ol{\AL}) \\
=&\id(\ol{\AL})-p_0^*\circ i_0^*(\ol{\AL})-\mu^*\circ \da[\A^1](\ol{\AL}) \\
=& \ol{\AL} 
\end{align*}
because  $\da[\A^1](\ol{\AL})=0$ and $i_0^*(\ol{\AL})=0$.

Using Proposition \ref{multhomo} again, a similar computation gives
\[
\da[\A^1](\ol{\Lcu_W})= \ol{\ALu}
\]
because  $\da[\A^1](\ol{\ALu})=0$ and $i_1^*(\ol{\ALu})=0$.

Now, as
\[
\Lc_W=j^*(\ol{\Lc_W}) \quad \mx{and} \quad \Lcu_W=j^*(\ol{\Lcu_W}),
\]
$\Lc_W$ and $\Lcu_W$ are equidimensional with respect to any faces over $X$ by
Proposition \ref{equiopen}, and
their closures in $\A^1\times \square^{2p-1}$ are exactly $\ol{\Lc_W}$ and
$\ol{\Lcu_W}$. As $j^*$ is a morphism of cdga's, $\Lc_W$ and $\Lcu_W$ satisfy the
expected 
differential equations, as do  $\ol{\Lc_W}$ and
$\ol{\Lcu_W}$; that is
\[
\dN(\Lc_W)=\AL
\qquad \mx{and}  \qquad 
\dN(\Lcu_W)=\ALu.
\]

This concludes the proof of the Lemma and of Theorem \ref{thm:cycleLcLcu}
\end{proof}

%\section{Concluding remarks}\label{mainsec:conclu}
\subsection{Examples in weight $4$ and $5$}

The first linear combination in the differential equation arises in weight $4$.
The first case where the differential equations for
$\Lc_W$ and $\Lcu_W$ are not the same arises in weight $5$. There are 
actually two such examples in
weight $5$. We give below one of them.

There are three Lyndon words in weight $4$: $0001$, $0011$ and $0111$. The first
linear combination arises from the word $0011$. The image of 
\[
\T{0011}(x)=\TLxxyy{x}
\]
under $\dc$, given at equation \ref{eq:dcyTw0011}, is
\begin{multline*}
\dc(\T{0011}(x))=(\T{0}(x)-\T{0}(1))\w \T{011}(x)
+(\T{001}(x)-\T{001}(1))\w \T{1}(x) \\ +\T{01}\w \T{01}(1).
\end{multline*}
Hence the cycle $\ol{\Lc_{0011}}$ is defined as the pull-back by $\mu$ of 
\[
\ol{\Lcz\Lc_{011}}+\ol{\Lcu_{001}\Lco} +\ol{\Lcu_{01}\Lc_{01}}.
\]

Consider the Lyndon word $00101$ in weight $5$. Its corresponding tree 
$\T{00101}(x)$ is 
\[
\begin{tikzpicture}[baseline=(current bounding box.center)]
\tikzstyle{every child node}=[intvertex]
\node[roots](root) {}
[deftree]
child{node{}
  child[edgesp=2]{node{}
    child[edgesp=1]{node(3){}
      edge from parent [N]}
    child[edgesp=1] {node{}  
       child {node(4){}  
        edge from parent [N]}
       child {node(5){} 
        edge from parent [N]}
    edge from parent [N]}
  edge from parent [N]
  }
  child[edgesp=2] {node{}  
     child[edgesp=1] {node(1){}  
      edge from parent [N]}
     child[edgesp=1] {node(2){} 
      edge from parent [N]}
  edge from parent [N]
  }
edge from parent [N]
}
;
%%%
\fill (root.center) circle (1/1*\lbullet) ;
\node[mathsc, xshift=-1ex] at (root.west) {x};
\node[labf] at (3.south){0};
\node[labf] at (4.south){0};
\node[labf] at (5.south){1};
%;
%%%
%\fill (root.center) circle (1/1*\lbullet) ;
%\node[mathsc, xshift=-1ex] at (root.west) {x};
\node[labf] at (1.south){0};
\node[labf] at (2.south){1};
\end{tikzpicture} %
-
\begin{tikzpicture}[baseline=(current bounding box.center)]
\tikzstyle{every child node}=[intvertex]
\node[roots](root) {}
[deftree]
child{node{}
  child{node{} 
    child{node(2){}
      edge from parent [N]}
    child{node{}
      child{node(3){}
        edge from parent [N]}
      child {node{}  
         child {node(4){}  
          edge from parent [N]}
         child {node(5){} 
          edge from parent [N]}
      edge from parent [N]}
    edge from parent [N]}
  edge from parent [N]
  }
  child{node(1){}
    edge from parent [N]
  }
edge from parent [N]
}
;
%%%
\fill (root.center) circle (1/1*\lbullet) ;
\node[mathsc, xshift=-1ex] at (root.west) {x};
\node[labf] at (1.south){1};
\node[labf] at (2.south){0};
\node[labf] at (3.south){0};
\node[labf] at (4.south){0};
\node[labf] at (5.south){1};
\end{tikzpicture}
\]
and computing $\dc(\T{00101}(x))$ gives 
\[
\dc(\T{00101})=\T{001}(x)\w\T{01}(x)-\T{0001}(x)\w\T{1}(x)-\T{1}(x)\w\T{0001}(1).
\] 
Finally, $\Lc_{00101}$ and $\Lcu_{00101}$ satisfy respectively
\begin{equation}\label{00101dL}
\dN(\Lc_{00101})=\Lc_{001} \Lc_{01}- \Lcu_{0001}\Lco
\end{equation}
and
\begin{equation}\label{00101dLu}
\dN(\Lcu_{00101})=-\Lcu_{001} \Lcu_{01}-\Lcu_{01} \Lc_{001}
+ \Lcu_{001}\Lc_{01} 
- \Lcu_{0001} \Lco .
\end{equation}
\subsection{An integral associated to $\Lc_{011}$}\label{subsec:int}
In this section, we  sketch  how to associate an integral to the cycle
$\Lc_{011}$. We directly follow the algorithm described in
\cite{GanGonLev05}[Section 9] and put in detailed practice in
\cite{GanGonLev06}.  There will be no general review of the direct Hodge realization
from Bloch-Kriz  motives \cite{BKMTM}[Section 8 and 9]. Gangl, Goncharov and
Levin's construction seems to consist in setting particular choices of
representatives in the intermediate Jacobians of $\square^n$ in relation with their
algebraic cycles. 

We do not extend this description here, nor do we generalize the
computations below. Relating the Bloch-Kriz approach to the explicit algorithms
described in \cite{GanGonLev05, GanGonLev06} and the application to our particular
family of cycles $\Lc_W$ will be the topic of a future paper, as it 
requires, in particular, a family ${\Lc_W}^B$ of elements in
$\HH^0(B(\cNg{X}{\bullet}))$ not yet at our disposal.

Let us recall the expression of $\Lc_{011}$ as a parametrized cycle:
\begin{align*}
\Lc_{011}&=[x,1-\frac{x}{t_2},\frac{t_1-t_2}{t_1-1}, t_1,1-t_1,1-t_2] \\
&=-[x; 1-\frac{x}{t_2},1-t_2,\frac{t_1-t_2}{t_1-1},t_1,1-t_1].
\end{align*}
where the second expression is due to the alternating projector and seems to the
author more suitable for the computations below.

We want to bound $\Lc_{011}$ by an algebraic-topological cycle
in a larger bar construction (not described here). This is done by introducing topological
variables $s_i$ in real 
simplices  
\[
\Delta_s^n=\{0\leqs s_1 \leqs \cdots \leqs s_n \leqs 1\}.
\]
 Let $d^s : \Delta^n_s \ra \Delta_s^{n-1}$ denote the simplicial differential
\[
d^s=\sum_{k=0}^{n}(-1)^k i_k^*
\] 
where $i_k : \Delta_s^{n-1} \ra \Delta_s^{n}$ is given by the face
$s_k=s_{k+1}$ in $\Delta_s^n$ with the usual conventions for $k=0,n$.

Let us define
\[
C_{011}^{s,1}=[x; 1-\frac{s_3x}{t_2},1-t_2,\frac{t_1-t_2}{t_1-1},t_1,1-t_1]
\]    
for $s_3$ going from $0$ to $1$. Then, $d^s(C_{011}^{s,1})=\Lc_{011}$,
since $s_3=0$
implies that the first cubical coordinate is $1$.

Now, the algebraic boundary $\dN$ of $C_{011}^{s,1}$ is given by the intersection with the
codimension $1$ faces of $\square^5$; giving 
\[
\dN(C_{011}^{s,1})=[x;1-s_3x, \frac{t_1-s_3x}{t_1-1},t_1,1-t_1].
\]
We can again bound this cycle by introducing a new simplicial variable $0\leqs
s_2\leqs s_3$, and the cycle 
\[
C_{011}^{s,2}=[x;1-s_3x,\frac{t_1-s_2x}{t_1-s_2/s_3},t_1,1-t_1].
\]
The intersections with the faces of the simplex $\{0\leqs s_2 \leqs s_3 \leqs
1\}$ given by 
$s_2=0$ and $s_3=1$ lead to empty cycles (since at least one cubical coordinate
equals $1$) and a negligible cycle respectively. Thus, up to a negligible
cycle, the simplicial boundary of $C_{011}^{s,2}$ satisfies  
\[
d^s(C_{011}^{s,2})=-\dN(C_{011}^{s,1})=-[x;1-s_3x,
\frac{t_1-s_3x}{t_1-1},t_1,1-t_1]. 
\]
Its algebraic boundary is given by 
\[
\dN(C_{011}^{s,2})=-[x;1-s_3x,s_2x,1-s_2x]
+[x;1-s_3x,\frac{s_2}{s_3},1-\frac{s_2}{s_3}].  
\]
Finally, we introduce a last simplicial variable $0\leqs s_1 \leqs s_2$ and a
purely topological cycle
\[
\wt{C_{011}^{s,3}}=-[x; 1-s_3x,s_2x,1-s_1x] + [x;1-s_3x,\frac{s_2}{s_3},1-\frac{s_1}{s_3}]. 
\]
Up to negligible terms, its simplicial differential is given on the one hand by
the face $s_1=s_2$:
\[
%d^s(\wt{C_{011}^{s,3}})=-\dN(C_{011}^{2,s})=
[x;1-s_3x,s_2x,1-s_2x]
-[x;1-s_3t,\frac{s_2}{s_3},1-\frac{s_2}{s_3}]
\]
which is equal to $-\dN(C_{011}^{2,s})$, and on the other hand by the face $s_2=s_3$:
\[
-[x;1-s_3 x,s_3 x ,1-s_1 x].
\]
In order to cancel this extra boundary, we define 
\[
C_{011}^{s,3}=\wt{C_{011}^{s,3}}+ [x ;1-s_2 x , s_3 x ,1-s_1 x]
\]
 whose algebraic boundary is $0$. 

Finally, we have 
\[
(d^s+\dN)(C_{011}^{s,1}+C_{011}^{s,2}+C_{011}^{s,3})=\Lc_{011}
\] 
up to negligible terms.

Now, we fix the situation at the fiber $x_0$, and following Gangl, Goncharov and
Levin, we associate to the algebraic cycle 
$\Lc_{011}|_{x=x_0}$ the integral $I_{011}(x_0)$ of the standard volume form
\[
\frac{1}{(2i\pi)^3}\frac{dz_1}{z_1}\frac{dz_2}{z_2}\frac{dz_3}{z_3}
\]
over the simplex given by $C_{011}^{s,3}$. That is:
\begin{multline*}
I_{011}(x_0)=-\frac{1}{(2i\pi)^3}\int_{0\leqs s_1 \leqs s_2 \leqs s_3 \leqs 1}
\frac{x_0\, ds_3}{1-x_0s_3} \w \frac{ds_2}{s_2} \w
\frac{x_0\, ds_1}{1-x_0s_1} \\
+ 
\frac{1}{(2i\pi)^3}\int_{0\leqs s_3 \leqs 1}\frac{x_0\, ds_3}{1-x_0s_3}
\int_{0 \leqs s_1 \leqs s_2 \leqs 1}\frac{ ds_2}{s_2} \w \frac{ ds_1}{1-s_1} \\
+
\frac{1}{(2i\pi)^3}\int_{0\leqs s_1 \leqs s_2 \leqs s_3 \leqs 1}
\frac{x_0\, ds_2}{1-x_0s_2} \w \frac{ds_3}{s_3} \w
\frac{x_0\, ds_1}{1-x_0s_1}
.
\end{multline*}
 
Taking care of the change of sign due to the numbering, the first term  in the
above sum is (for $x_0 \neq 0$ and up to the factor $(2i\pi)^{-3}$) equal to  
\[
Li_{1,2}^{\C}(x_0)=\int_{0\leqs s_1 \leqs s_2 \leqs s_3 \leqs 1}
\frac{ ds_1}{x_0^{-1}-s_1} \w \frac{ds_2}{s_2} \w
\frac{ds_3}{x_0^{-1}-s_3}
\] 
while the second term  (up to the same multiplicative factor) equals
\[
-Li_1^{\C}(x_0)Li_2^{\C}(1)
\]
and the third term (up to the inverse power of $2i \pi$ ) equals
\[
Li_{2,1}^{\C}(x_0).
\] 

Globally the integral is well-defined for $x_0=0$ and, more interestingly,
also for $x_0=1$, as the divergencies when $x_0$ goes to $1$ cancel
each other out in the above sums. A simple computation and the shuffle relation for
$Li_2^{\C}(x_0)$ shows that  the integral associated to the fiber of
$\Lc_{011}$ at $x_0=1$ is
\[
(2i\pi)^{3}I_{011}(1)=-Li_{2,1}^{\C}(1)=-\zeta(2,1).
\] 
\section{Concluding remarks}\label{mainsec:conclu}
\subsection{Comments about the setting of quasi-finite cycles}
In this paper we chose to work with cycles in the original Bloch complex
$\cNg{X}{\bullet} $ having the extra property that their projection onto the
base $X$ is equidimensional. This allows us to easily compare our construction
to previously constructed explicit cycles related to  polylogarithms (\cite{BKMTM}) or
multiple polylogarithms (with conditions on the parameters in
\cite{GanGonLev05}). Moreover, after computing the structure coefficients of
$\Tcl[1;x]$ (Definition \ref{def:brac1t}) up to some weight, our cycles can be written
 explicitly as parametrized cycles, because the pull-back by the multiplication
 $\mu^*$ introduces a term in $1-\frac{x}{t_{p-1}}$ on the new $\square^1$
 factor while the pull-back by the twisted multiplication $\nu^*$ introduces a
 terms in $\frac{t_{p-1}-x}{t_{p-1}-1}$.

In classical motivic constructions, however, it is preferable to work with
\emph{quasi-finite} cycles (cf. Definition \ref{def:qfcycle}). Quasi-finite
cycles behave better in particular when $X$ is of dimension $d\geqs 2$, with
functoriality or with a sheaf-theoretic approach. They are defined, following
Levine \cite{LEVTMFG}, as follows. 
\begin{defn}%[{\cite[]{LEVTMFG}}]
\label{def:qfcycle}Let $Y$ be an irreducible smooth variety.
\begin{itemize}
\item Let $\Zqf[p](Y,n)$ denote the free abelian group
  generated by irreducible closed subvarieties 
\[Z \subset Y\times \square^n \times(\p^1\sm\{1\})^p
\] 
such that the  restriction of the projection on $Y \times \square^n$, 
\[
 p_1 : Z \lra Y\times \square^n,
\]
is dominant and quasi-finite (i.e. of pure relative dimension  $0$).
\item We say that elements of $\Zqf[p](Y,n)$ are \emph{quasi-finite}.
\item The symmetric group $\Sn[p]$ acts on $\Zqf[p](Y,n)$ by permutation of the
  factors in $(\p^1\sm\{1\})^p$. Let $Sym_{\p^1\sm\{1\}}^p$ denote the projector
  corresponding to the \emph{symmetric} representation.
\item Following the definition of $\cNg{Y}{k}(p)$, let $\Ngqf{Y}{k}(p)$ denote 
\[
\Ngqf{Y}{k}(p)=Sym_{\p^1\sm\{1\}}^p\circ \Alt_{2p-k}\left(\Zqf[p](Y,2p-k )\otimes \Q\right).
\]
\item As in the classical case, the intersection with the codimension $1$ faces of
  $\square^{2p-k}$ induces a differential
\[
\da[Y]=\sum_{i=1}^{2p-k} (-1)^{i-1} (\dN_{i}^0 - \dN_{i}^{\infty}) 
\]
of degree $1$.
\item We define the complex of quasi-finite cycles as
\[
\Ngqf{Y}{\bullet}=\Q \oplus \bigoplus_{p \geqs 1}\Ngqf{Y}{\bullet}(p).
\]
\end{itemize}
\end{defn}
The product structure given by concatenation of factors and pull-back by the
diagonal makes  $\Ngqf{Y}{\bullet}$ into a cdga. The cohomology of
$\Ngqf{Y}{\bullet}$ agrees with higher Chow groups by \cite[Lemma
4.2.1]{LEVTMFG}. %
\begin{rem}\label{rem:condqf}
The condition on the quasi-finite cycles is much stricter than the one for
our equidimensional cycles, as it requires equidimensionality (of dimension $0$)
over 
\[
Y \times \square^n
\]
and not merely over $Y$ as in the case of equidimensional cycles.

In the case of $\Ngqf{Y}{\bullet}(p)$, however, the ambient space is much larger
due to the extra $(\p^1\sm\{1\})^p$ factors.
\end{rem}

Because of the remark above, propositions \ref{equistr} -- point (3) --,
\ref{equiopen}, \ref{multequi} and \ref{multhomo} hold with $\Ngqf{-}{\bullet}$ instead of
$\Nge{-}{\bullet}$.

Let $X$ denote $\ps$ as before. In order to obtain an equivalent of Theorem
\ref{thm:cycleLcLcu} in 
$\Ngqf{X}{\bullet}$, we need to have quasi-finite cycles $\Lczqf$ and $\Lcoqf$,
avatars of $\Lcz$ and $\Lco$ in $\Ngqf{X}{1}$. Once this is granted, 
Lemma \ref{d(EDLA)} holds in $\Ngqf{X}{2}$ by the same arguments using 
quasi-finiteness over 
\[
X \times \square^n
\]
and the empty fiber properties.

Below, we define these cycles $\Lczqf$ and $\Lcoqf$, each of which (up to 
the projectors) is given by a single irreducible variety of 
\[
X\times \square^1\times (\p^1\sm\{1\}).
\]

Let $x$ denote the standard affine coordinates on $X=\ps$; let
  $[U:V]$ be the standard projective coordinate on $\square^1$ and $[A:B]$ that
  on $\p^1\sm\{1\}$. Let $Z_0$ be defined by the following equation
\[
Z_0 : (U-V)(A-B)(U-xV)+x(1-x)UVB=0.
\] 
Similarly, let $Z_1$ be defined by the following equation
\[
Z_1 : (U-V)(A-B)(U-(1-x)V)+x(1-x)UVB=0.
\] 

\begin{prop} \label{prop:LzLoqf}
\begin{enumerate}
\item Let $\Lczqf$ and $\Lcoqf$ the images under  $Sym_{\p^1\sm\{1\}}^1\circ
  \Alt_{1} $ of $Z_0$ and $Z_1$ 
  respectively. Then $\Lczqf$ and $\Lcoqf$ are elements of $\Ngqf{X}{1}$.
\item Their Zariski closures $\ol{\Lczqf}$ and $\ol{\Lcoqf}$ to $\A^1$ lie in
  $\cNg{\A^1\times (\p^1\sm\{1\})}{1}(1)$ with empty fiber over 
\[
\{1\}\times \square^1 \times 
  (\p^1\sm\{1\})
\]
 and $\{0\}\times \square^1 \times(\p^1\sm\{1\})$ respectively.
\item The Zariski closure of their product $\ol{\Lczqf \Lcoqf}$ is an element of
  $\Ngqf{\A^1}{2}(2)$ with empty fiber at $0$ and $1$.  
\end{enumerate}
\end{prop}
\begin{proof} Part (3) is a consequence of part (2). 
The cases of $\Lczqf$ and $\Lcoqf$ are symmetric with respect to
  the role of $0$ and $1$, so we only consider the case of $\Lczqf$.

Note that the projection $Z_0 \lra X \times \square^1$ is dominant because the
defining equation is homogeneous of degree $1$ in $A$ and $B$.

Note that $U-V \neq 0$ in $\square^1$. In order to check the quasi-finiteness,
we write the equation as 
\[
A(U-V)(U-xV)=B\left(
(U-V)(U-xV)+x(1-x)UV
\right).
\]
Observe that when $x$, $1-x$, $U$ and $V$ are all non-zero, 
the equation uniquely determines $[A:B]$. For $U=0$ the equation becomes
\[
Ax=Bx.
\]
Hence $[A:B]$ should be  uniquely determined provided that $x\neq 0$ (which holds in $X
= \ps$). Recall that $[A:B]$ is the projective coordinate on $\p^1\sm\{1\}$. Thus
$A-B$ is invertible and 
\[
Z_0 \cap \left(X \times \{U=0\}\times (\p^1\sm\{1\}) \right)=\emptyset.
\]
 For $V=0$, the equation becomes 
\[
A=B
\]
without any restrictions on $x$. Thus we have
\[
Z_0 \cap \left(X \times \{V=0\}\times (\p^1\sm\{1\}) \right)=\emptyset.
\]
The above discussion shows that $Z_0$ is quasi-finite over $X$. It also shows
that its Zariski closure $\ol{Z_0}$ is not quasi-finite over $\A^1$ because of
the fiber at $x=0$. However, we have 
\[
\ol{Z_0 }\cap \left(\A^1 \times \{V=0\}\times (\p^1\sm\{1\}) \right)=\emptyset.
\]
and
\[
\ol{Z_0 }\cap \left(\A^1 \times \{U=0\}\times (\p^1\sm\{1\}) \right)=
\{0\}\times \{U=0\}\times (\p^1\sm\{1\})
\] 
which is of codimension $1$ in $\A^1 \times \{U=0\}\times (\p^1\sm\{1\})$. 
 This shows that  $\ol{Z_0}$ is an admissible cycle over $\A^1\times (\p^1\sm\{1\})$. 
Moreover the above intersection is concentrated in the fiber of $\ol{Z_0}$ over
$\{0\}\times \square^1$. 

The fiber of $\ol{Z_0}$ over $1$ is given by the equation
\[
A(U-V)^2=B(U-V)^2
\] 
and is empty in $\A^1\times \square^1 \times (\p^1\sm\{1\})$, because $U-V$
and $A-B$ are invertible there.
\end{proof}
From the above discussion we obtain the following result.
\begin{thm}\label{thm:qfcycle} For any Lyndon word $W$ of length $p\geqs 2$, there exist
  two cycles $\Lcqf W$ and $\Lcuqf W$ in $\Ngqf{X}{1}(p)$ satisfying the conditions
  of Theorem \ref{thm:cycleLcLcu} with quasi-finite cycles replacing
  equidimensional ones.
\end{thm}
\begin{rem} The proof of Proposition \ref{prop:LzLoqf} also shows that
 \begin{multline*}
 \HH^1\left((\Ngqf{\ps}{1}\right)(p)\simeq \\
 \HH^1(\Ngqf{ \Q}{1})(p) \oplus 
 \left( \HH^{0}(\Ngqf{\Q}{1})(p-1)\otimes \Q [\Lczqf]\right) 
 \oplus
 \left( \HH^{0}(\Ngqf{\Q}{1})(p-1)\otimes \Q [\Lcoqf]\right).
 \end{multline*}
where $[\Lczqf]$ and $[\Lcoqf]$ are the cohomology classes of $\Lczqf$ and
$\Lcoqf$ respectively.
\end{rem}
%\subsection{Toward multiple zeta value cycles}

%\blue{
Bloch and Kriz in \cite{BKMTM} constructed, for any $n \geqs 2$,  
  algebraic cycles corresponding to the value of the polylogarithms
  $Li_n^{\C}(x_0)$. This leads, by specialization   at $x_0=1$, to an algebraic
  cycle related to the (single) zeta value $\zeta(n)$. Algebraic cycles 
  corresponding to multiple logarithm away from the point $1$ have been
  constructed in \cite{GanGonLev06,GanGonLev05} but these cycles cannot be
  specialized at the point $1$. 
%}
\label{labelpage:lastcomments}
%\blue{
The algebraic cycles $\Lc_W$ given by Theorem \ref{thm:cycleLcLcu}  satisfy the
following properties which are also satisfied by the algebraic cycles given by  Theorem
\ref{thm:qfcycle} (with the appropriate changes in notations). Let $W$ be a
Lyndon word with $W\neq 0, 1$ and let  $x_0$ be a point  in $\A^1(\Q)$:
\begin{itemize}
\item When $W=0\cdots 01$ ($n-1$ zero),  $\Lc_W|_{x=x_0}$ is the
  algebraic cycle corresponding to the polylogarithm $Li_n^{\C}(x_0)$ given by
  Bloch and Kriz in \cite{BKMTM}.
\item The differential equations satisfied by the cycles $\Lc_W$ are closely related to the
  ones satisfied by the multiple polylogarithms. It is clear when
  $W=0...01$. For the others, let's recall that the differential equation comes
  from Ihara's cobracket, that is from Ihara's action of the fundamental group
  of $\ps$ on  itself. Another point of view is that the differential
  equation for $\Lc_W$ and hence for $\Lc_W|_{x=x_0}$  gives the coproduct of
  the induced element ${\left(\Lc_W|_{x=x_0}\right)}^B$ in   the $\HH^0$ of the
  bar construction over $ \cNg{\Spec(\Q)}{1}$. 
  This coproduct is Goncharov's motivic coproduct
  (cf. \cite{LEVTMFG})   on Deligne-Goncharov  motivic fundamental group
  (cf. \cite{DG}); however this coproduct is given by the differential of
  iterated integral. This other approach also shows that the differential of the
  cycles $\Lc_W$ encode the differential equations of the multiple polylogarithms.
\item A Lyndon word $W$ may contain more than one $1$ ; that is
  \emph{multiple} polylogarithms are present. 
\item For any Lyndon word $W$ and any $x_0$ as above, the cycle $\Lc_W|_{x=x_0}$
  is admissible, that is in $\cNg{\Spec(\Q)}{1}$. Hence, specializing at
  $x_0=1$, multiple zeta values are present.
% \item On low weight examples (weight $3$, $4$ and $5$) the integral associated
%   to the cycle $\Lc_W|_{x=x_0}$ is a linear combination of multiple
%   polylogarithms coverging to a linear 
\end{itemize}
%These 
%}

%\blue{
The  above reasons, together with the integral computed in the previous section,
lead the author to believe that cycles $\Lc_W|_{x=1}$ correspond to
multiple zeta values.
% : at least, the family of   associated integrals should be a set
%of generators for the $\Q$ vector space generated by the multiple zeta values (modulo
%shuffle products).
%}

%\blue{
In this direction, the author proved that the  bar elements associated to the
family $\Lcqf W$% in the bar construction of $\Ngqf{\ps}{\bullet}$ 
give a basis of% the
%motivic Lie coalgebra over $\ps$ (up to constant cycle) 
%relatively to the motivic Lie coalgebra over
%$\Spec(\Q)$ 
 Deligne-Goncharov
motivic fundamental group (see \cite{SouBarbase}). 
%; thus it gives a basis of Deligne-Goncharov
%motivic fundamental group. 
It remains to explicitly compute the periods of these
motives, that is the associated integral. This question will be addressed in a
future paper. 
%}
\bibliographystyle{amsalpha}
\bibliography{cyclecomplex}

\providecommand{\bysame}{\leavevmode\hbox to3em{\hrulefill}\thinspace}
\providecommand{\MR}{\relax\ifhmode\unskip\space\fi MR }
\providecommand{\MRhref}[2]{%
  \href{http://www.ams.org/mathscinet-getitem?mr=#1}{#2}
}
\providecommand{\href}[2]{#2}
\begin{thebibliography}{DGMS75}

\bibitem[BK94]{BKMTM}
Spencer Bloch and Igor Kriz, \emph{Mixed tate motives}, Anna. of Math.
  \textbf{140} (1994), no.~3, 557--605.

\bibitem[Blo86]{BlochACHKT}
Spencer Bloch, \emph{Algebraic cycles and higher {$K$}-theory}, Adv. in Math.
  \textbf{61} (1986), no.~3, 267--304.

\bibitem[Blo91]{BlochLie}
Spencer Bloch, \emph{Algebraic cycles and the {L}ie algebra of mixed {T}ate
  motives}, J. Amer. Math. Soc. \textbf{4} (1991), no.~4, 771--791.

\bibitem[Blo94]{BlochMLHCG}
S.~Bloch, \emph{The moving lemma for higher {C}how groups}, J. Algebraic Geom.
  \textbf{3} (1994), no.~3, 537--568.

\bibitem[Blo97]{BlochLMM}
Spencer Bloch, \emph{Lectures on mixed motives}, Algebraic geometry---{S}anta
  {C}ruz 1995, Proc. Sympos. Pure Math., vol.~62, Amer. Math. Soc., Providence,
  RI, 1997, pp.~329--359.

\bibitem[Bor74]{BorCAG}
Armand Borel, \emph{Stable real cohomology of arithmetic groups}, Ann. Sci.
  \'Ecole Norm. Sup. (4) \textbf{7} (1974), 235--272 (1975).

\bibitem[CD09]{CiDegTCM}
D.-C. Cisinski and F~D{\'e}glise, \emph{Triangulated categories of motives},
  http://arxiv.org/abs/0912.2110, 2009.

\bibitem[Del89]{GFPLDeli}
Pierre Deligne, \emph{Le groupe fondamental de la droite projective moins trois
  points}, Galois groups over {$\Q$}, MSRI Publ, vol.~16, Springer Verlag,
  1989, pp.~70--313.

\bibitem[DG05]{DG}
Pierre Deligne and Alexander~B. Goncharov, \emph{Groupes fondamentaux
  motiviques de {T}ate mixte}, Ann. Sci. \'Ecole Norm. Sup. (4) \textbf{38}
  (2005), no.~1, 1--56.

\bibitem[DGMS75]{DGMS}
Pierre Deligne, Phillip Griffiths, John Morgan, and Dennis Sullivan, \emph{Real
  homotopy theory of {K}\"ahler manifolds}, Invent. Math. \textbf{29} (1975),
  no.~3, 245--274.

\bibitem[GGL07]{GanGonLev06}
H.~Gangl, A.~B. Goncharov, and A.~Levin, \emph{Multiple logarithms, algebraic
  cycles and trees}, Frontiers in number theory, physics, and geometry. {II},
  Springer, Berlin, 2007, pp.~759--774.

\bibitem[GGL09]{GanGonLev05}
\bysame, \emph{Multiple polylogarithms, polygons, trees and algebraic cycles},
  Algebraic geometry---{S}eattle 2005. {P}art 2, Proc. Sympos. Pure Math.,
  vol.~80, Amer. Math. Soc., Providence, RI, 2009, pp.~547--593.

\bibitem[Gon95]{PAGGon}
Alexander~B. Goncharov, \emph{Polylogarithms in arithmetic and geometry},
  Proceedings of the {I}nternational {C}ongress of {M}athematicians, {V}ol.\ 1,
  2 ({Z}\"urich, 1994) (Basel), Birkh\"auser, 1995, pp.~374--387.

\bibitem[Gon05a]{GSFGGon}
A.~B. Goncharov, \emph{Galois symmetries of fundamental groupoids and
  noncommutative geometry}, Duke Math. J. \textbf{128} (2005), no.~2, 209--284.

\bibitem[Gon05b]{PRACGon}
\bysame, \emph{Polylogarithms, regulators, and {A}rakelov motivic complexes},
  J. Amer. Math. Soc. \textbf{18} (2005), no.~1, 1--60 (electronic).

\bibitem[GOV97]{GOVFLTLTG}
V.~V. Gorbatsevich, A.~L. Onishchik, and E.~B. Vinberg, \emph{Foundations of
  {L}ie theory and {L}ie transformation groups}, Springer-Verlag, Berlin, 1997,
  Translated from the Russian by A. Kozlowski, Reprint of the 1993 translation
  [{{\i}t Lie groups and Lie algebras. I}, Encyclopaedia Math. Sci., 20,
  Springer, Berlin, 1993;].

\bibitem[Iha90]{IharaAPBG}
Yasutaka Ihara, \emph{Automorphisms of pure sphere braid groups and {G}alois
  representations}, The {G}rothendieck {F}estschrift, {V}ol.\ {II}, Progr.
  Math., vol.~87, Birkh\"auser Boston, Boston, MA, 1990, pp.~353--373.

\bibitem[Iha92]{IharaSDABG}
\bysame, \emph{On the stable derivation algebra associated with some braid
  groups}, Israel J. Math \textbf{80} (1992), no.~1-2, 135--153.

\bibitem[Lev93]{LevTM}
M.~Levine, \emph{Tate motives and the vanishing conjecture for algebraic
  k-theory}, Algebraic {$K$}-{T}heory and {A}lgebraic {T}opology, Lake Louise
  1991 (Paul~G. Goerss and John~F. Jardine, eds.), NATO Adv. Sci. Inst. Ser C
  Math. Phys. Sci., no. 407, Kluwer Acad. Pub., Fevrier 1993, pp.~167--188.

\bibitem[Lev94]{LevBHCG}
Marc Levine, \emph{Bloch's higher {C}how groups revisited}, Ast\'erisque
  (1994), no.~226, 10, 235--320, $K$-theory (Strasbourg, 1992).

\bibitem[Lev05]{KTMMLevine}
Marc Levine, \emph{Mixed motives}, Handbook of {$K$}-{T}heory (E.M Friedlander
  and D.R. Grayson, eds.), vol.~1, Springer-Verlag, 2005, pp.~429--535.

\bibitem[Lev09]{LevSM}
Marc Levine, \emph{Smooth motives}, Motives and algebraic cycles, Fields Inst.
  Commun., vol.~56, Amer. Math. Soc., Providence, RI, 2009, pp.~175--231.

\bibitem[Lev11]{LEVTMFG}
\bysame, \emph{{Tate motives and the fundamental group.}}, {Srinivas, V. (ed.),
  Cycles, motives and Shimura varieties. Proceedings of the international
  colloquium, Mumbai, India, January 3--12, 2008. New Delhi: Narosa Publishing
  House/Published for the Tata Institute of Fundamental Research. 265-392
  (2011).}, 2011.

\bibitem[Reu93]{ReuFLA93}
Christophe Reutenauer, \emph{Free {L}ie algebras}, London Mathematical Society
  Monographs. New Series, vol.~7, The Clarendon Press Oxford University Press,
  New York, 1993, Oxford Science Publications.

\bibitem[Reu03]{ReuFLAH}
\bysame, \emph{Free {L}ie algebras}, Handbook of algebra, {V}ol. 3,
  North-Holland, Amsterdam, 2003, pp.~887--903.

\bibitem[Sou14]{SouBarbase}
Ismael Soud{\`e}res, \emph{A relative basis for mixed tate motives over the
  projective line minus three points}, http://arxiv.org/abs/1312.1849, 2014.

\bibitem[Spi ]{SpitzweckSCVTM}
Markus Spitzweck, \emph{Some constructions for voevodsky's triangulated
  categories of motives}, preprint, {~}.

\bibitem[Tot92]{Totaro}
Burt Totaro, \emph{Milnor {$K$}-theory is the simplest part of algebraic
  {$K$}-theory}, $K$-Theory \textbf{6} (1992), no.~2, 177--189.

\bibitem[Voe00]{Vo00}
V.~Voevodsky, \emph{Triangulated category of motives over a field}, Cycles,
  transfers, and motivic homology theories, Annals of Math. Studies, vol. 143,
  Princeton University Press., 2000, pp.~188--238.

\bibitem[Voe02]{MCHCGVoe}
Vladimir Voevodsky, \emph{Motivic cohomology groups are isomorphic to higher
  {C}how groups in any characteristic}, Int. Math. Res. Not. (2002), no.~7,
  351--355.

\bibitem[Zag91]{PDZKthZag}
Don Zagier, \emph{Polylogarithms, {D}edekind zeta functions and the algebraic
  {$K$}-theory of fields}, Arithmetic algebraic geometry ({T}exel, 1989),
  Progr. Math., vol.~89, Birkh\"auser Boston, Boston, MA, 1991, pp.~391--430.

\end{thebibliography}
\end{document}